\documentclass[11pt,a4paper]{article}
\usepackage{amsfonts,amscd,amsthm,amsgen,amsmath,amssymb}
\usepackage[all]{xy}
\usepackage[vcentermath]{youngtab}
\newtheorem{prop}{Proposition}[section]

\newtheorem{lemp}{Lemma}[section]

\theoremstyle{remark}
\newtheorem{rem}{Remark}[section]
\newtheorem{ex}{Example}[section]
\theoremstyle{definition}
\newtheorem{defn}{Definition}[section]

\title{Topological T-duality for general circle bundles}
\author{David Baraglia\footnote{Email: david.baraglia@anu.edu.au}
\footnote{This work is supported by the Australian Research Council Discovery Project DP110103745.} \\
Mathematical sciences institute \\ 
The Australian National University \\
Canberra ACT 0200, Australia}
\date{\today}
\begin{document}
\maketitle
\begin{abstract}
We extend topological T-duality to the case of general circle bundles. In this setting we prove existence and uniqueness of T-duals. We then show that T-dual spaces have isomorphic twisted cohomology, twisted $K$-theory and Courant algebroids. A novel feature is that we must consider two kinds of twists in de Rham cohomology and $K$-theory, namely by degree $3$ integral classes and a less familiar kind of twist using real line bundles. We give some examples of T-dual non-oriented circle bundles and calculate their twisted $K$-theory.
\end{abstract}
\newpage
\tableofcontents
\newpage


\section{Introduction}

T-duality is a duality between spaces equipped with a degree $3$ integral cohomology class which arose from string theory (see for example \cite{gpr} and references therein). In this setting the spaces are equipped with metrics or other geometric structures and the T-duality is a duality between the two geometries. 

Topological T-duality, as formulated in \cite{bem} is the result of disregarding the more geometric aspects of T-dualty resulting in a purely topological notion. Here one considers pairs $(E,h)$ where $E$ is a principal torus bundle and $h \in H^3(E,\mathbb{Z})$ is a degree $3$-cohomology class, the {\em $H$-flux} in the physical interpretation. In this paper we extend topological T-duality from principal circle bundles to general circles bundles which are not necessarily orientable. The idea of developing such an extension actually arose from considering Courant algebroids and their relation to T-duality. Many of the results will extend to higher rank affine torus bundles but we have chosen to focus on the circle bundle case since already this demonstrates most of the relevant modifications one needs to make while avoiding the complications that occur for higher rank T-duality \cite{brs}, \cite{bhm1} which have been explained in terms of non-commutative geometry \cite{mr1}, \cite{mr2} and even non-associative geometry \cite{bhm2}.\\

To explain our formulation of T-duality let us first recall the definition of topological T-duality in \cite{bem} for principal ${\rm U}(1)$-bundles. Let $(E,h)$, $(\hat{E},\hat{h})$ be pairs consisting of principal ${\rm U}(1)$-bundles over a space $M$ and $h \in H^3(E,\mathbb{Z})$, $\hat{h} \in H^3(\hat{E},\mathbb{Z})$. We then have a commutative diagram
\begin{equation*}\xymatrix{
& F \ar[dl]_p \ar[dr]^{\hat{p}} \ar[dd]^q & \\
E \ar[dr]_\pi & & \hat{E} \ar[dl]^{\hat{\pi}} \\
& M &
}
\end{equation*}
where $F$ is the fibre product $F = E \times_M \hat{E}$. The pairs $(E,h)$, $(\hat{E},\hat{h})$ are said to be {\em T-dual} if the following conditions hold:
\begin{itemize}
\item{the first Chern class $F = c_1(E)$ of $E$ equals $\hat{\pi}_* \hat{h}$}
\item{the first Chern class $\hat{F} = c_1(\hat{E})$ of $\hat{E}$ equals $\pi_* h$}
\item{$p^* h = \hat{p}^* \hat{h}$}
\end{itemize}
where $\pi_*$ is the push-forward or Gysin homomorphism in cohomology $\pi_* : H^*(E,\mathbb{Z}) \to H^{*-1}(M,\mathbb{Z})$ and similarly for $\hat{\pi}_*$. Roughly speaking T-duality can be understood as an interchange between the first Chern class and part of the $H$-flux. For every pair $(E,h)$ there exists a T-dual pair $(\hat{E},\hat{h})$ which is unique up to bundle isomorphism.\\

Recall that for each class $h \in H^3(E,\mathbb{Z})$ one may define a twisted $K$-theory group $K^*(X,h)$. A key property of T-duality is that the twisted $K$-theories are isomorphic up to a shift in degree:
\begin{equation*}
K^*(E,h) \simeq K^{*-1}(E,\hat{h}).
\end{equation*}
If the spaces are smooth manifolds and $H,\hat{H}$ are closed forms representing $h,\hat{h}$ in real cohomology then we can also define twisted cohomology and there is an isomorphism
\begin{equation*}
H^*(E,H) \simeq H^{*-1}(\hat{E},\hat{H})
\end{equation*}
where the twisted cohomology $H^*(E,H)$ is the cohomology of the space $\Omega^*(E)$ of differential forms on $E$ thought of as a $\mathbb{Z}_2$-graded complex with differential $d_H = d + H$.\\

We can formulate a similar notion of T-duality for general circle bundles. We prove existence and uniqueness of T-duals in this more general setting. We also prove an isomorphism between the twisted $K$-theories and twisted cohomologies of T-dual spaces but for this result to hold we need to consider an additional kind of twist in K-theory and cohomology which accounts for non-orientability of the circle bundles.\\

For T-duality we need to recall the classification of circle bundles and this is done in Section \ref{affcl}. Actually we classify affine torus bundles, those torus bundles which have a reduction of structure to the group ${\rm Aff}(T^n) = {\rm GL}(n,\mathbb{Z}) \ltimes T^n$ of affine transformations of $T^n$. Even if we are just interested in the case of circle bundles this more general classification is useful because the correspondence space $F = E \times_M \hat{E}$ is a higher rank torus bundle and our proof of existence of T-duals makes careful use of the Leray-Serre spectral sequence for $F \to M$. 

In summary affine torus bundles over $M$ are classified by a pair $(\rho,c)$ where $\rho$ is a homomorphism $\rho : \pi_1(M) \to {\rm GL}(n,\mathbb{Z})$ called the {\em monodromy} and $c \in H^2(M,\Lambda_\rho)$ is a degree $2$ cohomology class with coefficients in a local system $\Lambda_\rho$ determined by $\rho$, called the {\em twisted first Chern class}. Two pairs $(\rho,c)$, $(\rho',c')$ give the same $T^n$-bundle if and only if they are related by the natural action of ${\rm GL}(n,\mathbb{Z})$. 

For circle bundles the monodromy is just an element $\xi \in H^1(M,\mathbb{Z}_2)$ which we also call the {\em first Stiefel-Whitney class of $E$} and the twisted first Chern class an element $c \in H^2(M,\mathbb{Z}_\xi)$ where $\mathbb{Z}_\xi$ is the system of local coefficients determined by $\xi$.\\

Now we can now give our definition of T-duality for circle bundles. Let $E \to M$ be a circle bundle over a space $M$ and $h \in H^3(E,\mathbb{Z})$. There is a push-forward operation $\pi_* : H^*(E,\mathbb{Z}) \to H^{*-1}(M,\mathbb{Z}_\xi)$ where $\mathbb{Z}_\xi$ is the local system determined by the class $\xi = w_1(E)$. We say a pair $(\hat{E},\hat{h})$ consisting of a circle bundle $\hat{E}\to M$ and flux $\hat{h} \in H^3(\hat{E},\mathbb{Z})$ is T-dual to $(E,h)$ if the following conditions are satisfied:
\begin{itemize}
\item{the Stiefel-Whitney classes agree: $\xi = w_1(E) = w_1(\hat{E})$}
\item{$F = c_1(E) = \hat{\pi}_* \hat{h}$ and $\hat{F} = c_1(\hat{E}) = \pi_* h$}
\item{$p^*h = \hat{p}^* \hat{h}$.}
\end{itemize}

In Proposition \ref{tdualexist} we prove that for any pair $(E,h)$ there exists a $T$-dual pair $(\hat{E},\hat{h})$ having the desired properties and that $(\hat{E},\hat{h})$ is unique up to bundle isomorphism.\\

For smooth manifolds we can give a description of $T$-duality using differential forms. If $(E,h)$ and $(\hat{E},\hat{h})$ are $T$-dual pairs over $M$ and we represent $h$ in real cohomology by a closed $3$-form $H \in \Omega^3(E)$ then we can give an explicit construction for the dual $3$-form $\hat{H}$ representing $\hat{h}$. For principal circle bundles the usual construction is done by choosing connections on $E,\hat{E}$. We extend this construction to the more general case using a notion of {\em twisted connection} introduced in Section \ref{twcon}. Alternatively the fibre orientation class $\xi \in H^1(M,\mathbb{Z}_2)$ determines a double cover $m : M_2 \to M$ such that $m^*(E) \to M$ admits the structure of a principal ${\rm O}(2)$-bundle and a twisted connection amounts to a connection on $m^*(E)$. From this point of view T-duality of circle bundles can be interpreted as a T-duality of principal ${\rm O}(2)$-bundles.\\

Next we compare the twisted cohomology of the pairs $(E,H)$, $(\hat{E},\hat{H})$. We would like to define a map between twisted cohomologies of $(E,H)$ and $(\hat{E},\hat{H})$ by means of a kind of Fourier-Mukai transform. Thus we first perform a pull-back $p^* : H^*(E,H) \to H^*(F,p^*H)$, then since $p^* H = \hat{p}^* \hat{H}$ there is a $2$-form $\mathcal{B} \in \Omega^2(F)$ such that 
\begin{equation*}
p^*H - \hat{p}^* \hat{H} = d \mathcal{B}
\end{equation*}
inducing an isomorphism $e^{\mathcal{B}} : H^*(F , p^*H) \to H^*(F,\hat{p}^* \hat{H})$ given by multiplication by $e^{\mathcal{B}}$. The final step is a push-forward or integration over the fibres $\hat{p}_* : H^*(F,\hat{p}^* \hat{H}) \to H^{*-1}(\hat{E}, (\xi , \hat{H}))$ where we have introduced a second kind of twisted cohomology $H^{*}(\hat{E}, (\xi , \hat{H}))$. The lack of consistent fibre orientation means that fibre integration picks up this additional twist. Corresponding to the class $\xi \in H^1(M,\mathbb{Z}_2)$ is a flat line bundle $(\mathbb{R}_\xi , \nabla)$. We then define $H^*(\hat{E},(\xi , \hat{H}))$ to be the cohomology of the complex $\Omega^*(\hat{E},\mathbb{R}_\xi)$ with differential $d_{\nabla,\hat{H}} = d_{\nabla} + \hat{H}$. Putting it all together we have defined a T-duality map
\begin{equation*}
T = \hat{p}_* \circ e^{\mathcal{B}} \circ p^* : H^*(E,H) \to H^{*-1}(\hat{E},(\xi,\hat{H}) ).
\end{equation*}
In a similar manner there is also T-duality map
\begin{equation*}
T_\xi : H^*(E,(\xi,H)) \to H^{*-1}(\hat{E},\hat{H}).
\end{equation*} 
We prove in Proposition \ref{tci} that the maps $T,T_\xi$ are isomorphisms.\\

Turning now to twisted $K$-theory in addition to twists by classes in $H^3(X,\mathbb{Z})$ we would also like to be able to twist by classes in $H^1(X,\mathbb{Z}_2)$ in a way that accounts for fibre orientation. In fact such twists of $K$-theory have been around since the work of Donovan and Karoubi \cite{dk}, though only for torsion classes. More recently Freed, Hopkins and Teleman \cite{fht} consider precisely the type of twisted $K$-theory we require. For any pair $(\xi , h) \in H^1(X,\mathbb{Z}_2) \times H^3(X,\mathbb{Z})$ there is a twisted $K$-theory group $K^*(E,(\xi,h))$. More correctly twists of $K$-theory can be defined using a notion of {\em graded bundle gerbe} which is a slight extension of the more familiar ungraded bundle gerbes. We review this in Section \ref{grbg}.\\

For a T-dual pair $(E,h)$, $(\hat{E},\hat{h})$ we construct maps 
\begin{eqnarray*}
&T& : K^*(E,h) \to K^{*-1}(\hat{E}, (\xi, \hat{h} )) \\
&T_\xi & : K^*(E,(\xi,h)) \to K^{*-1}(\hat{E},\hat{h})
\end{eqnarray*}
which are similarly defined as a kind of Fourier-Mukai transform. In Proposition \ref{tki} we show that $T,T_\xi$ are isomorphisms, at least if the base $M$ admits a finite good cover.\\

There is a twisted Chern character in twisted $K$-theory. If $\mathcal{G}$ is a graded gerbe with connection and curving with curvature $H$ then the twisted Chern character has the form
\begin{equation*}
Ch_\mathcal{G} : K^*(X,\alpha) \to H^*(X,(\xi,H))
\end{equation*}
where $\xi$ is the $H^1(X,\mathbb{Z}_2)$-component of the class of $\mathcal{G}$ in $H^1(X,\mathbb{Z}_2) \times H^3(X,\mathbb{Z})$. The twisted Chern chacter is often described only for ungraded bundle gerbes however the extension to graded gerbes is quite trivial and we explain this in Section \ref{tcc}. With suitable conventions in place the twisted Chern character is preserved under $T$-duality in the sense that we can construct a commutative diagram
\begin{equation*}\xymatrix{
K^*(E,\mathcal{G}) \ar[r]^-{T} \ar[d]^{Ch} & K^{*-1}(\hat{E}, \hat{\pi}^*\gamma \otimes \hat{\mathcal{G}}) \ar[d]^{Ch} \\
H^*(E,H) \ar[r]^-{T} & H^{*-1}(E,(\xi,\hat{H}))
}
\end{equation*}
here $\gamma$ is a graded gerbe that can be associated to the class $\xi$. To show this we extend the Riemann-Roch formula in twisted $K$-theory from the case of twists classified by $H^3(X,\mathbb{Z})$ to twists classified by $H^1(X,\mathbb{Z}_2) \times H^3(X,\mathbb{Z})$.\\

In Section \ref{catd} we explain the connection between T-dual circle bundles and Courant algebroids. Let $\pi : E \to M$ be a circle bundle over $M$ and $h \in H^3(E,\mathbb{Z})$ a degree $3$ cohomology class. Exact Courant algebroids on $E$ are classified by $H^3(E,\mathbb{R})$ so there is a unique exact Courant algebroid on $E$ classified by $h$. If $E$ is given an affine structure then one can take invariant sections which determines a corresponding Courant algebroid $\mathcal{E}(E,h)$ over the base $M$. One can do the same thing for the T-dual pair $(\hat{E},\hat{h})$ to get another Courant algebroid $\mathcal{E}(\hat{E},\hat{h})$ on $M$. We show in Proposition \ref{cai} that the Courant algebroids $\mathcal{E}(E,h),\mathcal{E}(\hat{E},\hat{h})$ are isomorphic. The existence of such an isomorphism is a weaker statement than T-duality between $(E,h)$ and $(\hat{E},\hat{h})$ due to a loss of information about torsion classes in cohomology. Nevertheless it is an important part of the T-duality story since geometric structures on $\mathcal{E}(E,h)$ can be transported to corresponding geometric structures on $\mathcal{E}(\hat{E},\hat{h})$. This ties in with the more geometric aspects of T-duality.\\

In section \ref{examps} we study some low dimensional examples of T-duality of non-orientable circle bundles. In some cases we are able to calculate all the relevant twisted $K$-theory groups and the results are consistent with T-duality. In other case we are only able to calculate the twisted K-theory group modulo an extension problem. In this case we can actually use T-duality to determine the group by calculating the corresponding twisted K-theory group of the T-dual.


\section{Local coefficients}

We provide a review of cohomology with local coefficients for the benefit of the reader and to establish notation. References for local coefficients are \cite{hat1},\cite{spa}.\\

Let $X$ be a path-connected topological space admitting a universal cover $\tilde{X}$. Let $A$ be an abelian group and ${\rm Aut}(A)$ the group automorphisms of $A$. Consider the following structures:

\begin{itemize}
\item{Representations $\rho : \pi_1(X) \to {\rm Aut}(A)$.}
\item{Covering spaces $Y \to X$ with fibres isomorphic to $A$ and a fibrewise addition $Y \times_X Y \to Y$ which restricted to the fibres corresponds to addition in $A$.}
\item{Sheaves of groups $\Gamma$ on $X$ such that each point $x \in X$ has a neighborhood $U$ such that $\Gamma$ restricted to $U$ is the constant sheaf $A$.}
\end{itemize}

Up to isomorphism all three of these structures are in bijection. In the case of representations $\pi_1(X) \to {\rm Aut}(A)$ two representations are isomorphic if they are conjugate under ${\rm Aut}(A)$. From a representation $\rho$ we form a covering space $A_\rho = \tilde{X} \times_\rho A$ and a sheaf $\Gamma(A_\rho)$, the sheaf of sections of $A_\rho$.

Given this correspondence we will refer to any one these three structures as a {\em system of local coefficients} and call the group $A$ the {\em coefficient group} of the local system. Sometimes we will refer to the representation $\rho : \pi_1(X) \to {\rm Aut}(A)$ as the {\em monodromy} of the local system.\\

We can consider homology and cohomology with local coefficients in a number of ways. First we look at singular homology and cohomology. Starting from the representation $\rho$ we may form the chain complex with coefficients in $A_\rho$, $C_n(X,A_\rho) = C_n(\tilde{X} , \mathbb{Z}) \otimes_{\rho} A$. Here the tensor product $\otimes_{\rho}$ means a tensor product of right and left $\mathbb{Z}[\pi_1(X)]$-modules. The boundary operator $\partial$ for $C_*(\tilde{X},A)$ clearly descends to a boundary operator on $C_*(X,A_\rho)$ and the homology $H_*(X,A_\rho)$ of this complex is by definition the homology of $X$ with local coefficients in $X_A$. Similarly we can form the cochain complex $C^n(X,A_\rho) = {\rm Hom}_{\mathbb{Z}[\pi_1(X)]}( C_n(\tilde{X},\mathbb{Z}) , A)$ and the coboundary operator $\delta$ on $C^*(\tilde{X},A)$ descends to a coboundary operator on $C^*(X,A_\rho)$. We let $H^*(X,A_\rho)$ denote the corresponding cohomology.\\

If we consider a system of local coefficients as a sheaf $\Gamma(A_\rho)$ then we can take \v{C}ech cohomology groups $\breve{H}^*(X,\Gamma(A_\rho))$. We have:
\begin{prop}\label{singcech}
If $X$ is triangulable then there is a canonical isomorphism $\breve{H}^*(X,\Gamma(A_\rho)) \simeq H^*(X,A_\rho)$.
\end{prop}

In particular this applies to all smooth manifolds. Henceforth whenever we consider cohomology with local coefficients of a space $X$ we assume $X$ is triangulable and use $H^*(X,A_\rho)$ to denote either singular or \v{C}ech cohomology.\\

Next we will consider de Rham cohomology twisted by local coefficients. Let $X$ be a smooth manifold and $A_\rho \to X$ a local system with coefficients a finitely generated abelian group $A$. Then $A_\mathbb{R} = A \otimes_{\mathbb{Z}} \mathbb{R}$ is a finite dimensional vector space and $(A_\rho)_\mathbb{R} = \tilde{X} \times_\rho A_\mathbb{R}$ is a vector bundle with flat connection $\nabla$. The monodromy of $\nabla$ is the image of $\rho$ under ${\rm Aut}(A) \to {\rm GL}(A_\mathbb{R})$. We can form the de Rham complex twisted by $\nabla$ which is the space $\Omega^*(X,(A_\rho)_\mathbb{R}) = \mathcal{C}^\infty( \wedge^* T^*X \otimes (A_\rho)_\mathbb{R})$. There is a twisted exterior derivative $d_\nabla$ characterized by the property
\begin{equation*}
d_\nabla ( \omega \otimes a) = d\omega \otimes a + (-1)^p \omega \wedge \nabla a
\end{equation*}
where $\omega$ is a $p$-form and $a$ is a section of $(A_\rho)_\mathbb{R}$. We let $H^*_\nabla(X,(A_\rho)_\mathbb{R} )$ denote the cohomology of this complex.

\begin{prop}
Let $X$ be a smooth manifold and $A_\rho$ a system of local coefficients with monodromy $\rho$ for a finitely generated abelian group $A$. There is a canonical isomorphism
\begin{equation*}
H^*_\nabla(X,(A_\rho)_\mathbb{R}) \simeq H^*(X,(A_\rho)_\mathbb{R})
\end{equation*}
where $(A_\rho)_\mathbb{R}$ is the local system $(A_\rho)_\mathbb{R} = \tilde{X} \times_\rho A_\mathbb{R}$.
\begin{rem}
Note that the set $(A_\rho)_\mathbb{R}$ can be viewed as either a vector bundle or a covering space. The topologies are different.
\end{rem}
\begin{proof}
Let $\Gamma((A_\rho)_\mathbb{R})$ be the sheaf of sections of $(A_\rho)_\mathbb{R}$ with the covering space topology. Then clearly $\Gamma((A_\rho)_\mathbb{R})$ is also the sheaf of covariantly constant sections of $(A_\rho)_\mathbb{R}$ viewed as a flat vector bundle. From here the proof follows the usual \v{C}ech-de Rham isomorphism. We have an acyclic resolution
\begin{equation*}\xymatrix{
0 \ar[r] & \Gamma((A_\rho)_\mathbb{R}) \ar[r]^-i & \Omega^0(X,(A_\rho)_\mathbb{R}) \ar[r]^-{d_\nabla} & \Omega^1(X,(A_\rho)_\mathbb{R}) \ar[r]^-{d_\nabla} & \dots
}
\end{equation*}
Break the sequence up into short exact sequences, take the associated long exact sequences and use the fact that each $\Omega^k(X,(A_\rho)_\mathbb{R})$ is acyclic. We thus obtain canonical isomorphisms 
$H^*_\nabla(X,(A_\rho)_\mathbb{R}) \simeq \breve{H}^*(X,\Gamma((A_\rho)_\mathbb{R}))$.
\end{proof}
\end{prop}

Next we recall that Poincar\'{e} duality has an extension to compact non-oriented manifolds using local coefficients:
\begin{prop}[Poincar\'{e} duality]
Let $X$ be a compact $n$-manifold. There is a local system $\mathbb{Z}_{\rm orn}$ with coefficients in $\mathbb{Z}$ given by assigning to $x \in X$ the group $H_n(X,X-\{x\},\mathbb{Z}) \simeq \mathbb{Z}$. For any local system $A_\rho$ There are canonical isomorphisms
\begin{equation*}
H^i(X, A_\rho) \simeq H_{n-i}(X , A_\rho \otimes \mathbb{Z}_{\rm orn}).
\end{equation*}
\end{prop}

We would also like to point out the existence of a push-forward map for cohomology with local coefficients. For our purposes we really just need the following special case: let $\pi : E \to B$ be a fibre bundle with fibre $F$ a compact oriented manifold of dimension $d$. We get a local system $\Gamma_F$ where $\Gamma_F$ is the sheaf associated to the presheaf $U \mapsto H^d(\pi^{-1}(U) , \mathbb{Z})$. Since $F$ is oriented we see that $\Gamma_F$ is a local system with coefficient group $\mathbb{Z}$. If $A_\rho$ is any other local system on $B$ there is a push-forward map
\begin{equation*}
\pi_* : H^i(E,\pi^*(A_\rho)   ) \to H^{i-d}(B , A_\rho \otimes \Gamma_F).
\end{equation*}
One way to define the push-forward is through the Leray-Serre spectral sequence for the fibration $\pi : E \to B$ \cite{bh} in which case it corresponds to the composition
\begin{equation*}
H^i(E , \pi^*( A_\rho) ) \to E^{i-d,d}_\infty \to E^{i-d,d}_2 = H^{i-d}(B , A_\rho \otimes \Gamma_F).
\end{equation*}
In the case of de Rham cohomology twisted by a local system the push-forward really is an integration over the fibres. The integration picks up a sign ambiguity in choice of fibre orientation and this is accounted for by the local system $\Gamma_F$.


\section{Affine torus bundles}

\subsection{Classification}\label{affcl}

Let $T^n = \mathbb{R}^n/\mathbb{Z}^n$ be the standard $n$-dimensional torus. The classification of principal $T^n$-bundles is straightforward and very well known. Since we are interested in extending topological T-duality to the non-principal case we would like a classification of arbitrary torus bundles. This is a difficult question but there is an easier intermediate, namely affine torus bundles. For $n \le 3$ it turns out that every torus bundle admits an affine structure, so this gives a complete classification of $T^n$-bundles for $n \le 3$.\\

Let ${\rm Diff}(T^n)$ be the group of diffeomorphisms of $T^n$. Given a subgroup $G$ of ${\rm Diff}(T^n)$ we can consider torus bundles constructed out of transition maps valued in $G$. Principal $T^n$-bundles for example have transitions functions valued in $T^n$ which acts on itself by translation. We will consider a more general class of torus bundles, namely {\em affine torus bundles}, those built out of transition functions valued in the group ${\rm Aff}(T^n) = {\rm GL}(n,\mathbb{Z}) \ltimes T^n$ which acts on $T^n$ by affine transformations. It turns out that for $n \le 3$ the inclusion ${\rm Aff}(T^n) \to {\rm Diff}(T^n)$ is a homotopy equivalence, from which it follows that every $T^n$-bundle with $n \le 3$ admits an affine structure. The case $n=1$ is straightforward, the case $n=2$ was proved in \cite{eaee} and $n=3$ in \cite{hat}.\\

In what follows we will come to a classification of affine torus bundles. The result is that affine $T^n$-bundles correspond to pairs $(\rho , c)$ where $\rho : \pi_1(M) \to {\rm GL}(n,\mathbb{Z})$ is a homomorphism and $c \in H^2(M , \Lambda_\rho)$, where $\Lambda_\rho$ is a local system determined by $\rho$. We identify pairs $(\rho , c)$, $(\rho' , c')$ when they are related by the action of ${\rm GL}(n,\mathbb{Z})$.\\

Let $V = \mathbb{R}^n$ denote an $n$-dimensional vector space and $\Lambda \subset V$ a lattice of rank $n$. Let $\Lambda^* = {\rm Hom}_{\mathbb{Z}}(\Lambda , \mathbb{Z}) \subset V^*$ be the dual lattice. The cohomology ring $H^*(T^n , \mathbb{Z})$ of the torus $T^n = V/\Lambda$ can be identified with the exterior algebra $\wedge^* \Lambda^*$ of the dual lattice, that is the subring of $\wedge^* V^*$ generated over $\mathbb{Z}$ by $\Lambda^*$.\\

Let $\pi : E \to M$ be an affine $T^n$-bundle. Thus $E$ is associated to a principal ${\rm Aff}(T^n)$-bundle over $M$. Using the homomorphism ${\rm Aff}(T^n) \to {\rm GL}(n,\mathbb{Z})$ we get a principal ${\rm GL}(n,\mathbb{Z})$-bundle over $M$ or equivalently a representation $\rho : \pi_1(M) \to {\rm GL}(n,\mathbb{Z})$ called the {\em monodromy representation} of $E$. The monodromy representation has a more direct interpretation as we now explain. Let $\tilde{M} \to M$ be the corresponding principal ${\rm GL}(n,\mathbb{Z})$-bundle determined by $\rho$ and let $\Lambda_\rho$ be the local system of groups $\Lambda_\rho = \tilde{M} \times_{{\rm GL}(n,\mathbb{Z})} \Lambda$. It is not hard to see that the dual local system $\Lambda^*_\rho$ is the sheaf associated to the presheaf which maps an open set $U \subseteq M$ to $H^1(\pi^{-1}(U),\mathbb{Z})$. More generally we get a local system by taking fibre cohomology of any degree. Since $E$ is a torus bundle the sheaf associated to the presheaf $U \mapsto H^k(\pi^{-1}(U),\mathbb{Z})$ can be identified with $\wedge^k \Lambda^*_\rho$.\\

As above we let $\tilde{M}$ be the principal ${\rm GL}(n,\mathbb{Z})$-bundle corresponding to $\rho$. Let $\tilde{E}$ be the pull-back of $E$ to a torus bundle on $\tilde{M}$. The local system $\Lambda_\rho$ pulled back to $\tilde{M}$ becomes trivial. It follows that $\tilde{E}$ admits the structure of a principal $T^n$-bundle compatible with the existing affine structure. Moreover, since $\tilde{E}$ is a pull-back by the map $\pi : \tilde{M} \to M$ is follows that there is a right action of ${\rm GL}(n,\mathbb{Z})$ on $\tilde{E}$ by fibre bundle automorphisms. The action of ${\rm GL}(n,\mathbb{Z})$ on $\tilde{E}$ mixes up the principal bundle action of $T^n$ on $\tilde{E}$ according to the usual action of ${\rm GL}(n,\mathbb{Z})$ as group automorphisms of $T^n$. To put it more concretely $\tilde{E} \to M$ is a principal ${\rm GL}(n,\mathbb{Z}) \ltimes T^n$-bundle, that is a principal ${\rm Aff}(T^n)$-bundle over $M$. Of course $E$ is the associated fibre bundle $E = \tilde{E} \times_{ {\rm Aff}(T^n) } T^n$.\\

Let $\tilde{E} \to M$ be a principal ${\rm Aff}(T^n)$-bundle over $M$. We identify ${\rm Aff}(T^n)$ with ${\rm GL}(n,\mathbb{Z}) \times T^n$ with group law
\begin{equation*}
(g , t )(g' , t') = (gg' , t g(t')).
\end{equation*}
There is an open cover $\{ U_i \}$ of $M$ such that $\tilde{E}$ admits trivializations $U_i \times {\rm Aff}(T^n)$ with transition maps $h_{ij} : U_{ij} \to {\rm Aff}(T^n)$ of the form 
\begin{equation*}
h_{ij} = g_{ij} t_{ij} = (g_{ij} , g_{ij}(t_{ij}) )
\end{equation*}
where $g_{ij} : U_{ij} \to {\rm GL}(n,\mathbb{Z})$ are transition functions for the underlying principal ${\rm GL}(n,\mathbb{Z})$-bundle $\tilde{M} \to M$ and $t_{ij}$ are maps $t_{ij} : U_{ij} \to T^n$.\\
We identify the trivializations $U_i \times {\rm Aff}(T^n)$ with $U_i \times {\rm GL}(n,\mathbb{Z}) \times T^n$ by multiplication ${\rm GL}(n,\mathbb{Z}) \times T^n \to {\rm Aff}(T^n)$. Under this identification the transition maps $h_{ij}$ act as
\begin{equation}\label{htran}
h_{ij}(u,g,v) = (u , g_{ij}(u)g , t_{ij}(u,g)v )
\end{equation}
where $t_{ij}(u,g)$ is defined to be $g^{-1}t_{ij}(u)$. One checks that $t_{ij}(u,g) : U_{ij} \times {\rm GL}(n,\mathbb{Z}) \to T^n$ are the transition maps defining $\tilde{E}$ as principal $T^n$-bundle over $\tilde{M}$.\\

Conversely given a principal ${\rm GL}(n,\mathbb{Z})$-bundle $\tilde{M} \to M$ defined by transition functions $g_{ij} : U_{ij} \to {\rm GL}(n,\mathbb{Z})$ and functions $t_{ij}: U_{ij} \times {\rm GL}(n,\mathbb{Z}) \to T^n$ satisfying the conditions:
\begin{itemize}
\item{$t_{ii} = 1$ on $U_i$}
\item{$t_{ij}(u,g) t_{ji}(u,g_{ij}(u)g) = 1$ on $U_{ij}$}
\item{$t_{ij}(u,g) t_{jk}(u , g_{ij}(u)g) = t_{ik}(u,g)$ on $U_{ijk}$ and}
\item{$t_{ij}(u,gh) = h^{-1}t_{ij}(u,g)$}
\end{itemize}
we construct a principal ${\rm Aff}(T^n)$-bundle $\tilde{E} \to M$ with transition maps (\ref{htran}).\\

Assume now that the cover $\{ U_i \}$ is such that on double intersections we can find lifts $\tilde{t}_{ij} : U_{ij} \times {\rm GL}(n,\mathbb{Z}) \to \mathbb{R}^n$ of the maps $t_{ij} : U_{ij} \times {\rm GL}(n,\mathbb{Z}) \to T^n$. We may also assume the lifts $\tilde{t}_{ij}$ are equivariant in that $\tilde{t}_{ij}(u,gh) = h^{-1} \tilde{t}_{ij}(u,g)$. The coboundary $c_{ijk} = \tilde{t}_{ij} + \tilde{t}_{jk} + \tilde{t}_{ki}$ is a map $c_{ijk} : U_{ijk} \times {\rm GL}(n,\mathbb{Z}) \to \Lambda$ which is likewise equivariant. Thus $\{ c_{ijk} \}$ determines an element in $H^2(M,\Lambda_\rho)$, degree $2$ cohomology with local coefficients. It is not hard to see that the cohomology class of $\{ c_{ijk} \}$ is independent of the choice of cover and choice of local trivializations. Thus associated to any principal ${\rm Aff}(T^n)$-bundle $\tilde{E} \to M$ is a class $c \in H^2(M,\Lambda_\rho)$ which we call the {\em twisted Chern class} of $\tilde{E}$.

\begin{prop}
For every representation $\rho : \pi_1(M) \to {\rm GL}(n,\mathbb{Z})$ and class $c \in H^2(M,\Lambda_\rho)$ there is a principal ${\rm Aff}(T^n)$-bundle over $M$ with monodromy $\rho$ and twisted Chern class $c$.
\begin{proof}
Let $p : \tilde{M} \to M$ be the principal ${\rm GL}(n,\mathbb{Z})$-bundle corresponding to $\rho$. If we take the exact sequence of sheaves on $\tilde{M}$
\begin{equation*}
1 \to \Lambda \to \mathcal{C}( \mathbb{R}^n) \to \mathcal{C}( {\rm U}(1)^n ) \to 1
\end{equation*}
and quotient by ${\rm GL}(n,\mathbb{Z})$ we get a corresponding exact sequence of local systems on $M$ of the form
\begin{equation*}
1 \to \Lambda_\rho \to \mathcal{C}( \mathbb{R}_\rho^n) \to \mathcal{C}( {\rm U}(1)_\rho^n ) \to 1.
\end{equation*}
Now since $\mathcal{C}( \mathbb{R}_\rho^n)$ is a fine sheaf we have an isomorphism $H^2(M,\Lambda_\rho) = H^1(M , \mathcal{C}( {\rm U}(1)_\rho^n ) )$. Elements of $H^1(M , \mathcal{C}( {\rm U}(1)_\rho^n ) )$ can be represented as follows. There is an open cover $\{ U_{ij} \}$ of $M$ such that $p^{-1}(U_i) = U_i \times {\rm GL}(n,\mathbb{Z})$ and a collection $\{ t_{ij} \}$ of functions $t_{ij} : p^{-1}(U_{ij}) = U_{ij} \times {\rm GL}(n,\mathbb{Z}) \to {\rm U}(1)^n$. The functions $\{ t_{ij} \}$ satisfy precisely the conditions required to construct a principal ${\rm Aff}(T^n)$-bundle $\tilde{E}$. Moreover the corresponding element $c \in H^2(M,\Lambda_\rho)$ is easily seen to be the twisted Chern class of $\tilde{E}$. 
\end{proof}
\end{prop}

If $E \to M$ is an affine torus bundle then we have seen that there is a principal ${\rm GL}(n,\mathbb{Z})$-bundle $p : \tilde{M} \to M$ such that $\tilde{E} = p^*(E)$ admits the structure of a principal ${\rm Aff}(T^n)$-bundle and that $E$ is the associated bundle $E = \tilde{E} \times_{{\rm Aff}(T^n)} T^n$. Associated to $\tilde{E}$ is a twisted Chern class $c \in H^2(M,\Lambda_\rho)$. However distinct principal ${\rm Aff}(T^n)$-bundles may give rise to the same $T^n$-bundle, so different classes in $H^2(M,\Lambda_\rho)$ may be associated to the same $T^n$-bundle. The following proposition determines precisely when this happens:

\begin{prop}\label{glact}
If $M$ is connected then two affine torus bundles $E \to M$, $E' \to M$ with monodromy representations $\rho,\rho'$ and twisted Chern classes $c,c'$ are isomorphic as $T^n$-bundles if and only if the pairs $(\rho,c)$ and $(\rho',c')$ are related by the action of an element of ${\rm GL}(n,\mathbb{Z})$.
\begin{proof}
Consider the case where $E,E' \to M$ are in fact principal $T^n$-bundles. There exists an open cover $\{ U_i \}$ such that we have local trivializations $U_i \times T^n$ for $E,E'$ with transition functions $t_{ij},t'_{ij} : U_{ij} \to T^n$. We can assume the cover is such that $t_{ij},t'_{ij}$ admit lifts $\tilde{t}_{ij},\tilde{t}'_{ij} : U_{ij} \to \mathbb{R}^n$. The Chern classes of $E,E'$ are then represented by $c_{ijk} = \tilde{t}_{ij} + \tilde{t}_{jk} + \tilde{t}_{ki}$ and $c'_{ijk} = \tilde{t}'_{ij} + \tilde{t}'_{jk} + \tilde{t}'_{ki}$.\\

Now suppose $\phi : E \to E'$ is a bundle isomorphism, but not necessarily a principal bundle isomorphism. In local trivializations $\phi$ is given by maps $\phi_i : U_i \times T^n \to U_i \times T^n$ fitting in to commuative diagrams
\begin{equation*}\xymatrix{
U_{ij} \times T^n \ar[r]^{\phi_j} \ar[d]^{t_{ij}} & U_{ij} \times T^n \ar[d]^{t'_{ij}} \\
U_{ij} \times T^n \ar[r]^{\phi_i}  & U_{ij} \times T^n
}
\end{equation*}
passing to homology we see that $(\phi_i)_* = (\phi_j)_* : \Lambda \to \Lambda$ so in each trivialization $(\phi_i)_*$ acts as some fixed $g \in {\rm GL}(n,\mathbb{R})$.\\

A homeomorphism $T^n \to T^n$ lifts to a homeomorphism $\mathbb{R}^n \to \mathbb{R}^n$ because $\mathbb{R}^n$ is the universal cover of $T^n$. In particular we can lift the maps $\phi_i$ to $\tilde{\phi}_i : U_i \times \mathbb{R}^n \to U_i \times \mathbb{R}^n$. The commutations relations $\phi_i t_{ij} = t'_{ij} \phi_j$ lift up to a map $\mathbb{R}^n \to \mathbb{R}^n$ which is the identity on $T^n$. Such a map must be translation by an element of $\Lambda$. If we assume the $U_{ij}$ are simply connected then we can find $\tau_{ij} \in \Lambda$ such that
\begin{equation*}
\tilde{\phi}_i \tilde{t}_{ij} \tau_{ij} = \tilde{t}'_{ij} \tilde{\phi}_j
\end{equation*}
where we identify an element of $\mathbb{R}^n$ with the corresponding translation $\mathbb{R}^n \to \mathbb{R}^n$. Now observe that
\begin{eqnarray*}
\tilde{\phi}_i c_{ijk} \tau_{ij} \tau_{jk} \tau_{ki} &=& \tilde{\phi}_i \tilde{t}_{ij} \tau_{ij} \tilde{t}_{jk} \tau_{jk} \tilde{t}_{ki} \tau_{ki} \\
&=& \tilde{t}'_{ij} \tilde{\phi}_j \tilde{t}_{jk} \tau_{jk} \tilde{t}_{ki} \tau_{ki} \\
&=& \tilde{t}'_{ij} \tilde{t}'_{jk} \tilde{\phi}_k \tilde{t}_{ki} \tau_{ki} \\
&=& \tilde{t}'_{ij} \tilde{t}'_{jk} \tilde{t}'_{ki} \tilde{\phi}_i \\
&=& c'_{ijk} \tilde{\phi}_i.
\end{eqnarray*}
As elements of the \v{C}ech cochain complex we can write this as
\begin{equation*}
c_{ijk} + (\delta \tau)_{ijk} = \tilde{\phi}_i^{-1}(c'_{ijk})
\end{equation*}
but note that since $c'_{ijk}$ is valued in $\Lambda$ we have that $\tilde{\phi}_i^{-1}(c'_{ijk}) = g^{-1} c'_{ijk}$. We have thus shown that the Chern classes of $E,E'$ are related by an element of ${\rm GL}(n,\mathbb{Z})$.\\

In the more general case of affine torus bundles a similar proof applies.
\end{proof}
\end{prop}

{\bf Classification of circle bundles:} consider the special case of circle bundles over $M$. Since every circle bundle $\pi : E \to M$ has an affine structure they are determined by pairs $(\xi,c)$ where $\xi$ is a homomorphism $\xi : \pi_1(M) \to {\rm GL}(1,\mathbb{Z}) = \mathbb{Z}_2$. Thus $\xi$ is an element of $H^1(M,\mathbb{Z}_2)$. The class $\xi$ measures the orientability of $E \to M$ so we call it the {\em first Stiefel-Whitney class} of $E$ and write $\xi = w_1(E)$. In the differentiable case $w_1(E)$ is the first Stiefel-Whitney class of the vertical bundle ${\rm Ker}(\pi_*)$.

Let $\mathbb{Z}_\xi$ be the local system with $\mathbb{Z}$-coefficients determined by $\xi$. So the twisted Chern class of a circle bundle with monodromy $\xi$ is an element of $H^2(M,\mathbb{Z}_\xi)$.

We conclude that circle bundles correspond to pairs $(\xi,c)$ with $\xi \in H^1(M,\mathbb{Z}_2)$ and $c \in H^2(M,\mathbb{Z}_\xi)$ modulo the action of ${\rm GL}(1,\mathbb{Z}) = \mathbb{Z}_2$ where $-1$ sends a pair $(\xi,c)$ to $(\xi , -c)$.


\subsection{Leray-Serre spectral sequence}

For T-duality we will need to know the differentials on the $E_2$ stage of the Leray-Serre spectral sequence for affine torus bundles. This section is dedicated to a proof of their structure, as given in Proposition \ref{ssdif}.\\

Let $\pi : E \to M$ be an affine $T^n$-bundle over $M$. We assume throughout that $M$ is connected. We can associate to $E$ a monodromy representation $\rho : \pi_1(M) \to {\rm GL}(n,\mathbb{Z})$ and a twisted Chern class $c \in H^2(M, \Lambda_\rho)$. As we have shown in Proposition \ref{glact} the pair $(\rho,c)$ is only unique up to the action of ${\rm GL}(n,\mathbb{Z})$.\\

Much information about the cohomology ring $H^*(E,\mathbb{Z})$ can be obtained by the Leray-Serre spectral sequence of the fibration $\pi : E \to M$. Since we do not assume the action of $\pi_1(M)$ on the cohomology of the fibres is trivial we must use cohomology with local coefficients. The $E_2$ page of the spectral sequence is given by
\begin{equation*}
E_2^{p,q} = H^p(M , \wedge^q \Lambda^*_\rho).
\end{equation*}
The ring structure on $E_2$ is the natural multiplication
\begin{equation*}
H^p(M,\wedge^q \Lambda^*_\rho) \otimes H^{p'}(M,\wedge^{q'} \Lambda^*_\rho) \to H^{p+p'}(M,\wedge^{q+q'} \Lambda^*_\rho)
\end{equation*}
obtained from the wedge product on $\wedge^* \Lambda_\rho^*$.\\

Recall that the fibration $E \to M$ induces a filtration on $H^m(E)$
\begin{equation*}
0 = F^{m+1,m} \subseteq F^{m,m} \subseteq F^{m-1,m} \subseteq \dots \subseteq F^{0,m} = H^m(E)
\end{equation*}
such that the spectral sequence converges to the associated graded ring
\begin{equation*}
{\rm Gr}(H^*(E)) = \bigoplus_{p,m} \left( F^{p,m} / F^{p+1,m} \right)
\end{equation*}
or more specifically,
\begin{equation*}
E_\infty^{p,q} = F^{p,p+q}/F^{p+1,p+q}.
\end{equation*}

Of particular interest are the bottom row ($q = 0$) and top row ($q = n$) of the spectral sequence for we have natural maps
\begin{equation}\label{pull}
H^p(M,\mathbb{Z}) = E_2^{p,0} \to E_\infty^{p,0} = F^{p,p} \to H^p(E,\mathbb{Z})
\end{equation}
and
\begin{equation}\label{push}
H^{n+p}(E,\mathbb{Z}) \to F^{p,n+p}/F^{p+1,n+p} = E_\infty^{p,n} \to E_2^{p,n} = H^p(M , \wedge^n \Lambda^*_\rho ).
\end{equation}

It is well known that the map (\ref{pull}) is the pull-back $\pi^* : H^p(M,\mathbb{Z}) \to H^p(E,\mathbb{Z})$. As we have already argued the map (\ref{push}) is the push-forward $\pi_* : H^{n+p}(M,\mathbb{Z}) \to H^p(M, \wedge^n \Lambda^*_\rho )$.\\

We will need a description of the differentials for the $E_2$ stage of the spectral sequence. Consider first the case of a principal $T^n$-bundle. Here the monodromy $\rho$ is trivial and we thus have $E_2^{p,q} = H^p(M,\mathbb{Z}) \otimes \wedge^q \Lambda^*$. In particular consider the differential $d_2 : H^0(M , \mathbb{Z}) \otimes \Lambda^* \to H^2(M,\mathbb{Z})$. This is an element of ${\rm Hom}_{\mathbb{Z}}( \Lambda^* , H^2(M , \mathbb{Z})) = H^2(M , \Lambda)$. It is not hard to see that this is precisely the Chern class $c \in H^2(M , \Lambda)$. Now we can completely determine the differentials $d_2 : E_2^{p,q} \to E_2^{p+2,q-1}$ using the derivation property. Indeed $d_2 : H^p(M , \mathbb{Z}) \otimes \wedge^q \Lambda^* \to H^{p+2}(M,\mathbb{Z}) \otimes \wedge^{q-1} \Lambda^*$ is given $d_2(x) = c \smallsmile x$ where $c \smallsmile x$ is obtained by using the product in cohomology and the contraction $\Lambda \otimes \wedge^q \Lambda^* \to \wedge^{q-1} \Lambda^*$. More precisely if $a \in H^2(M,\mathbb{Z})$, $b \in \Lambda$, $c \in H^p(M,\mathbb{Z})$, $d \in \wedge^q \Lambda^*$ then we define
\begin{equation*}
(a \otimes b) \smallsmile (c \otimes d) = (-1)^p (a \smallsmile b) \otimes ( i_b d)
\end{equation*}
where $i_b d$ denotes the contraction of $d$ by $b$.\\

For affine $T^n$-bundles with non-trivial monodromy $\rho$ one might guess that the differentials $d_2 : H^p(M , \wedge^q \Lambda_\rho^*) \to H^{p+2}(M, \wedge^{q-1} \Lambda_\rho^*)$ are similarly given by $d_2(x) = c \smallsmile x$ where $c$ is the twisted Chern class $F \in H^2(M , \Lambda_\rho)$ and $c \smallsmile x$ is the combination of the product in cohomology with contraction. In fact this is the case however the proof is more subtle since generally the product $H^0(M , \wedge^q \Lambda_\rho^*) \otimes H^p(M , \mathbb{Z}) \to H^p(M , \wedge^q \Lambda^*_\rho )$ need not be surjective. 

\begin{prop}\label{ssdif}
The differentials $d_2 : H^p(M , \wedge^q \Lambda_\rho^*) \to H^{p+2}(M, \wedge^{q-1} \Lambda_\rho^*)$ for the Leray-Serre spectral sequence are given by the product with the twisted Chern class followed by the contraction $\Lambda_\rho \otimes \wedge^q \Lambda_\rho^* \to \wedge^{q-1} \Lambda_\rho^*$.
\begin{proof}
To prove this result we will step back and examine the construction of the Leray-Serre spectral sequence. The approach that is the most convenient for us uses a combination of \v{C}ech and singular cohomology \cite{botttu}.

Let $\mathcal{U} = \{ U_i \}_{i \in J}$ be a good cover of $M$, that is each $U_i$ as well as the multiple intersections $U_{i_0 i_1 \dots i_p} = U_{i_0} \cap U_{i_1} \cap \dots \cap U_{i_p}$ are all contractible. We also assume $J$ is a countable ordered set.\\

Let $\pi : E \to M$ be a fibre bundle over $M$ and let $\pi^{-1}\mathcal{U} = \{ \pi^{-1}(U_i) \}_{i \in J}$ be the the pull-back cover. We note that $\pi^{-1}(U_{i_0}) \cap \pi^{-1}(U_{i_1}) \cap \dots \cap \pi^{-1}(U_{i_p}) = \pi^{-1}( U_{i_0 i_1 \dots i_p })$. Consider the double complex $K^{*,*}$ given by
\begin{equation*}
K^{p,q} = C^p( \pi^{-1}\mathcal{U} , S^q ) = \prod_{ i_0 < i_1 < \dots < i_p } S^q( \pi^{-1}(U_{i_0 i_1 \dots i_p }) , \mathbb{Z} )
\end{equation*}
where $S^q$ denotes the singular cochain complex. There are two commuting differentials $\delta_S,\delta_C$ on $K^{*,*}$, namely the differential $\delta_S$ coming from the singular cochain complex and the differential $\delta_C$ coming from the \v{C}ech cochain complex. If we let $D' = \delta_C$, $D'' = (-1)^p \delta_S$ then $D = D' + D''$ is a differential $ D : K^i \to K^{i+1}$ where $K^i = \bigoplus_{p+q=i} K^{p,q}$. In \cite{botttu} it is shown that the cohomology of $(K^* , D)$ is the singular cohomology $H^*(E,\mathbb{Z})$ of $E$. The Leray-Serre spectral sequence can be derived as the spectral sequence associated to this double complex.\\

Now let $\pi : E \to M$ be an affine $T^n$-bundle over $M$ with monodromy representation $\rho$. Let $p : \tilde{M} \to M$ be the principal ${\rm GL}(n,\mathbb{Z})$-bundle corresponding to $\rho$, so $\tilde{E} = p^{-1}E$ admits the structure of a principal $T^n$-bundle over $\tilde{M}$ and is furthermore a principal ${\rm Aff}(T^n)$-bundle over $M$. We have a commutative diagram of the form
\begin{equation*}\xymatrix{
\tilde{E} \ar[r]^{\tilde{p}} \ar[d]^{\tilde{\pi}} \ar[dr]^q & E \ar[d]^\pi \\
\tilde{M} \ar[r]^p & M
}
\end{equation*}

There are local trivializations $q^{-1}(U_i) \simeq U_i \times {\rm GL}(n,\mathbb{Z}) \times T^n$ which under the identification $m : {\rm GL}(n,\mathbb{Z}) \times T^n \to {\rm Aff}(T^n)$ given by multiplication become trivializations of $\tilde{E} \to M$ as a principal ${\rm Aff}(T^n)$-bundle.

Let $t_{ij} : U_{ij} \times {\rm GL}(n,\mathbb{Z}) \to T^n$ be transition functions for the $T^n$-bundle $\tilde{E} \to \tilde{M}$. We can take $t_{ij}$ to be equivariant in the sense that $t_{ij}(u,g) = g^{-1} t_{ij}(u,1)$. Write $t_{ij}(u) = t_{ij}(u,1)$. The let $g_{ij} : U_{ij} \to {\rm GL}(n,\mathbb{Z})$ be transition functions for $\tilde{M} \to M$. It follows that the ${\rm Aff}(T^n)$-bundle $\tilde{E} \to M$ has transition functions $h_{ij}(u) = g_{ij}(u) t_{ij}(u)$.\\

Since $\mathcal{U}$ is a good cover we can find lifts $\tilde{t}_{ij} : U_{ij} \to \mathbb{R}^n$ of the $t_{ij}$ which extend to equivariant maps $\tilde{t}_{ij} : U_{ij} \times {\rm GL}(n,\mathbb{Z}) \to \mathbb{R}^n$. Moreover we get equivariant maps $c_{ijk} : U_{ijk} \times {\rm GL}(n,\mathbb{Z}) \to \Lambda$ given by $c_{ijk} = \tilde{t}_{ij} + \tilde{t}_{jk} + \tilde{t}_{ki}$. It is clear that $\{ c_{ijk} \}$ represents the twisted Chern class for $E \to M$.\\

We now have two fibre bundles $\pi : E \to M$ and $\tilde{\pi} : \tilde{E} \to \tilde{M}$ and two open covers $\mathcal{U},p^{-1}\mathcal{U}$ of the bases $M,\tilde{M}$ so we get two associated double complexes. We denote these by $K^{p,q}(\pi)$, $K^{p,q}(\tilde{\pi})$ and the associated spectral sequences $E_r^{p,q}(\pi)$, $E_r^{p,q}(\tilde{\pi})$. Note that the action of ${\rm GL}(n,\mathbb{Z})$ on $\tilde{M}$ lifts to an action on $\tilde{E}$ and this action restricts to an action on the subsets $q^{-1}( U_{i_0 i_1 \dots i_p} )$ lifting the action on $p^{-1}( U_{i_0 i_1 \dots i_p} )$. This induces an action on the spaces $S^q( q^{-1}(U_{i_0 i_1 \dots i_p }))$ of singular cochains and thus an action on $K^{p,q}(\tilde{\pi})$ commuting with the differentials. If $U_{i_0 i_1 \dots i_p}$ is non-empty then $p^{-1}( U_{i_0 i_1 \dots i_p} )$ is isomorphic to $U_{i_0 i_1 \dots i_p} \times {\rm GL}(n,\mathbb{Z})$ as a principal ${\rm GL}(n,\mathbb{Z})$-bundle. It follows that $K^{p,q}(\pi)$ can be identified with the ${\rm GL}(n,\mathbb{Z})$-invariant subspace of $K^{p,q}(\tilde{\pi})$.\\

We need also to consider the $E_1$-stage of the spectral sequences. It is clear that $E_1^{p,q}(\pi) = C^p( \mathcal{U} , \wedge^q \Lambda^*_\rho )$ and that $E_1^{p,q}(\tilde{\pi}) = C^p( p^{-1} \mathcal{U} , \wedge^q \Lambda^* )$. Notice that $E_1^{p,q}(\pi)$ can be identified with the ${\rm GL}(n,\mathbb{Z})$-invariant subspace of $E_1^{p,q}(\tilde{\pi})$. Note however that by the $E_2$-stage a similar identification need not hold. It is for this reason that we need to go back to the earlier stages of the spectral sequence.\\

Take an element $a \in E_2^{p,q}(\pi)$. We need to determine $d_2(a)$. Represent $a$ by an element $w \in K^{p,q}(\pi)$, which we can think of as an invariant element of $K^{p,q}(\tilde{\pi})$. An element $z \in K^{p+2,q-1}(\pi)$ is a representative for $d_2(a)$ if there exists elements $x \in K^{p+1,q}(\pi)$, $y \in K^{p+1,q-1}(\pi)$ related as follows:
\begin{equation*}\xymatrix{
0 & & \\
w \ar[u]^{(-1)^p\delta_S} \ar[r]^{\delta_C} & x & \\
& y \ar[r]^{\delta_C} \ar[u]^{(-1)^{p-1}\delta_S} & z
}
\end{equation*}
Modulo the image of $\delta_S$ every representative of $d_2(a)$ has this form for some $w,x,y,z$.\\

Consider first an element $a \in E_2^{0,1}(\tilde{\pi})$ which is not assumed to be ${\rm GL}(n,\mathbb{Z})$-invariant. Let $w_1$ be a representative for $a$ in $E_1^{0,1}(\tilde{\pi})$. We have that $E_1^{0,1}(\tilde{\pi}) = C^0( p^{-1}\mathcal{U} , \Lambda^*_\rho) = \prod_i {\rm Map}({\rm GL}(n,\mathbb{Z}) , \Lambda^*)$. Thus $w_1$ is given by a collection $\{ \alpha_i \}$ of maps $\alpha_i : {\rm GL}(n,\mathbb{Z}) \to \Lambda^*$. For reasons that will become clear later we restrict to the case where the $\alpha_i$ are constant maps $\alpha_i : {\rm GL}(n,\mathbb{Z}) \to \Lambda^*$. Since $w_1$ represents the element $a \in E_2^{0,1}(\tilde{\pi})$ we have $d_1(w_1) = 0$. This can only be possible if for all non-empty $U_{ij}$ we have $\alpha_i = \alpha_j$. Thus there exists $\alpha \in \Lambda^*$ such that $\alpha = \alpha_i$ for all $i$.\\

We can represent $w_1$ by $w \in K^{0,1}(\tilde{\pi})$ as follows: for each trivialization we have a projection map $pr_i : U_i \times {\rm GL}(n,\mathbb{Z}) \times T^n \to T^n$. Since $\alpha_i$ is a constant it can be viewed as an element of $\Lambda^* = H^1(T^n,\mathbb{Z})$. Choose a lift $\tilde{\alpha} \in S^1(T^n,\mathbb{Z})$ of $\alpha$, where $S^*(T^n,\mathbb{Z})$ is the singular cochain complex. We then take $w \in K^{0,1}(\tilde{\pi}) = \prod_i S^1 ( U_i \times {\rm GL}(n,\mathbb{Z}) \times T^n , \mathbb{Z})$ to be $w = \{ \beta_i \}$ where $\beta_i = pr_i^*( \tilde{\alpha} )$. It is clear that $w$ is a representative for $w_1$ in $K^{0,1}(\tilde{\pi})$.\\

Now we attempt to determine $d_2(a)$ starting from the representative $w$. Set $x = \delta_C w$. Thus $x = \{ r_{ij} \}$ is the element of $\prod_{i < j} S^1 ( U_{i j} \times {\rm GL}(n,\mathbb{Z}) \times T^n , \mathbb{Z})$ given by
\begin{equation}\label{rij}
r_{ij} = h^*_{ij} (\beta_j|_{U_{ij} \times {\rm GL}(n,\mathbb{Z}) \times T^n} ) - \beta_i|_{U_{ij} \times {\rm GL}(n,\mathbb{Z}) \times T^n}
\end{equation}
We are identifying $q^{-1}(U_{ij})$ with $U_{ij} \times {\rm GL}(n,\mathbb{Z}) \times T^n $ through the inclusion $U_{ij} \to U_i$ and $h_{ij} : U_{ij} \times {\rm GL}(n,\mathbb{Z}) \times T^n \to U_{ij} \times {\rm GL}(n,\mathbb{Z}) \times T^n$ are the transition functions for $\tilde{E} \to M$. Recall that these have the form
\begin{equation*}
h_{ij}(u,g,v) = (u , g_{ij}(u)g , t_{ij}(u,g)v ).
\end{equation*}
In addition if we let $pr_{ij}$ be the projection $U_{ij} \times {\rm GL}(n,\mathbb{Z}) \times T^n \to T^n$ then (\ref{rij}) can be rewritten as 
\begin{equation*}
r_{ij} = h^*_{ij} \beta_{ij} -  \beta_{ij}
\end{equation*}
where $\beta_{ij} = pr^*_{ij} \tilde{\alpha}$. 

Next we want to find functions $s_{ij} : U_{ij} \times {\rm GL}(n,\mathbb{Z}) \times T^n \to \mathbb{Z}$ such that $r_{ij} = h^*_{ij} (\beta_{ij} ) - \beta_{ij} = \delta_S (s_{ij})$. The function $t_{ij}$ acts on $T^n$ by translation: $t_{ij} v = v + t_{ij}$ for $v \in T^n$. For $t \in I = [0,1]$ let $t\tilde{t}_{ij}$ be the translation $t\tilde{t}_{ij} (v) = v + t \tilde{t}_{ij} \, ({\rm mod} \, \Lambda)$. This gives a homotopy between $t_{ij}$ and the identity and one sees easily that $h^*_{ij} (\beta_{ij} ) - \beta_{ij} = \delta_S (s_{ij})$ where $s_{ij}$ is defined by
\begin{equation*}
s_{ij}( u , g , v ) = \tilde{\alpha} ( \gamma_{ij}^{u,g,v}(t) )
\end{equation*}
and where $\gamma_{ij}^{u,g,v}(t)$ is the path $\gamma_{ij}^{u,g,v}(t) : I \to T^n$ given by $t \mapsto v + t \tilde{t}_{ij}(u,g)$.\\

We let $y = \{ s_{ij} \} \in K^{1,0}(\tilde{\pi})$ and set $z = \delta_C y = \{ u_{ijk} \}$. To write $u_{ijk}$ in local coordinates we use inclusions $U_{ijk} \to U_{ij} \to U_i$ so that
\begin{eqnarray*}
u_{ijk}(u,g,v) &=& s_{ij}(u,g,v) + s_{jk}(u, g_{ij}(u)g , (t_{ij}(u,g))v) \\
&& + s_{ki}(u, g_{ik}(u)g , (t_{ik}(u,g))v).
\end{eqnarray*}
It follows that
\begin{equation*}
u_{ijk}( u , g , v ) = \tilde{\alpha} ( \mu_{ijk}^{u,g,v}(t) )
\end{equation*}
where $\mu_{ijk}^{u,g,v}(t)$ is the path $\mu_{ijk}^{u,g,v}(t) : I \to U_{ijk} \times T^n$ given by $t \mapsto v + t c_{ijk}(u,g))$. Now since $c_{ijk}$ is integral this path closes and we can think of $\tilde{\alpha} (\mu_{ijk}^{u,g,v}(t))$ as the pairing of $\alpha \in \Lambda^* = H^1(T^n,\mathbb{Z})$ with the twisted Chern class $c = \{ c_{ijk} \} \in H^2(M , \Lambda_\rho )$.\\

Now let $f \in E_2^{p,q}(\pi)$. Let $f_1 \in K^{p,q}(\tilde{\pi})$ be an invariant representative for $f$. We can write $f_1$ as follows:
\begin{equation*}
f_1 = \sum_{|I| = q} f_I e^I.
\end{equation*}
This expression requires some explanation. The sum is over ordered subsets of $\{ 1 , 2 , \dots , n \}$ of size $q$. If $I = (i_1 , i_2 , \dots , i_q )$ then $e^I = e^{i_1} \wedge e^{i_2} \wedge \dots e^{i_q}$ where $\{ e^i \}$ is a basis for $\Lambda^*$. We view elements of $\Lambda^*$ as elements of $E_1^{0,1}(\tilde{\pi}) = C^0( p^{-1}\mathcal{U} , \Lambda^*_\rho) = \prod_i {\rm Map}({\rm GL}(n,\mathbb{Z}) , \Lambda^*)$ corresponding to constant maps ${\rm GL}(n,\mathbb{Z}) \to \Lambda^*$. The $f_I$ are elements of $E_1^{p,0}(\tilde{\pi}) = C^p( p^{-1}\mathcal{U} , \mathbb{Z})$.\\

For $g \in {\rm GL}(n,\mathbb{Z})$ we may write $g e^i = g^i_j e^j$ for some integer coefficients $g^i_j$ defining $g$. The action of $g$ on $\wedge^q \Lambda^*$ can similarly be written $g e^I = g^I_J e^J$ for some coefficients $g^I_J$ depending on the $g^i_j$. Let $R_g$ denote the right action of $g \in {\rm GL}(n,\mathbb{Z})$ on $\tilde{M}$ and $\tilde{R}_g$ the right action of $g$ on $\tilde{E}$. Then by construction we have
\begin{equation*}
\tilde{R}_g^* e^I = g^I_J e^I
\end{equation*}
and invariance of $f_1$ under ${\rm GL}(n,\mathbb{Z})$ translates to
\begin{equation}\label{tran}
(R_g^* f_I) g^I_J = f_J.
\end{equation}

We can represent $f_1$ in $K^{p,q}(\tilde{\pi})$ as
\begin{equation*}
f_1 = \sum_{|I| = q} f_I \tilde{e}^I
\end{equation*}
where for $I = (i_i , i_2 , \dots , i_q)$ we have $\tilde{e}^I = \tilde{e}^{i_1} \wedge \tilde{e}^{i_2} \wedge \dots \wedge \tilde{e}^{i_q}$ and $\tilde{e}^i$ is a lift of $e^i$ corresponding to a lift of $\Lambda^* = H^1(T^n,\mathbb{Z})$ to singular cochains $S^1(T^n,\mathbb{Z})$.\\

Using the fact that $d_1(f_1) = 0$, where $d_1$ is the differential $d_1 : K^{p,q}(\tilde{\pi}) \to K^{p+1,q}(\tilde{\pi})$ one finds that $\delta_C f_I = 0$. For the $\tilde{e}^i$ we have elements $x^i,y^i,z^i$ related by
\begin{equation*}\xymatrix{
0 & & \\
\tilde{e}^i \ar[u]^{-\delta_S} \ar[r]^{\delta_C} & x^i & \\
& y^i \ar[r]^{\delta_C} \ar[u]^{\delta_S} & z^i
}
\end{equation*}
From our earlier results it follows that
\begin{eqnarray*}
\tilde{R}^*_g \tilde{e}^i &=& g^i_j \tilde{e}^j \\
\tilde{R}^*_g x^i &=& g^i_j x^j \\
\tilde{R}^*_g y^i &=& g^i_j y^j \\
\tilde{R}^*_g z^i &=& g^i_j z^j
\end{eqnarray*}
and $z^a = \{ s^a_{ijk} \}$ where $s^a_{ijk}(u,g,v) = e^a( \mu_{ijk}^{u,g,v})$ with $\mu_{ijk}^{u,g,v}$ a path joining $v$ to $v + c_{ijk}(u,g)$. For $I = (i_1 , i_2 , \dots , i_q)$ define $x^I$ by
\begin{equation*}
x^I = \sum_{j = 1}^q (-1)^{j-1} \tilde{e}^{i_1} \wedge \tilde{e}^{i_2} \wedge \dots \wedge x^{i_j} \wedge \dots \wedge \tilde{e}^{i_q}
\end{equation*}
and similarly define $y^I,z^I$. Then $\tilde{e}^I,x^I,y^I,z^I$ are related as follows:
\begin{equation*}\xymatrix{
0 & & \\
\tilde{e}^I \ar[u]^{-\delta_S} \ar[r]^{\delta_C} & x^I & \\
& y^I \ar[r]^{\delta_C} \ar[u]^{\delta_S} & z^I
}
\end{equation*}
and it follows that we have corresponding relations
\begin{equation*}\xymatrix{
0 & & \\
f_I \tilde{e}^I \ar[u]^{(-1)^p\delta_S} \ar[r]^{\delta_C} & f_I x^I & \\
& f_I y^I \ar[r]^{\delta_C} \ar[u]^{(-1)^{p-1}\delta_S} & f_I z^I
}
\end{equation*}
By construction the terms in this diagram are invariant so $d_2(f)$ is represented in $E^{p+2,q-1}(\tilde{\pi})$ by $f_I z^I$ which can be identified with $c \smallsmile f$ where $c$ acts by a combination of cup product and contraction.
\end{proof}
\end{prop}

A slight generalization of this result is to consider the Leray-Serre spectral sequence for an affine torus bundle $E \to M$ twisted by local coefficients. If $L$ is a system of local coefficients on $M$ with coeficient group a free $\mathbb{Z}$-module then the $E_2$-term for the spectral sequence has form $E_2^{p,q} = H^p(M , L \otimes \wedge^q \Lambda^*_\rho)$. The differential $d_2 : E_2^{p,q} \to E_2^{p+2,q-1}$ is again given by the product with the twisted Chern class followed by a contraction. The above proof carries over to this case with one small modification. The elements $f_I$ instead of being valued in $C^p(p^{-1}\mathcal{U} , \mathbb{Z})$ are taken to have values in $C^p(p^{-1}\mathcal{U} , L)$ and the transformation law (\ref{tran}) is modified accordingly.\\

{\bf Gysin sequence with local coefficients:} consider the special case where $n=1$, that $E \to M$ is a circle bundle. In this case the spectral sequence only has two rows so by a standard argument we obtain a long exact sequence:
\begin{equation*}\xymatrix{
\dots \ar[r] & H^i(M,\mathbb{Z}) \ar[r]^-{\pi^*} & H^i(E,\mathbb{Z}) \ar[r]^-{\pi_*} & H^{i-1}(M , \Lambda^*_\rho) \ar[r]^-{c \smallsmile } & H^{i+1}(M, \mathbb{Z}) \ar[r] & \dots
}
\end{equation*}
This is the Gysin sequence for non-oriented circle bundles. There is a similar sequence for sphere bundles \cite{las} but we only need the circle bundle case.


\section{Topological T-duality of circle bundles}

We are now ready to describe a version of topological T-duality which applies to pairs $(E,H)$ where $\pi : E \to M$ is a circle bundle over $M$ and $H \in H^3(E,\mathbb{Z})$ is an integral degree $3$ cohomology class on $E$. There is a straightforward notion of isomorphism of such pairs namely $(E,H)$ and $(E',H')$ are said to be isomorphic if there is a bundle isomorphism $\phi : E \to E'$ covering the identity such that $H = \phi^*(H')$.

\begin{defn}
Two pairs $(E,H)$, $(\hat{E},\hat{H})$ consisting of circle bundles $\pi : E \to M$, $\hat{\pi} : \hat{E} \to M$ and fluxes $H \in H^3(F,\mathbb{Z})$, $\hat{H} \in H^3(\hat{F},\mathbb{Z})$ are said to be {\em T-dual} if
\begin{itemize}
\item{The first Stiefel-Whitney classes agree: $\xi = w_1(E) = w_1(\hat{E})$}
\item{The twisted Chern classes and fluxes are related by $\pi_*(H) = \hat{F} = c_1(\hat{E})$, $\hat{\pi}_*(\hat{H}) = F = c_1(E)$}
\item{Let $p,\hat{p}$ be the projections $p : E \times_M \hat{E} \to E$ and $\hat{p} : E \times_M \hat{E} \to \hat{E}$. Then $p^*(H) = \hat{p}^*(\hat{H})$.}
\end{itemize}
\end{defn}

Denote by $F$ the fibre product $F = E \times_M \hat{E}$. We call $F$ the {\em correspondence space}. It is helpful to visualize the relation between the spaces as follows
\begin{equation*}\xymatrix{
& F \ar[dl]_p \ar[dr]^{\hat{p}} \ar[dd]^q & \\
E \ar[dr]_\pi & & \hat{E} \ar[dl]^{\hat{\pi}} \\
& M &
}
\end{equation*}
Where we denote by $q$ the projection $q = p \pi = \hat{p} \hat{\pi} : F \to M$. In this diagram there are $5$ different torus bundles corresponding to $\pi,\hat{\pi},p,\hat{p}$ and $q$. We will need to refer to the associated Leray-Serre spectral sequences. To avoid confusion we label the different spectral sequences by the associated map. Thus for example the spectral sequence associated to $\pi : E \to M$ will be denoted $E_r^{p,q}(\pi)$ and similarly for the other maps.

\begin{prop}\label{tdualexist}
For any circle bundle $\pi : E \to M$ with flux $H \in H^3(E,\mathbb{Z})$ there exists a T-dual $(\hat{E},\hat{H})$ which is unique up to isomorphism. Furthermore the T-duality relation is symmetric so the T-dual of $(\hat{E},\hat{H})$ is $(E,H)$.
\begin{proof}
Let $\xi = w_1(E)$ and $F = c_1(E)$. Define $\hat{F} = \pi_*(H) \in H^2(M,\mathbb{Z}_\xi)$. Then $\hat{F}$ is the twisted Chern class for a circle bundle $\hat{p} : \hat{E} \to M$ unique up to isomorphism and by construction we have $w_1(\hat{E}) = \xi = w_1(E)$. From the Gysin sequence for $E$ we have $\hat{F} \smallsmile F = 0$. Feeding this into the Gysin sequence for $\hat{E}$ we find that there exists $H' \in H^3(\hat{E},\mathbb{Z})$ such that $\hat{\pi}_*(H') = F$. We would like to say that $(\hat{E},H')$ is the T-dual of $(E,H)$ but there are two problems. First $H'$ is not uniquely determined by the condition $\hat{\pi}_*(H') = F$, second  we need not have $p^*(H) = \hat{p}^*(H')$. 

\begin{lemp}\label{exist}
There exists $a \in H^3(M,\mathbb{Z})$ such that $p^*(H) - \hat{p}^*(H') = q^*(a)$.
\end{lemp}

We will prove Lemma \ref{exist} below. Given this result we define $\hat{H} = H' + a$. Then we have that $(E,H)$ and $(\hat{E},\hat{H})$ are T-dual. For uniqueness suppose $(\hat{E},\hat{H} + \mu)$ is also T-dual to $(E,H)$. In this case we must have $\hat{\pi}_* \mu = 0$ and $\hat{p}^* \mu = 0$.

\begin{lemp}\label{unique}\label{lem2}
Let $\mu \in H^3(\hat{E},\mathbb{Z})$ be such that $\hat{\pi}_* \mu = 0$ and $\hat{p}^* \mu = 0$. Then there exists $\alpha \in H^1(M,\mathbb{Z}_\xi)$ such that $\mu = \hat{\pi}^*( \alpha \smallsmile F)$.
\end{lemp}

From this Lemma it follows that $\hat{H} + \mu = \hat{H} + \hat{\pi}^*( \alpha \smallsmile F)$. Uniqueness up to bundle isomorphism will follow from:
\begin{lemp}\label{gauge}\label{buniso}
Let $\pi : E \to M$ be a circle bundle with fibre orientation class $\xi \in H^1(M,\mathbb{Z}_2)$. For any $\alpha \in H^1(M,\mathbb{Z}_\xi)$ there exists a bundle isomorphism $\phi : E \to E$ such that for any $x \in H^k(E,\mathbb{Z})$ we have
\begin{equation*}
\phi^* x = x + \pi^*( \alpha \smallsmile \pi_* x ).
\end{equation*}
\end{lemp}
Thus we can find a bundle isomorphism $\phi : \hat{E} \to \hat{E}$ corresponding to $\alpha \in H^1(M,\mathbb{Z}_2)$ such that
\begin{equation*}
\phi^* \hat{H} = \hat{H} + \hat{\pi}^*( \alpha \smallsmile \hat{\pi}_* \hat{H}) = \hat{H} + \hat{\pi}^*( \alpha \smallsmile F).
\end{equation*}
This establishes existence and uniqueness.
\end{proof}
\end{prop}

\begin{proof}[Proof of Lemma \ref{exist}]
Recall that associated to the fibration $F \to M$ is a filtration on the cohomology of $F$. In the case at hand we have
\begin{equation*}
p^* \pi^* (H^3(M,\mathbb{Z})) = F^{3,3} \subseteq F^{2,3} \subseteq F^{1,3} = H^3(F,\mathbb{Z})
\end{equation*}
and $F^{p,3}/F^{p+1,3}$ is isomorphic to $E_\infty^{p,3-p}(q)$ in the Leray-Serre spectral sequence. We must show that the class $d = p^*(H) - \hat{p}^*( \hat{H}) \in H^3(F,\mathbb{Z})$ lies in the subspace $F^{3,3}$. Note that we have $q_* d = 0$, which is precisely the condition that $d$ lies in the subspace $F^{2,3}$. Let $d'$ be the image of $d$ in $F^{2,3}/F^{3,3} = E_\infty^{2,1}(q)$. It remains to show that $d' = 0$. Associated to the morphism of bundles
\begin{equation*}\xymatrix{
F \ar[r]^p \ar[d]^q & E \ar[d]^\pi \\
M \ar@{=}[r] & M
}
\end{equation*}
is a morphism of spectral sequences $E_2^{p,q}(\pi) \to E_r^{p,q}(q)$. The image of $H$ in $E_2^{2,1}(\pi) = H^2(M,\mathbb{Z}_\xi)$ is given by $\pi_* H = \hat{F}$. A similar statement holds for $\hat{H}$. Therefore we find that $d' \in E_\infty^{2,1}(q)$ is represented by $d'' = (\hat{F},-F) \in E_2^{2,1}(q) = H^2(M,\mathbb{Z}_\xi) \oplus H^2(M,\mathbb{Z}_\xi)$. Clearly $d''$ lies in the image of $d_2 : H^0(M,\mathbb{Z}) = E_2^{0,2}(q) \to E_2^{2,1}(q)$. Therefore $d' = 0$ and the result follows.
\end{proof}

\begin{proof}[Proof of Lemma \ref{unique}]
From the Gysin sequence for $\hat{E} \to M$ we have that $\mu = \hat{\pi}^* \nu$ for some $\nu \in H^3(M,\mathbb{Z})$. Then $\nu \in H^3(M,\mathbb{Z}) = E_2^{3,0}(\hat{\pi})$ represents $\mu$ in the Leray-Serre spectral sequence for $\hat{E} \to M$. Now consider the morphism of fibre bundles
\begin{equation*}\xymatrix{
F \ar[r]^{\hat{p}} \ar[d]^q & \hat{E} \ar[d]^{\hat{\pi}} \\
M \ar@{=}[r] & M
}
\end{equation*}
This induces a morphism of spectral sequences $E^{p,q}_r(\hat{\pi}) \to E_r^{p,q}(q)$. The element $\nu \in E_2^{3,0}(\hat{\pi})$ representing $\mu$ gets mapped to $\nu \in E_2^{3,0}(q) = H^3(M,\mathbb{Z})$ representing $\hat{p}^* \mu$. Now the condition $\hat{p}^* \mu = 0$ tells us that $\nu$ lies in the image of the differentials $d_2,d_3$. Consider now a second morphism of fibre bundles
\begin{equation*}\xymatrix{
F \ar@{=}[r] \ar[d]^{\hat{p}} & F \ar[d]^q \\
\hat{E} \ar[r]^{\hat{\pi}} & M
}
\end{equation*}
inducing a morphism $E_r^{p,q}(q) \to E_r^{p,q}(\hat{p})$. Under this morphism $\nu \in E_2^{3,0}(q)$ representing $\hat{p}^* \mu = 0$ gets sent to $\hat{\pi}^* \nu = \mu \in E_2^{3,0}(\hat{p})$. Under the morphism $E_r^{p,q}(q) \to E_r^{p,q}(\hat{p})$ the differential $d_3$ gets sent to zero and it follows that $\mu = d_2 (\hat{\pi}^* \alpha)$ for some $\alpha \in E_2^{1,1}(q) = H^1(M,\mathbb{Z}_\xi)$. But this is exactly the condition $\mu = \hat{\pi}^*\alpha \smallsmile \hat{\pi}^* F = \hat{\pi}^*( \alpha \smallsmile F)$.
\end{proof}

\begin{proof}[Proof of Lemma \ref{gauge}]
We follow the proof of \cite[Theorem 2.16]{bunksch} and adapt it to general circle bundles. The class $\xi \in H^1(M,\mathbb{Z}_2)$ determines a double cover $M_2 \to M$. Let $\pi_2 : E_2 \to M_2$ denote the pull-back of $E$ over $M_2$. Then as we have seen $E_2$ admits the structure of a principal circle bundle and there is an involution $\sigma : E_2 \to E_2$ which covers the involution on $M_2$ and makes $E_2$ into a principal ${\rm O}(2)$-bundle over $M$.\\

Consider now the following commutative diagram:
\begin{equation*}\xymatrix{
E_2 \times {\rm U}(1) \ar[r]^-m \ar[d]^{pr} & E_2 \ar[d]^{\pi_2} \\
E_2 \ar[r]^{\pi_2} & M_2
}
\end{equation*}
where $pr$ is the projection to the first factor and $m$ is the principal bundle right multiplication. The involution $\sigma$ acts on all spaces in this diagram, where we let $\sigma$ act on $E_2 \times {\rm U}(1)$ by $\sigma( e , g ) = (\sigma e , g^{-1})$ for $e \in E_2$, $g \in {\rm U}(1)$. Let $\tilde{E} = \left( E_2 \times {\rm U}(1) \right)/\sigma$. The diagram then descends to the following:
\begin{equation*}\xymatrix{
\tilde{E} \ar[r]^-m \ar[d]^{pr} & E \ar[d]^{\pi} \\
E \ar[r]^{\pi} & M
}
\end{equation*}
Note that since $E_2 \times {\rm U}(1)$ is a trivial principal bundle we have that $pr : \tilde{E} \to E$ has trivial twisted Chern class. Indeed $\tilde{E}$ admits a section $s : E \to \tilde{E}$ given by $s(e) = (e',1)$ where $e' \in E_2$ is any lift of $e \in E$ and $1 \in {\rm U}(1)$ is the identity. It follows that the Gysin sequence for $\tilde{E} \to E$ splits giving
\begin{equation}\label{isom}
H^k(\tilde{E},\mathbb{Z}) = H^k(E,\mathbb{Z}) \oplus H^{k-1}(E,\mathbb{Z}_\xi).
\end{equation}
The isomorphism (\ref{isom}) is given by sending $a \in H^k(\tilde{E},\mathbb{Z})$ to $(s^*a , pr_* a)$.

Now for $x \in H^k(E,\mathbb{Z})$ we would like to determine $m^* x \in H^k(\tilde{E},\mathbb{Z})$. We use the diagram (which is commutative except for $s^*$)
\begin{equation*}\xymatrix{
H^k(M,\mathbb{Z}) \ar[r]^{\pi^*} \ar[d]^{\pi^*} & H^k(E,\mathbb{Z}) \ar[r]^-{\pi_*} \ar[d]^{m^*} & H^{k-1}(M,\mathbb{Z}_\xi) \ar[d]^{\pi^*} \\
H^k(E,\mathbb{Z}) \ar@<1ex>[r]^{pr^*} & H^k(\tilde{E},\mathbb{Z}) \ar@<1ex>[l]^{s^*} \ar[r]^{pr_*} & H^{k-1}(E,\mathbb{Z}_\xi) 
}
\end{equation*}
to see that $pr_* m^* x = \pi^* \pi_* x$. One easily checks that $m s = {\rm id} : E \to E$ so that $s^* m^* x = x$. Thus $m^*x$ corresponds to $(x,\pi^* \pi_* x)$ under (\ref{isom}).\\

Now consider an equivariant gauge transform, that is a map $g : M_2 \to {\rm U}(1)$ such that $g(\sigma p) = g(p)^{-1}$, for $p \in M_2$. The map $g$ determines a principal bundle automorphism of $E_2$ which can be described as the composition
\begin{equation*}\xymatrix{
E_2 \ar[r]^-{({\rm id},\pi_2)} & E_2 \times M_2 \ar[r]^-{({\rm id},g)} & E_2 \times {\rm U}(1) \ar[r]^-m & E_2
}
\end{equation*}
and this descends to a bundle automorphism $\phi : E \to E$ given by
\begin{equation*}\xymatrix{
E \ar[r]^-{({\rm id},\pi)} & \left( E_2 \times M_2 \right)/\sigma \ar[r] & \tilde{E} \ar[r]^-m & E
}
\end{equation*}
For a class $x \in H^k(E,\mathbb{Z})$ we would like to determine $\phi^* x$. To do this we work out what happens for each of the three steps defining $\phi$. We have already described $m^*x$ so the next step is to understand pull-back under $\left( E_2 \times M_2 \right)/\sigma \to \tilde{E}$. For this we work with $\mathbb{Z}_2$-equivariant cohomology in the sense of Eilenberg \cite{eil} and the corresponding map $({\rm id} , g ) : E_2 \times M_2 \to E_2 \times {\rm U}(1)$. Now $g$ defines an element of $H_\sigma^0(M_2 , \mathcal{C}^\infty( {\rm U}(1) ) )$, where the subscript denotes $\sigma$-equivariance. But this is isomorphic to $H^1(M,\mathbb{Z}_\xi)$. For any $\alpha \in H^1(M,\mathbb{Z}_\xi)$ choose a corresponding equivariant map $g : M_2 \to {\rm U}(1)$ representing the class $\alpha$. Let $\mathbb{Z}_\sigma$ denote the $\mathbb{Z}_2 = \langle \sigma \rangle$-module which as an abelian group is $\mathbb{Z}$ but where $\sigma$ acts by $-1$. Then there is a canonical class $[ dt ] \in H^1_\sigma ( {\rm U}(1) , \mathbb{Z}_\sigma)$. A pair $(a,b) \in H^k(E,\mathbb{Z}) \oplus H^{k-1}(E,\mathbb{Z}_\xi)$ representing an element of $H^k(\tilde{E},\mathbb{Z})$ corresponds under $H^k(\tilde{E},\mathbb{Z}) \simeq H^k_\sigma(E_2 \times {\rm U}(1),\mathbb{Z})$ to
\begin{equation*}
a + [dt] \smallsmile b
\end{equation*}
where we think of $a$ as an element of $H^k_\sigma(E_2 \times {\rm U}(1) , \mathbb{Z})$ and $b$ an element of $H_\sigma^{k-1}(E_2 \times {\rm U}(1) , \mathbb{Z}_\sigma )$.

Under pull-back by the map $({\rm id},g) : E_2 \times M_2 \to E_2 \times {\rm U}(1)$ this becomes $a + \alpha \smallsmile b$ where we think of $\alpha$ as an element of $H^1_\sigma (M_2 , \mathbb{Z}_\sigma )$. Putting it all together we thus get
\begin{equation*}
\phi^*(x) = x + \pi^*(\alpha) \smallsmile \pi^* \pi_* x = x + \pi^*( \alpha \smallsmile \pi_* x )
\end{equation*}
as required.
\end{proof}


\section{Differential form approach to T-duality}

We will now present an alternative approach to T-duality via differential forms. This approach has the advantage that the construction of the $T$-dual is more transparent but has the drawback that we work with real cohomology so all torsion information is lost. On the other hand the differential form approach will lead directly to an isomorphism between twisted cohomology on $T$-dual spaces and is relevant when we consider the twisted Chern character and Courant algebroids.


\subsection{Twisted connections}\label{twcon}

Let $\pi : E \to M$ be a circle bundle over $M$, where $M$ and $E$ are now assumed to be smooth manifolds. If $E$ is a principal bundle we can choose a connection on $E$ to relate invariant forms on $E$ with forms on $M$. We would like to employ similar techniques in the general case of circle bundles.\\

We have seen that every circle bundle has the form $E = \tilde{E} \times_{{\rm O}(2)} S^1$ for some principal ${\rm O}(2)$-bundle. We can then view $E$ as a sort of {\em twisted principal circle bundle}. By this we mean that locally $E$ admits the structure of a principal ${\rm U}(1)$-bundle but that the local actions are related on overlaps by an element of $\mathbb{Z}_2 = {\rm Aut}(S^1)$. This local action depends on the identification of $E$ as the circle bundle associated to a principal ${\rm O}(2)$-bundle so is not canonical, however any two local actions are related by a fibre bundle automorphism of $E$ so that any two local actions are essentially the same.\\

Given a local action on $E$ the vertical bundle ${\rm Ker} (\pi_* : TE \to TM)$ can be viewed as a flat line bundle corresponding to an element $\xi \in H^1(M , \mathbb{Z}_2) = {\rm Hom}(\pi_1(M) , \mathbb{Z}_2)$. More precisely there is a flat line bundle $V$ on $M$ such that $\pi^*(V)$ can be identified with the vertical bundle.

We can find an open cover $\{ U_i \}$ of $M$ and vertical vector fields $X_i \in \Gamma(U_i,V)$ which generate the local action. On  the overlaps we have $X_i = \xi_{ij} X_j$ where $\xi_{ij}$ are valued in $\mathbb{Z}_2$ and is a representative for $\xi = w_1(V)$ in \v{C}ech cohomology. Put another way there is a double cover $M_2 \to M$ such that sections of $M_2 \to M$ correspond precisely to vector fields which restrict to $\pm X_i$ over $U_i$. Let $\mathcal{V}$ denote this sheaf.

\begin{defn}
A differential form $\omega \in \Omega^*(E)$ is {\em invariant} if for every open subset $U \subseteq M$ and section $X \in \mathcal{V}(U)$ we have $\mathcal{L}_X \omega|_U = 0$.
\end{defn}
Alternatively the double cover $p : \tilde{E} \to E$ is a principal ${\rm O}(2)$-bundle and a form $\omega \in \Omega^*(E)$ is invariant if $p^*\omega$ is invariant in the usual sense. Note that if $W$ is a flat vector bundle on $M$ then it makes sense to speak of invariant differential forms on $E$ with values in $\pi^*(W)$. There is also a notion of an invariant vector field which is defined similarly to invariant differential forms.

Next consider extending the notion of a connection to such bundles. 
\begin{defn}
A {\em twisted connection} on $E$ is an invariant form $A \in \Omega^1(E,\pi^*(V))$ such that restricted to the vertical bundle $A$ is the identity $\pi^*(V) \to \pi^*(V)$.
\end{defn}
A twisted connection can be viewed as a choice of invariant horizontal complement to the vertical bundle of $E$ or simply as an ordinary connection for $\tilde{E}$. We thus have:
\begin{prop}
Twisted connections always exist and form an affine space modeled on $\Omega^1(M,V)$.
\end{prop}

\begin{defn}
Given a twisted connection $A \in \Omega^1(E,\pi^*(V))$ there exists a unique form $F \in \Omega^2(M,V)$ such that 
\begin{equation*}
d_{\nabla} A = \pi^*(F)
\end{equation*}
where $\nabla$ is the flat connection on $V$. We call $F$ the {\em curvature} of $A$. Clearly we have $d_\nabla F = 0$ and thus the curvature determines a cohomology class with local coefficients $[F] \in H^2(M , V)$. It is not hard to see that $[F]$ is the image of the twisted Chern class $c_1 \in H^2(M , \Lambda)$ under the coefficient homomorphism $\Lambda_\xi \to \Lambda_\xi \otimes_{\mathbb{Z}} \mathbb{R} = V$.
\end{defn}


\subsection{Invariant cohomology}

Let $\pi : E \to M$ be a circle bundle an equip $E$ with the structure of a locally principal $S^1$-bundle. We then have a notion of invariant forms $\Omega^*_{\rm inv}(E) \subseteq \Omega^*(E)$ and it is clear that the subspace of invariant forms is preserved by the exterior derivative. Therefore we may consider the cohomology of $d$ restricted to invariant forms which we denote by $H_{\rm inv}^*(E)$. The inclusion $i : \Omega^*_{\rm inv}(E) \to \Omega^*(E)$ induces a morphism $i_* : H_{\rm inv}^*(E) \to H^*(E)$.

\begin{prop}\label{iso}
The morphism $i_* : H_{\rm inv}^*(E) \to H^*(E)$ is an isomorphism.
\begin{proof}
Let $\xi \in H^1(M , \mathbb{Z}_2)$ represent the monodromy of the vertical bundle. There is a corresponding principal $\mathbb{Z}_2$-bundle $m: M_2 \to M$ over $M$ with the property that $E_2 = m^*(E)$ is a principal $\mathbb{Z}_2 \ltimes {\rm U}(1) = {\rm O}(2)$-bundle over $M$.\\

It is clear that invariant forms on $E$ correspond precisely to ${\rm O}(2)$-invariant forms on $E_2$. Thus $H^*_{\rm inv}(E)$ is the cohomology of the complex $\left( \Omega^*(E_2)^{{\rm O}(2)} , d \right)$ of ${\rm O}(2)$-invariant differential forms on $E_2$ with respect to the exterior derivative. Next we argue this is the same as the $\mathbb{Z}_2$-invariant subspace of $H^*(E_2)$. In fact, since ${\rm O}(2)$ is compact we can average over ${\rm O}(2)$ giving a projection operator $P : \Omega^*(E_2) \to \Omega^*(E_2)$ which commutes with $d$. It follows that the cohomology of the complex $\left( \Omega^*(E_2)^{{\rm O}(2)} , d \right)$ is precisely the subspace of $H^*(E_2)$ fixed by ${\rm O}(2)$. However the action of ${\rm O}(2)$ must preserve the image of the integral cohomology $H^*(E_2)$ and so the identity component ${\rm SO}(2)$ of ${\rm O}(2)$ acts trivially, so the action factors through to the action of $\mathbb{Z}_2$.\\

So far we have shown that $H^*_{\rm inv}(E)$ is isomorphic the $H^*(E_2 )^{\mathbb{Z}_2}$, the subspace of $H^*(E_2)$ invariant under the $\mathbb{Z}_2$-action. However since $\mathbb{Z}_2$ is finite we can average over the fibres of $E_2 \to E$ and this gives an isomorphism $H^*(E_2)^{\mathbb{Z}_2} = H^*(E)$.
\end{proof}
\end{prop}

Let $H \in \Omega^3(E)$ represent a class in $H^3(E)$. The $H$-twisted cohomology is the cohomology of the complex $\left( \Omega^*(E) , d_H \right)$ where the twisted differential $d_H$ is defined as $d_H \omega = d \omega + H \wedge \omega$. Note that there is a $\mathbb{Z}_2$-grading given by even and odd forms. Let $H^*(E,H)$ denote $H$-twisted cohomology. Twisted cohomology depends up to isomorphism only on the cohomology class of $H$ and thus we may take a representative form $H$ which is invariant. Upon making such a choice the twisted differential $d_H$ preserves the subspace of invariant forms and thus we can take cohomology of the complex $\left( \Omega^*_{\rm inv}(E) , d_H \right)$. Let $H^*_{\rm inv}(E,H)$ denote the cohomology of this complex. There is an evident map $i_* : H^*_{\rm inv}(E,H) \to H^*(E,H)$.

\begin{prop}\label{inviso}
The map $i_* : H^*_{\rm inv}(E,H) \to H^*(E,H)$ is an isomorphism.
\begin{proof}
We would like to repeat the proof of Proposition \ref{iso}, however there is an additional technical detail to address. Let $M_2 \to M$ and $m : E_2 \to E$ be the double covers as in \ref{iso}. We similarly find that $H^*_{\rm inv}(E,H)$ is isomorphic to the ${\rm O}(2)$-invariant subspace of $H^*(E_2,m^*H)$. What is not so obvious is that the ${\rm O}(2)$-action once again factors through to a $\mathbb{Z}_2$-action. To see this note that the twisted differential $d_{m^*H}$ on $\Omega^*(E_2)$ preserves the filtration $\Omega^{* \ge \, j}(E_2)$ of forms of degree greater than or equal to some integer. There is an associated spectral sequence which computes the twisted cohomology. The $E_1$ page in this sequence is just the ordinary cohomology $H^*(E_2)$. 

The ${\rm O}(2)$-action on forms preserves the filtration and thus induces an action on each page the spectral sequence. However we can see that the ${\rm SO}(2)$-subgroup must act trivially on the $E_1$ page since it preserves the lattice of integral cohomology classes. 

The filtration on forms induces a filtration on the twisted cohomology $H^*(E_2,m^*H)$ and the spectral sequence converges to the associated graded space. The Lie algebra of ${\rm O}(2)$ acts trivially on the associated graded space, so it must act on $H^*(E_2,m^*H)$ by nilpotent elements. But since this Lie algebra action must integrate to an action of ${\rm SO}(2)$, we are looking for a nilpotent map $N : H^*(E_2,m^*H) \to H^*(E_2,m^*H)$ such that $e^{2 \pi N} = 1$. This can only happen if $N=0$.\\

We have thus shown that $H^*_{\rm inv}(E,H)$ is isomorphic to the $\mathbb{Z}_2$-invariant subspace of $H^*(E_2,m^*H)$, but by a similar argument to \ref{iso} this is precisely $H^*(E,H)$.
\end{proof}
\end{prop}


\subsection{T-duality}\label{tdco}
Let $\pi : E \to M$ be a circle bundle as before and now introduce a class $H \in H^3(E,\mathbb{Z})$. Suppose that $H$ is represented by a differential form in $\Omega^3(E)$ which we also denote by $H$. By Proposition \ref{iso} we also assume the representative is invariant. Next let $A \in \Omega^1(E,\pi^*(V))$ be a twisted connection with curvature $F \in \Omega^2(F,V)$. We may now decompose $H$ as follows:
\begin{equation}\label{hdec}
H = \pi^*(H_3) + A \wedge \pi^* \hat{F}
\end{equation}
where $H_3 \in \Omega^3(M)$ and $\hat{F} \in \Omega^2(M,V^*)$. The wedge product in (\ref{hdec}) combines the wedge product of forms with the pairing of local systems $V \otimes V = V \otimes V^* \to \mathbb{R}$. That $H$ is closed is equivalent to the pair of equations
\begin{eqnarray*}
dH_3 + F \wedge \hat{F} &=& 0 \\
d_{\nabla} \hat{F} &=& 0.
\end{eqnarray*}
Notice that the second equation implies that $\hat{F}$ determines a cohomology class $[\hat{F}] \in H^2(M,V^*)$ and the first equation implies $[F] \smallsmile [\hat{F}] = 0$. Indeed it is clear that $[\hat{F}]$ is the image in $H^2(M,V^*)$ of the class $\hat{c} = \pi_*(H) \in H^2(M,\Lambda^*)$ under change of coefficients $\Lambda^* \to \Lambda^* \otimes_{\mathbb{Z}} \mathbb{R} = V^*$. As we have seen the class $\hat{c}$ determines a circle bundle $\hat{\pi} : \hat{E} \to M$. It follows that there exists a twisted connection $\hat{A} \in \Omega^1(\hat{E} , \hat{\pi}^* (V^*))$ on $\hat{E}$ such that $\hat{F}$ is the curvature of $\hat{A}$. Now we define $\hat{H} \in \Omega^3(\hat{E})$ by
\begin{equation}\label{hhdec}
\hat{H} = \hat{\pi}^*(H_3) + \hat{A} \wedge \hat{\pi}^* F.
\end{equation}
By construction we find that $d \hat{H} = 0$ and thus $\hat{H}$ represents a class $[\hat{H}] \in H^3(\hat{E})$. Observe also that $[\hat{\pi}_* \hat{H}] = [F]$ and that
\begin{equation*}
p^*H - \hat{p}^* \hat{H} = d ( \hat{p}^*\hat{A} \wedge p^* A )
\end{equation*}
so that $p^* [H] - \hat{p}^* [\hat{H}] = 0 \in H^3(E \times_M \hat{E})$.\\

The construction of the pair $(\hat{E},\hat{H})$ relied on a number of choices, namely a choice of invariant representative for $[H]$ and choice of twisted connections $A,\hat{A}$. If one goes through the above construction allowing for all different possible choices we find that $\hat{H}$ changes at most by a term of the form $\pi^*( e \wedge F )$ where $e$ is a $d_\nabla$-closed element of $\Omega^1(M,V^*)$. Since the class of $e$ in $H^1(M,V^*)$ need not be integral, there is more ambiguity in defining $\hat{H}$ than arises from bundle isomorphisms as determined by Lemma \ref{buniso}. Therefore this construction does not fix the isomorphism class of $\hat{H}$ completely. However we have the following:
\begin{prop}
For any choice of representatives $H,A,\hat{A}$ one can make a change $\hat{A} \mapsto \hat{A} + \hat{\pi}^* e$ for some $d_\nabla$-closed element $e \in \Omega^1(M,V^*)$ so that the class of $\hat{H}$ lies in the image $H^3(\hat{E},\mathbb{Z}) \to H^3(\hat{E},\mathbb{R})$ of integral cohomology in real cohomology. With this additional condition the pair $(\hat{E} , [\hat{H}])$ is unique up to bundle isomorphism. Moreover for each $h \in H^3(E,\mathbb{Z})$ such that $[H]$ is the image of $h$ in real cohomology there exists $\hat{h} \in H^3(\hat{E},\mathbb{Z})$ such that $[\hat{H}]$ is the image of $\hat{h}$ in real cohomology and the pairs $(E,h)$,$(\hat{E},\hat{h})$ are $T$-dual.
\begin{proof}
Let $h \in H^3(E,\mathbb{Z})$ be a class such that the image $h_{\mathbb{R}}$ of $h$ in real cohomology equals $[H]$. Let $\hat{h} \in H^3(\hat{E},\mathbb{Z})$ be such that $(\hat{E},\hat{h})$ is T-dual to $(E,h)$. Let $a = \hat{h}_\mathbb{R} - [\hat{H}] \in H^3(\hat{E},\mathbb{R})$. By construction $\hat{p}^* a = 0$ and $\hat{\pi}_* a = 0$, so by Lemma \ref{lem2} there exists a class $[e] \in H^1(M,\mathbb{R})$ represented by a $d_\nabla$-closed form $e \in \Omega^1(M,V^*)$ such that $a = \hat{\pi}^*( [e] \smallsmile F)$. 

Now replace the twisted connection $\hat{A}$ by $\hat{A} + \hat{\pi}^* e$. We still have $d_\nabla \hat{A} = \hat{\pi}^* F$ and we find that $\hat{H}$ changes to $\hat{H} + \hat{\pi}^* ( e \wedge F)$. The result now follows because $[ \hat{H} + \hat{\pi}^* ( e \wedge F ) ] = [\hat{H} ] + a = \hat{h}_\mathbb{R}$.
\end{proof}
\end{prop}


\subsection{T-duality of twisted cohomology}
In the case of principal circle bundles it is one of the main properties of T-duality that the twisted cohomologies $H^*(E,H)$ and $H^{*-1}(\hat{E},\hat{H})$ are isomorphic. In the case of general circle bundles a similar result applies but needs to be modified to take into account the local system of fibre orientations. Recall that one proves the isomorphism of twisted cohomologies by means of the Hori formula \cite{bem}
\begin{equation}\label{hori}
\phi (\omega) = \hat{p}_* ( e^{-(p^*A \wedge \hat{p}^* \hat{A}) } \, p^* \omega )
\end{equation}
where $\phi$ is a map $\phi : \Omega^*(E) \to \Omega^*(\hat{E})$ that induces an isomorphism on twisted cohomology. In the case of non-oriented circle bundles this formula can still be applied, however the fibre integration $\hat{p}_*$ picks up the sign factor ambiguity due to fibre orientation. To formalize this we introduce a variant of de Rham cohomology twisted by closed $3$-forms and flat bundles. This twisted cohomology will be the target of the twisted Chern character for twisted $K$-theory.\\

Let $X$ be a smooth manifold. Given a closed $3$-form $H \in \Omega^3(X)$ and a flat vector bundle $(V,\nabla)$ we may define the $(\nabla,H)$-twisted de Rham complex. The space of the complex is $\Omega^*(X,V) = \mathcal{C}^\infty(\wedge^* T^*X \otimes V)$, the space of form-valued sections of $V$ and the differential $d_{(\nabla,H)}$ is characterized by the property
\begin{equation*}
d_{(\nabla,H)} (\omega \otimes v) = d \omega \otimes v + H \wedge \omega \otimes v + (-1)^p \omega \wedge \nabla v
\end{equation*}
where $\omega$ is a $p$-form and $v$ a section of $V$. Note that $d^2_{(\nabla , H)} = 0$ so we have a complex. The cohomology of the complex is the $(\nabla,H)$-twisted cohomology denoted $H^*(X,(\nabla,H))$ or $H^*(X,(V,H))$ when the connection $\nabla$ is understood. Note that the twisted cohomology is only $\mathbb{Z}_2$-graded where the grading comes from form parity. Up to isomorphism the twisted cohomology depends only on the cohomology class $H \in H^3(X,\mathbb{R})$ and the monodromy representation corresponding to $\nabla$. If $\rho : \pi_1(X) \to {\rm GL}(n,\mathbb{R})$ is a representation of the fundamental group of $X$ then we also use notation $H^*(X,(\rho,H))$ for the $(\nabla,H)$-twisted cohomology where $\nabla$ is the flat connection associated to $\rho$.\\

Now returning to T-duality suppose we have T-dual pairs $(E,H)$ and $(\hat{E},\hat{H})$ where $H$ and $\hat{H}$ are representative $3$-forms and $A,\hat{A}$ are twisted connections related to $H,\hat{H}$ as in (\ref{hdec}),(\ref{hhdec}) and recall
\begin{equation*}
p^*H - \hat{p}^* \hat{H} = d( \hat{p}^*\hat{A} \wedge p^* A ) = d \mathcal{B}
\end{equation*}
where we have defined $\mathcal{B} = \hat{p}^* \hat{A} \wedge p^* A \in \Omega^2(F)$. We define T-duality maps
\begin{eqnarray*}
&T& : H^*(E,H) \to H^{*-1}(\hat{E},(\xi,\hat{H})) \\
&T_\xi& : H^*(E,(\xi,H)) \to H^{*-1}(\hat{E},\hat{H})
\end{eqnarray*}
by the Hori formula (\ref{hori}):
\begin{eqnarray*}
T &=& \hat{p}_* \circ e^{\mathcal{B}} \circ p^* \\
T_\xi &=& \hat{p}_* \circ e^{\mathcal{B}} \circ p^*.
\end{eqnarray*}
Let $\mathbb{R}_\xi$ be the flat line bundle corresponding to $\xi \in H^1(M,\mathbb{Z}_2)$. The pushforward maps $\hat{p}_* : \Omega^k(F) \to \Omega^{k-1}(\hat{E},\mathbb{R}_\xi)$ and $\hat{p}_* : \Omega^k(F,\mathbb{R}_\xi) \to \Omega^{k-1}(\hat{E})$ are essentially fibre integration.\\

By symmetry of the T-duality relation we can easily define T-duality transformations $\hat{T},\hat{T}_\xi $ in the reverse direction:
\begin{eqnarray*}
\hat{T} &=& p_* \circ e^{-\mathcal{B}} \circ \hat{p}^* \\
\hat{T}_\xi &=& p_* \circ e^{-\mathcal{B}} \circ \hat{p}^*.
\end{eqnarray*}

\begin{prop}\label{tci}
Let $(E,H)$,$(\hat{E},\hat{H})$ be T-dual pairs with $\xi$ the orientation class of the fibres. We have isomorphisms of twisted cohomologies
\begin{eqnarray*}
H^*(E,H) &\simeq& H^{*-1}(\hat{E},(\xi,\hat{H})) \\
H^*(E,(\xi,H)) &\simeq& H^{*-1}(\hat{E},\hat{H}).
\end{eqnarray*}
More specifically the T-duality maps $T,T_\xi,\hat{T},\hat{T}_\xi$ are isomorphisms with $T^{-1} = -\hat{T}_\xi$ and $T_\xi^{-1} = -\hat{T}$.
\begin{proof}
We proved in Proposition \ref{inviso} that the twisted cohomology $H^*(E,H)$ can be computed using the subcomplex of invariant forms. The proof can easily to be extended to show that the twisted cohomology $H^*(E,(\xi,H))$ with local coefficients is likewise isomorphic to the subcomplex of invariant forms. The only difference is that $H^*_{\rm inv}(E,(\xi,H))$ corresponds to the $-1$-eigenspace of the $\mathbb{Z}_2$-action on $H^*(E_2 ,m^*H)$, where $m : E_2 \to E$ is the double cover determined by $\xi$.\\

Take $x \in \Omega_{\rm inv}^*(E)$. Using the twisted connection $A$ we may uniquely decompose $x$ as follows:
\begin{equation*}
x = \pi^*(\alpha) + A \wedge \pi^*(\beta)
\end{equation*}
where $\alpha \in \Omega^*(E)$, $\beta \in \Omega^{*-1}(E,\mathbb{R}_\xi)$. One easily checks that $Tx \in \Omega_{\rm inv}^{*-1}(E,\mathbb{R}_\xi)$ is given by
\begin{equation*}
Tx = \hat{\pi}^*(\beta) - \hat{A} \wedge \hat{\pi}^* (\alpha ).
\end{equation*}
Repeating this we find that $\hat{T}_\xi T x = -x$, so that $T$ is invertible on the space of invariant forms with inverse $-\hat{T}_\xi$. Since every cohomology class has an invariant representative the result follows. The proof that $T_\xi^{-1} = -\hat{T}$ is identical.
\end{proof}
\end{prop}


\section{Twisted $K$-theory}

Having demonstrated that T-duality gives an isomorphism between twisted cohomologies we would like to lift our results to twisted $K$-theory. For this we need $K$-theory twisted not only by elements of $H^3(X,\mathbb{Z})$ but also by $H^1(X,\mathbb{Z}_2)$. We will provide an overview of such twisted $K$-theory in terms of graded bundle gerbes.\\

The $K$-theory of a space twisted by classes in $H^1(X,\mathbb{Z}_2) \times H^3(X,\mathbb{Z})$ was introduced by Donovan and Karoubi \cite{dk} in the torsion case and by Rosenberg \cite{ros} for arbitrary classes in $H^3(X,\mathbb{Z})$. Other papers introducing twisted $K$-theory are \cite{bcmms}, \cite{atseg1} and \cite{fht}. We mostly follow the paper \cite{fht} of Freed, Hopkins and Teleman since they work with general classes in $H^1(X,\mathbb{Z}_2) \times H^3(X,\mathbb{Z})$. However we have chosen to describe twists of $K$-theory in terms of bundle gerbes rather than graded central extensions of groupoids.


\subsection{Graded bundle gerbes}\label{grbg}

On a given space $X$ the different possible twists of $K$-theory are classified by $H^1(X,\mathbb{Z}_2) \times H^3(X,\mathbb{Z})$, however this only classifies twists up to isomorphism. In the various models of twisted $K$-theory one can introduce a category of twists on $X$ which has $H^1(X,\mathbb{Z}_2) \times H^3(X,\mathbb{Z})$ as the set of isomorphism classes of objects. We will describe twists using a notion of graded bundle gerbes, a mild extension of the notion of bundle gerbe as introduced by Murray \cite{mm}. Our presentation is a straightforward adaptation of \cite{mm}, \cite{stev}, \cite{bcmms}. Graded bundle gerbes are close to the definition of twists used in \cite{fht} using graded central extensions of groupoids.\\

Let $X$ be a Hausdorff space. We will be concerned with pairs $(Y,\pi)$ where $Y$ is a Hausdorff space and $\pi : Y \to X$ a surjective map that admits local sections. If $(Y,\pi)$,$(Z,\mu)$ are two such pairs over $X$ then the fibre product $(Y \times_X Z , \pi \times_X \mu)$ is also a pair. When we turn to the twisted Chern character we will need to work in the smooth category so a (smooth) pair $(Y,\pi)$ over a smooth manifold $X$ is a smooth manifold $Y$ and a surjective submersion $\pi : Y \to X$ that admits local sections. In this case the fibre product of two smooth pairs is again a smooth pair.\\

A pair $(Y,\pi)$ over a space $X$ has a groupoid interpretation. We think of $X$ as a topological groupoid with only identity morphisms and we think of $Y$ as a topological groupoid with objects $Y^{[1]} = Y$ and morphisms $Y^{[2]} = Y \times_X Y$, so that any two objects of $Y$ have exactly one morphism between them if they lie in the same fibre of $\pi$. There is an evident functor $\pi : Y \to X$ which is a local equivalence in the language of \cite{fht}.

\begin{ex} Let $\mathcal{U} = \{ U_i \}_{i \in I}$ be an open cover of $X$. We let $X_{\mathcal{U}} = \coprod_{i \in I} U_i$ be the disjoint union. The evident map $\pi : X_{\mathcal{U}} \to X$ gives a pair $(X_{\mathcal{U}},\pi)$ over $X$.
\end{ex}

\begin{ex} Let $\pi : P \to X$ be a principal $G$-bundle. Then the groupoid structure on $P$ associated to the map $\pi : P \to X$ coincides with the groupoid structure on $P$ given by the right action of $G$. This groupoid is usually denoted $P /\!/ G$.
\end{ex}

\begin{defn} A {\em grading} on a groupoid $G = ({\rm Obj}(G),{\rm Mor}(G))$ is a functor $\epsilon : G \to \{{\rm pt}\} /\!/ \mathbb{Z}_2$, that is an assignment $\epsilon : {\rm Mor}(G) \to \mathbb{Z}_2$ of $\pm 1$ to morphisms of $G$ such that $\epsilon(fg) = \epsilon(f)\epsilon(g)$ when the composition is defined.
\end{defn}

\begin{ex} A grading for a pair $(X_{\mathcal{U}},\pi)$ associated to a cover $\mathcal{U}$ of $X$ is a collection of continuous maps $\xi_{ij} : U_{ij} \to \mathbb{Z}_2$ satisfying the cocycle condition. That is a grading for $(X_{\mathcal{U}},\pi)$ is a $\mathbb{Z}_2$-valued \v{C}ech $1$-cocycle.
\end{ex}

\begin{ex} If $X$ is connected and $\pi : P \to X$ a principal $G$-bundle then a grading for $P /\!/ G$ is precisely a group homomorphism $G \to \mathbb{Z}_2$.
\end{ex}

\begin{defn} Let $(Y,\pi)$ be a pair over $X$ regarded as a groupoid. $Y$. A {\em central extension} of $Y$ is a topological groupoid $L$ and functor $p : L \to Y$ such that $p : {\rm Obj}(L) \to Y^{[1]}$ is a homeomorphism, $p : {\rm Mor}(L) \to Y^{[2]}$ is a Hermitian line bundle and the groupoid multiplication in $L$ is complex linear and Hermitian in the fibres of $p$. 

A {\em bundle gerbe} $(Y,\pi,L,p)$ on $X$ is consists of a pair $(Y,\pi)$ over $X$ and a central extension $p : L \to Y$.

A {\em graded bundle gerbe} is a bundle gerbe $(Y,\pi,L,p)$ and a grading $\epsilon : Y \to \{{\rm pt}\} /\! / \mathbb{Z}_2$.

We will say that a gerbe $(Y,\pi,L,p)$ or graded gerbe $(Y,\pi,L,p,\epsilon)$ is {\em associated} to the pair $(Y,\pi)$.
\end{defn}

To each morphism $f \in Y^{[2]}$ we get a complex line $L_f$ and isomorphisms $m_{f,g} : L_f \otimes L_g \to L_{f g}$ whenever the composition $fg$ is defined. This composition must be associative in the evident sense. We can think of the grading $\epsilon : Y^{[2]} \to \mathbb{Z}_2$ as assigning to each line $L_f$ a grading $\epsilon(f) \in \mathbb{Z}_2$ and the multiplication $m_{f,g}$ respects the grading. Throughout this section the tensor product is to be understood in the graded sense, so there is a sign involved when exchanging factors of a tensor product.\\

Fix a pair $(Y,\pi)$ over $X$. We say that two graded gerbes $(Y,\pi,L,p,\epsilon)$, $(Y,\pi,L',p',\epsilon')$ are {\em strictly isomorphic} if there is an invertible functor of groupoids $F : L \to L'$ such that $p = p' \circ F$ and $F$ is a linear isometry between the fibres of $p : L \to Y$ and $p' : L' \to Y$.\\

The product of graded gerbes $(Y,\pi,L,p,\epsilon) \otimes (Y,\pi,L',p',\epsilon')$ is a graded gerbe of the form $(Y,\pi, L \otimes L' , \epsilon + \epsilon')$. Here $L \otimes L'$ denotes the graded tensor product of $L$ and $L'$. Let $f,g \in Y^{[2]}$ be morphisms such that the composition $fg$ is defined. The groupoid structure on $L \otimes L'$ is defined as follows:
\begin{equation*}\xymatrix{
(L_f \otimes L'_f ) \otimes (L_g \otimes L'_g) \ar[r]^{sw} & (L_f \otimes L_g) \otimes (L'_f \otimes L'_g) \ar[rr]^-{m_{f,g} \otimes m'_{f,g}} && L_{fg} \otimes L'_{fg}
}
\end{equation*}
where $sw$ is the morphism which swaps the inner two factors with appropriate sign factor. It is this sign factor that makes a difference between graded gerbes and ungraded gerbes.\\

One way to construct a graded gerbe is a follows: let $(T,t)$ be a graded Hermitian line bundle over $Y^{[1]}$, that is $T \to Y^{[1]}$ is a line bundle and $t : Y^{[1]} \to \mathbb{Z}_2$ a locally constant function which specifies a grading on the fibres of $T$. Let $\pi_0,\pi_1 : Y^{[2]} \to Y^{[1]}$ be the source and target maps. Then we define a graded gerbe $(Y,\pi,L,p,\epsilon)$ as follows: we let $L = \pi_0^*(T) \otimes \pi_1(T^{-1})$ and $\epsilon = \pi_0^*(t) - \pi_1^*(t)$. 

To define the groupoid structure on $L$, let $f,g \in Y^{[2]}$ be a pair such that the composition $fg$ is defined. Under the inclusion $Y^{[2]} \subseteq Y^{[1]} \times Y^{[1]}$ we may write $f = (i,j),g = (j,k),fg = (i,k)$. The groupoid structure on $L$ is the following:
\begin{equation*}\xymatrix{
L_f \otimes L_g \ar@{=}[r] & T_i \otimes T_j^{-1} \otimes T_j \otimes T_k^{-1} \ar[r] & T_i \otimes T_k^{-1} \ar@{=}[r] & L_{fg}.
}
\end{equation*}
We let $\delta(T)$ denote this gerbe. Such a gerbe is called {\em trivial}.

Suppose that $(T,t),(T',t')$ are graded line bundles over $Y^{[1]}$ and $\delta(T),\delta(T')$ the associated gerbes. A grading preserving isomorphism $\sigma : T \to T'$ induces a canonical strict isomorphism $\delta(\sigma) : \delta(T) \to \delta(T')$.\\

We will now make a $2$-category out of graded bundle gerbes associated to a pair $(Y,\pi)$. A $1$-morphism $(T,t,\tau)$ between $(Y,\pi,L,p,\epsilon)$ and $(Y,\pi,L',p',\epsilon')$ consists of a graded line bundle $(T,t)$ over $Y^{[1]}$ and a strict isomorphism $\tau : L \otimes \delta(T) \to L'$.

A $2$-morphism $\sigma$ between $1$-morphisms $(T,t,\tau),(T',t',\tau')$ is an isomorphism $\sigma : T \to T'$ of graded line bundles inducing a strict isomorphism $\delta(\sigma) : \delta(T) \to \delta(T')$ such that the following diagram of strict isomorphisms commute:
\begin{equation*}\xymatrix{
L \otimes \delta(T) \ar[r]^\tau \ar[d]_{1 \otimes \delta(\sigma)} & L' \\
L \otimes \delta(T') \ar[ur]^{\tau'} & 
}
\end{equation*}
\\

Observe that the automorphisms of a graded gerbe modulo $2$-morphisms can be identified with the group $H^0(X,\mathbb{Z}_2) \times H^2(X,\mathbb{Z})$ of isomorphism classes of graded line bundles on $X$.\\

For a given pair $(Y,\pi)$ over $X$ let $\mathcal{G}_Y$ denote the $2$-category of gerbes associated to $(Y,\pi)$. There is an associated $1$-category $[\mathcal{G}_Y]$ which has the same objects as $\mathcal{G}_Y$ but the morphisms of $[\mathcal{G}_Y]$ are morphisms of $\mathcal{G}_Y$ modulo equivalence under $2$-morphisms.


\subsection{Twists of $K$-theory}\label{kthtw}

So far we have made a $2$-category $\mathcal{G}_Y$ out of graded bundle gerbes associated to a fixed pair $(Y,\pi)$ over $X$. Objects of $\mathcal{G}_Y$ determine twisted $K$-theory groups of $X$. With some effort we can assemble the $\mathcal{G}_Y$ for different $Y$ into a single category of twists on $X$. We will not spell out the details since the construction is somewhat complicated and not necessary for our application. The details can be found in \cite{fht} or \cite{stev}. However we will still need to work with gerbes defined for different pairs so we discuss the relevant points.\\

Given a pair $(Y,\pi)$ over $X$ twisted $K$-theory can be thought of as a functor $K : \mathcal{G}_Y \to {\rm Ab}$ from $\mathcal{G}_Y$ to the category of abelian groups which is made into a $2$-category with only identity $2$-morphisms. Put another way the functor $K$ factors to the $1$-category quotient $K : [\mathcal{G}_Y] \to {\rm Ab}$. Given a graded gerbe $\alpha \in \mathcal{G}_Y$ we let $K^*(X,\alpha)$ denote the corresponding twisted $K$-theory. It is a $\mathbb{Z}_2$-graded abelian group. The definition is given in Section \ref{kdef}.\\

To compare gerbes on different spaces we consider a general notion of strict morphism between graded gerbes. Let $X,X'$ be spaces $(Y,\pi),(Y',\pi')$ pairs over $X,X'$ and let $\alpha,\alpha'$ be graded gerbes $\alpha \in \mathcal{G}_Y$, $\alpha' \in \mathcal{G}_{Y'}$. Let $L \to Y^{[2]}$, $L' \to Y'^{[2]}$ be the corresponding line bundles. We will use the term {\em strict morphism} to mean a triple of maps $(f,g,h)$ forming a commutative diagrams
\begin{equation*}\xymatrix{
L \ar[d]^p \ar[r]^h & L' \ar[d]^{p'} \\
Y^{[2]} \ar[r]^{g^{[2]}} & Y'^{[2]}
}
\end{equation*}
and
\begin{equation*}\xymatrix{
Y \ar[r]^g \ar[d]^\pi & Y' \ar[d]^{\pi'} \\
X \ar[r]^f & X'
}
\end{equation*}
such that for each $u \in Y^{[2]}$ the map $h_u : L_u \to L'_{g^{[2]}u}$ is an isomorphism and such that $h$ respects the grading and multiplication structures in the evident sense \cite{stev}. We write $(f,g,h) : (X,\alpha) \to (X',\alpha')$. Provided that the map $f : X \to X'$ between bases is a proper map such a triple induces a pull-back operation $(f,g,h)^* : K^*(X',\alpha') \to K^*(X,\alpha)$. The pull-back makes $K$-theory into a contravariant functor between spaces with graded gerbes.

In the special case that $X' = X$ and $f = {\rm id}$ is the identity it follows from \cite{fht} that the pull-back $({\rm id},g,h)^* : K^*(X,\alpha') \to K^*(X,\alpha)$ is actually an isomorphism.\\

There is a notion of homotopy between triples $(f_0,g_0,h_0)$ and $(f_1,g_1,h_1)$. It can be viewed as a strict morphism $(f_t,g_t,h_t) : (X \times [0,1] , p^* \alpha) \to (X',\alpha')$ where $p$ is the projection $p : X \times [0,1] \to X$. For twisted $K$-theory we want to assume $f_0,f_1$ are proper and $f_t$ is a homotopy through proper maps. Homotopy invariance of twisted $K$-theory is the property that if $(f_0,g_0,h_0),(f_1,g_1,h_1)$ are homotopic triples then $(f_0,g_0,h_0)^* = (f_1,g_1,h_1)^*$.\\

There are some special cases of this pull-back operation to consider. Given a proper map $f : X \to X'$ between spaces, a pair $(Y',\pi')$ over $X'$ and a graded gerbe $\alpha' \in \mathcal{G}_{Y'}$ then there is an evident pull-back $Y = f^*(Y')$, $\alpha = f^*\alpha' \in \mathcal{G}_{Y}$ which defines a pull-back map $f^* : K^*(X',\alpha) \to K^*(X,f^*\alpha')$. Next consider the case $X' = X$, $f = {\rm id}$. Thus suppose that $(Y,\pi),(Y',\pi')$ are pairs over $X$. Given a map $g : Y \to Y'$ such that $\pi = \pi' \circ g$ and an element $\alpha' \in \mathcal{G}_{Y'}$ there is an evident pull-back $g^*\alpha' \in \mathcal{G}_Y$. This is a special type of strict morphism so we get a pull-back $g^* : K^*(X,\alpha') \to K^*(X,g^*\alpha')$, which as we have already remarked is in fact an isomorphism.\\

Using pull-back we can form the product of gerbes $A,B$ associated to different pairs $(Y,\pi),(Y',\pi')$ over $X$. We pull-back $A,B$ to gerbes associated to the fibre product $(Y \times_X Y' , \pi \times_X \pi')$ and then use the product stucture on $\mathcal{G}_{Y \times_X Y'}$ that we have already defined. The product defined in this way is associative. Two graded gerbes $A,B$ are said to be {\em stably isomorphic} if the product $A \otimes B^*$ of $A$ and the dual of $B$ is trivial \cite{bcmms}.\\

Given a graded bundle gerbe $\alpha$ over a pair $(Y,\pi)$ we can find an open cover $\{ U_i \}$ with sections $s_i : U_i \to Y$. This further defines sections $s_{ij} = (s_i,s_j) : U_{ij} \to Y^{[2]}$. If $Q\to Y^{[2]}$ is the principal ${\rm U}(1)$-bundle defined by $\alpha$ then we may assume the cover is such that $s_{ij}^* Q$ is trivial. Let $r_{ij} : U_{ij} \to Q$ be a section. Define ${\rm U}(1)$-valued maps $g_{ijk} : U_{ijk} \to {\rm U}(1)$ by $m( r_{ij} , r_{jk} ) = r_{ik} g_{ijk}$ where $m$ denotes the gerbe multiplication. One sees that $\{ g_{ijk} \}$ is a \v{C}ech cocycle and that the class of $\{g_{ijk} \}$ in $H^2(X, \mathcal{C}({\rm U}(1))) = H^3(X,\mathbb{Z})$ does not depend on the choice of cover or choice of local sections. The class $[g_{ijk}]$ is called the {\em (ungraded) Diximier-Douady class} of $\alpha$. If we assume also that the $U_{ij}$ are connected then $s_{ij}^* Q$ has either even or odd parity and this gives us a $\mathbb{Z}_2$-valued cocycle $e_{ij}$, the class of which lies in $H^1(X,\mathbb{Z}_2)$. The pair $([e_{ij}],[g_{ijk}]) \in H^1(X,\mathbb{Z}_2) \times H^3(X,\mathbb{Z})$ will be called the {\em graded Diximier-Douady class} of $\alpha$, or just the Diximier-Douady class if no confusion will arise. Stably isomorphic graded gerbes have the same Diximier-Douady class and it is not hard to see that the converse is also true.\\

Let $\mathcal{TW}_X$ denote the set of all graded gerbes on $X$. Elements of $\mathcal{TW}_X$ are what we use for twists of $K$-theory. The set $| \mathcal{TW}_X |$ of stable isomorphism classes of graded gerbes is isomorphic to $H^1(X,\mathbb{Z}_2) \times H^3(X,\mathbb{Z})$ through the Diximier-Douady class. The multiplication defines a group structure on the set $| \mathcal{TW}_X |$ of isomorphism classes of graded gerbes. One finds that $| \mathcal{TW}_X |$ fits into an exact sequence
\begin{equation*}
0 \to H^3(X,\mathbb{Z}) \to |\mathcal{TW}_X| \to H^1(X,\mathbb{Z}_2) \to 0.
\end{equation*}
Under the canonical splitting $| \mathcal{TW}_X | = H^1(X,\mathbb{Z}_2) \times H^3(X,\mathbb{Z})$ given by the Diximier-Douady class the group operation becomes the following \cite{fht}:
\begin{equation}\label{grplaw}
(v, V) (w,W) = (v+w , V+W + \beta(vw) )
\end{equation}
where $\beta : H^2(X,\mathbb{Z}_2) \to H^3(X,\mathbb{Z})$ is the Bockstein homomorphism.\\

Given a graded gerbe $\alpha \in \mathcal{TW}_X$ we let $|\alpha|$ denote the isomorphism class of $\alpha$ in $| \mathcal{TW}_X | = H^1(X,\mathbb{Z}_2) \times H^3(X,\mathbb{Z})$. Since the twisted $K$-theory group $K^*(X,\alpha)$ depends up to isomorphism only of the stable isomorphism class of $\alpha$ we sometimes write $K^*(X,|\alpha|)$. It is important to note however that if $\alpha,\alpha'$ have the same graded Diximier-Douady class then $K^*(X,\alpha)$ and $K^*(X,\alpha')$ are not canonically isomorphic.


\subsection{Definition}\label{kdef}

We outline the definition of twisted $K$-theory associated to a graded bundle gerbe. For the most part we will not work directly with the definition which is quite techincal. It is included mainly for completeness. Our definition is based upon \cite{fht}, \cite{cw0}, \cite{atseg1}.\\

Let $\mathcal{H} = \mathcal{H}_0 \oplus \mathcal{H}_1$ be a $\mathbb{Z}_2$-graded separable complex Hilbert space with both the even and odd degree subspaces infinite dimensional. Following \cite{atseg1} the group ${\rm U}(\mathcal{H})$ of unitary operators is given the compact-open topology as a subspace of ${\rm End}(\mathcal{H}) \times {\rm End}(\mathcal{H})$ through the map $g \mapsto (g , g^{-1})$. We then give the projective unitary group ${\rm PU}(\mathcal{H})$ quotient topology. If $Z$ is metrizable and we are given a continuous map $Z \to {\rm U}(\mathcal{H})$ then the associated map $Z \times \mathcal{H} \to Z \times \mathcal{H}$ is a homeomorphism, so we can speak of Hilbert bundles with structure group ${\rm U}(\mathcal{H})$. Similarly for any space on which ${\rm U}(\mathcal{H})$ or ${\rm PU}(\mathcal{H})$ act continuously we can speak of fibre bundles with structure group ${\rm U}(\mathcal{H})$ or ${\rm PU}(\mathcal{H})$. More relevant here is the subgroup ${\rm U}(\mathcal{H}_0) \times {\rm U}(\mathcal{H}_1)$ of ${\rm U}(\mathcal{H})$ preserving the $\mathbb{Z}_2$-grading and the subgroup $\mathbb{Z}_2 \ltimes \left( {\rm U}(\mathcal{H}_0) \times {\rm U}(\mathcal{H}_1) \right)$ of elements that either preserve or reverse the grading. A $\mathbb{Z}_2$-graded Hilbert bundle will be assumed to have structure group ${\rm U}(\mathcal{H}_0) \times {\rm U}(\mathcal{H}_1)$.\\

Let $\pi : Y \to X$ be a pair and $L \to Y^{[2]}$ a graded bundle gerbe associated to $Y$. A {\em Hilbert module} for $L$ is a $\mathbb{Z}_2$-graded Hilbert bundle $H \to Y$ together with an isomorphism $\phi : L \otimes \pi_1^* H \to \pi_2^* H$ of graded Hilbert bundles on $Y^{[2]}$. The isomorphism $\phi$ is required to be compatible with the gerbe multiplication on $Y^{[3]}$ in the evident sense. There is also an evident notion of morphism between Hilbert modules which requires compatibility with the gerbe module structure.\\

Following \cite{fht} a Hilbert module $H$ for $L$ is called {\em universal} if for every Hilbert module $V$ for $L$ there is a unitary embedding $V \to H$ and {\em locally universal} if $H|_U$ is universal for every open subset $U \subseteq X$. According to \cite[Lemma 3.12]{fht} every graded bundle gerbe admits a locally universal Hilbert module $H$ which is unique up to unitary isomorphism. For each $y \in Y$, the isomorphism $L_{(y,y)} \otimes H_y \to H_y$ determines a representation of ${\rm U}(1)$. By \cite[Lemma 3.11]{fht} the eigenbundle $H(1) \subseteq H$ on which ${\rm U}(1)$ acts by its defining representation is also locally universal. Henceforth we assume this property whenever we choose a locally universal Hilbert module. It follows that the projectivization $\mathbb{P}(H(1))$ descends to a bundle of projective Hilbert spaces over $X$. The associated structure group is the group $\mathbb{Z}_2 \ltimes P({\rm U}(\mathcal{H}_0) \times {\rm U}(\mathcal{H}_1))$ of projective unitary transformations of $\mathcal{H} = \mathcal{H}_0 \oplus \mathcal{H}_1$ that either preserve or interchange the subspaces $\mathbb{P}(\mathcal{H}_0),\mathbb{P}(\mathcal{H}_1)$.\\

We now introduce a space on which the group $\mathbb{Z}_2 \ltimes P({\rm U}(\mathcal{H}_0) \times {\rm U}(\mathcal{H}_1))$ acts. Let ${\rm Fred}^{(0)}(\mathcal{H})$ denote the space of odd skew-adjoint Fredholm operators $A$ for which $A^2 + 1$ is compact. The topology of ${\rm Fred}^{(0)}(\mathcal{H})$ is induced by the map ${\rm Fred}^{(0)} \to {\rm End}(\mathcal{H}) \times \mathcal{K}(\mathcal{H})$ sending $A$ to $(A,A^2+1)$, where ${\rm End}(\mathcal{H})$ is given the compact-open topology and $\mathcal{K}(\mathcal{H})$ is the space of compact operators with the norm topology. According to \cite{atseg1} the space ${\rm Fred}^{(0)}(\mathcal{H})$ is a classifying space for untwisted $K$-theory. 

For any graded bundle gerbe $L$ on $X$ and locally universal Hilbert module $H$ we get a bundle of projective spaces on $X$ with structure group $\mathbb{Z}_2 \ltimes P({\rm U}(\mathcal{H}_0) \times {\rm U}(\mathcal{H}_1))$. Since this group acts continuously on ${\rm Fred}^{(0)}(\mathcal{H})$ we get an associated fibre bundle ${\rm Fred}_L^{(0)}(H)$ of Fredholm operators. We would also like to define bundles $\Omega_X^n {\rm Fred}_L^{(0)}(H)$ whose fibres are iterated based loop spaces. For this a little care is required. The space ${\rm Fred}^{(0)}$ doesn't have a natural basepoint. We can take any invertible element as a base point but such a point will not be fixed by the group action. Thus to define the bundles $\Omega_X^n {\rm Fred}_L^{(0)}(H)$ of loop spaces we need to first choose a section $s_0 : X \to {\rm Fred}_L^{(0)}(H)$ which takes values in the invertible operators. Such a section exists because the space of invertible elements of ${\rm Fred}^{(0)}(\mathcal{H})$ is contractible.

\begin{defn} The twisted $K$-theory group $K^{-n}(X,L)$ of the graded bundle gerbe $L$ is the space of homotopy classes of compactly supported sections of the fibre bundle $\Omega_X^n {\rm Fred}_L^{(0)}(H)$. A section of ${\rm Fred}_L^{(0)}(H)$ is compactly supported if there is a compact subset of $X$ outside of which the section takes values in the invertible operators. A similar notion of compactly supported section applies to the bundles $\Omega_X^n{\rm Fred}_L^{(0)}(H)$.
\end{defn}

As it stands the definition appears to depend on the choice of Hilbert module $H$. However any two Hilbert modules are isomorphic and any two such isomorphisms can be shown to be homotopic \cite{fht}, so that there is a canonical identification of homotopy classes of sections of the associated bundles. From the definition we see that the $K^{-n}(X,L)$ have a natural abelian group structure.\\

Let $f : Y \to X$ be a proper map and suppose $L$ is a graded gerbe on $X$ and $H$ a locally universal Hilbert module. We can define the pullback $f^*L$ of $L$ and $f^*H$ with the property that $f^*H$ is a Hilbert module for $f^*L$. Let $H'$ be a locally universal Hilbert module for $f^*L$ and choose a unitary embedding $f^*H \to H'$. This induces a map $f^* \Omega^n_X {\rm Fred}^{(0)}_L (H) \to \Omega^n_Y {\rm Fred}^{(0)}_{f^*L}(H')$ and further a map $f^* : K^{-n}(X,L) \to K^{-n}(Y,f^*L)$ which can be shown to be independent of choices involved. More generally there is a pullback operation $(f,g,h)^*$ associated to triples $(f,g,h)$ as described in Section \ref{kthtw}.\\

In \cite{atseg1} it is shown that there is an equivariant homotopy equivalence ${\rm Fred}^{(0)}(\mathcal{H}) \to \Omega^2 {\rm Fred}^{(0)}(\mathcal{H})$. This gives natural isomorphisms $K^{-n-2}(X,L) \simeq K^{-n}(X,L)$ and allows the definition of $K^i(X,L)$ for positive $i$ by periodicity.


\subsection{Some properties of twisted $K$-theory}

We now outline some of the main propeties of twisted $K$-theory as developed in \cite{fht},\cite{cw0},\cite{atseg1}. These properties make it possible to carry out calculations of twisted $K$-theory groups without directly working with the defintion.\\

{\bf Multiplication}: given twists $\alpha , \beta \in \mathcal{TW}_X$ there is a corresponding multiplication:
\begin{equation*}
K^i(X,\alpha) \otimes K^j(X,\beta) \to K^{i+j}(X,\alpha \otimes \beta)
\end{equation*}
which is associative and graded commutative. In particular this gives $K^*(X,\alpha)$ the structure of a $K^*(X)$-module. \\

{\bf Push-forward}: To define the push-forward first we must explain how a graded gerbe can be associated to a vector bundle. Let $V \to X$ be a rank $k$ orthogonal vector bundle over $X$. Associated to $V$ is the principal ${\rm O}(k)$-bundle $\pi : P \to X$ of orthonormal frames. We view $(P,\pi)$ as a pair over $X$ or as the groupoid $P / \! / {\rm O}(k)$. There is a natural identification $P^{[2]} = P \times {\rm O}(k)$. Now the group\footnote{There are two distinct pin groups ${\rm Pin}_{\pm}(k)$ but the associated pin$\!\,^c$ groups are isomorphic.} ${\rm Pin}^c(k)$ fits into an exact sequence \cite{abs}
\begin{equation*}
1 \to {\rm U}(1) \to {\rm Pin}^c(k) \to {\rm O}(k) \to 1
\end{equation*}
and this defines a principal ${\rm U}(1)$-bundle over $P^{[2]}$ and an associated line bundle $L \to P^{[2]}$. The group multiplication in ${\rm Pin}^c(k)$ then determines a multiplication on $L$ giving it the structure of a bundle gerbe. Moreover the determinant homomorphism ${\rm O}(k) \to \mathbb{Z}_2$ gives $L$ the structure of a graded gerbe. We call $L = L(V)$ the {\em lifting gerbe} of $V$. The Dixmier-Douady class of $L$ is $|L| = (w_1(V),W_3(V))$, where $w_1(V)$ is the first Stiefel-Whitney class of $V$ and $W_3(V)$ is the third integral Stiefel-Whitney class of $V$.\\

Let $X,Y$ be smooth manifolds and let $f : X \to Y$ be a smooth map. Choose Riemannian metrics on $X$ and $Y$. The bundle $TX \oplus f^*(TY)$ has a metric so we may take the lifting gerbe $L(TX \oplus f^*(TY))$ which we also denote by $L(f)$. If we choose different metrics on $X$ and $Y$ we get isomorphic lifting gerbes. Indeed choosing a path of metrics joining the given metrics we can identify the frame bundles by parallel translation and this gives an isomorphism. Two different paths are homotopic and this induces a homotopy between the two isomorphisms.

Let $\alpha \in \mathcal{TW}_Y$ be a graded gerbe on $Y$. The push-forward is a map
\begin{equation*}
f_* : K^*(X , L(f) \otimes f^*(\alpha)) \to K^{*-d}(Y,\alpha)
\end{equation*}
where $d = {\rm dim}(X) - {\rm dim}(Y)$. Note that the class of $L(f)$ is trivial if and only if $TX \oplus f^*(TY)$ admits a ${\rm Spin}^c$-structure which is the condition one usually requires to define a push-forward in untwisted $K$-theory.

The push-forward respects compositions and commutes with pull-backs in the obvious sense.\\

The case where $f : X \to Y$ is a fibre bundle deserves special attention. Let $V = {\rm Ker}(f_*)$ be the vertical bundle. Choose metrics on $Y$ and $V$ and choose an isomorphism $TX = V \oplus f^*(TY)$. This determines a corresponding metric on $X$. As bundles with metrics we have $TX \oplus f^*(TY) = V \oplus f^*(TY \oplus TY)$. We show there is a canonical isomorphism between the lifting gerbe for $TX \oplus f^*(TY) = V \oplus f^*(TY \oplus TY)$ and the lifting gerbe for $V$. For this it suffices to show that for any rank $k$ vector bundle $E$ with metric there is a canonical ${\rm Pin}^c(2k)$-structure on $E \oplus E$. Recall that the map ${\rm U}(k) \to {\rm SO}(2k)$ admits a lift ${\rm U}(k) \to {\rm Spin}^c(2k)$. We then have a commutative diagram
\begin{equation*}\xymatrix{
{\rm O}(k) \ar[r] \ar[dd] \ar[drr] & {\rm U}(k) \ar[r] & {\rm Spin}^c(2k) \ar[d] \\
& & {\rm Pin}^c(2k) \ar[d] \\
{\rm O}(k) \times {\rm O}(k) \ar[rr] & & {\rm O}(2k)
}
\end{equation*}
which gives the desired ${\rm Pin}^c(2k)$-structure. Let $L(V)$ be the lifting gerbe for $V$ and let $\alpha \in \mathcal{TW}_Y$ be any graded gerbe on $Y$. It follows that there is a canonical isomorphism $K^*(X , L(f) \otimes f^*(\alpha)) \simeq K^*(X,L(V)\otimes f^*(\alpha))$. We then have the following alternative form of the push-forward for fibre bundles:
\begin{equation*}
f_* : K^*(X , L(V) \otimes f^*(\alpha)) \to K^{*-d}(Y,\alpha).
\end{equation*}
This form of the push-forward also respects composition and commutes with pull-backs.\\

Finally another special case of the push-forward occurs for open embeddings $i : U \to X$. Suppose $\alpha$ is a gerbe on $X$ and by restriction defines a gerbe $\alpha|_U$ on $U$. The push-forward $i_* : K^*(U,\alpha|_U) \to K^*(X,\alpha)$ amounts to taking a homotopy class of section supported on $U$ and extending it trivially to $X$.\\

{\bf Mayer-Vietoris}: if $U_1,U_2$ are open sets covering $X$ and $\alpha \in \mathcal{TW}_X$ is a twist then we have a natural long exact sequence
\begin{equation*}\xymatrix{
K^0(X,\alpha) \ar[d] & K^0(U_1,\alpha) \oplus K^0(U_2,\alpha) \ar[l] & K^0(U_1 \cap U_2 , \alpha) \ar[l] \\
K^1(U_1 \cap U_2,\alpha) \ar[r] & K^1(U_1,\alpha) \oplus K^1(U_2,\alpha) \ar[r] & K^1(X,\alpha) \ar[u]
}
\end{equation*}
where $\alpha$ defines twists on $U_1,U_2, U_1 \cap U_2$ by restriction. Moreover the Mayer-vietoris sequence is natural is the sense that pull-back, push-forward and isomorphism of twists yield commuting Mayer-Vietoris sequences.


\subsection{T-duality of twisted $K$-theory}\label{tdtk}
Now we consider the behavior of twisted $K$-theory under T-duality of circle bundles. Let $(E,h),(\hat{E},\hat{h})$ be T-dual pairs over a base $M$. We assume that $E,\hat{E},M$ are smooth manifolds since we want to use the push-forward in twisted $K$-theory. It may be possible to define the push-forward in a more general context but we will not pursue this. Represent $h,\hat{h}$ by twists $\alpha_h \in \mathcal{TW}_E$, $\alpha_{\hat{h}} \in \mathcal{TW}_{\hat{E}}$ such that $h,\hat{h}$ are the graded Dixmier-Douady classes of $\alpha_h,\alpha_{\hat{h}}$. As usual let $F = E \times_M \hat{E}$ be the correspondence space with projections $p,\hat{p}$. Since $p^*h = \hat{p}^* \hat{h}$ it follows that there is an isomorphism $\eta : p^* (\alpha_h) \to \hat{p}^* (\alpha_{\hat{h}})$ inducing an isomorphism
\begin{equation*}
u : K^*(F,p^*\alpha_h) \to K^*(F,\hat{p}^*\alpha_{\hat{h}})
\end{equation*}
which generally depends on the choice of isomorphism $\eta$. Let $\xi \in H^1(M,\mathbb{Z}_2)$ denote the Stiefel-Whitney class of $E,\hat{E}$. Let $\gamma = L(V) \in \mathcal{TW}_M$ be the lifting gerbe for the vertical bundle of $E$ (and also $\hat{E}$). Note that the Dixmier-Douady class of $\gamma$ is $(\xi,0) \in |\mathcal{TW}_M|$. Then $\beta_h = \gamma \otimes \alpha_h$ represents $(\xi,h)$ and $\beta_{\hat{h}} = \gamma \otimes \alpha_{\hat{h}}$ represents $(\xi,\hat{h})$. We then have an isomorphism $1 \otimes \eta : \beta_h \to \beta_{\hat{h}}$ so there is an associated isomorphism
\begin{equation*}
u_\xi : K^*(F,p^* \beta_h) \to K^*(F,\hat{p}^* \beta_{\hat{h}}).
\end{equation*}
Following the example of twisted cohomology we will define a T-duality transformation
\begin{equation*}
T_\xi : K^*(E,\beta_h) \to K^{*-1}(\hat{E}, \alpha_{\hat{h}})
\end{equation*}
by the composition
\begin{equation*}
T_\xi = \hat{p}_* \circ u_\xi \circ p^*.
\end{equation*}
Expressed as a diagram the map $T_\xi$ is the composition
\begin{equation*}\xymatrix{
K^*(E,\beta_h) \ar[r]^{p^*} & K^*(F,p^*\beta_h) \ar[r]^{u_\xi} & K^*(F,\hat{p}^* \beta_{\hat{h}}) \ar[r]^{\hat{p}_*} & K^{*-1}(\hat{E} , \alpha_{\hat{h}}).
}
\end{equation*}
Note that this is well defined since the push-forward $\hat{p}_*$ has the form
\begin{equation*}
\hat{p}_* : K^*(F,\hat{p}^* \beta_{\hat{h}}) = K^*(F,\hat{p}^* \gamma \otimes \hat{p}^* \alpha_{\hat{h}}) \to K^{*-1}(\hat{E} , \alpha_{\hat{h}})
\end{equation*}
since $\hat{p}^* \gamma$ is the lifting gerbe for the vertical bundle of $F \to \hat{E}$.\\

We can similarly define a $T$-duality map $T : K^*(E,\alpha_h) \to K^{*-1}(\hat{E},\beta_{\hat{h}})$ by the following composition
\begin{equation*}\xymatrix{
K^*(E,\alpha_h) \ar[r]^{p^*} & K^*(F,p^*\alpha_h) \ar[r]^{u} & K^*(F,\hat{p}^* \alpha_{\hat{h}}) \ar[r]^{\hat{p}_*} & K^{*-1}(\hat{E} , \beta_{\hat{h}}).
}
\end{equation*}
Note that in order to define the push-forward here we are identifying $\hat{p}^* \gamma \otimes \hat{p}^* \beta_{\hat{h}}$ with $\hat{p}^* \alpha_{\hat{h}}$ by means of a trivialization of $\gamma \otimes \gamma$. Since $\gamma$ is a lifting gerbe we have seen that there is a trivialization of $\gamma \otimes \gamma$ which at the level of $K$-theory gives a canonical isomorphism $K^*(F , \hat{p}^* \gamma \otimes \hat{p}^* \beta_{\hat{h}}) \to K^*(F , \hat{p}^* \alpha_{\hat{h}})$.\\

At the level of isomorphism classes we can view $T,T_\xi$ as maps
\begin{eqnarray*}
&T& : K^*(E,h) \to K^{*-1}(\hat{E},(\xi,\hat{h})) \\
&T_\xi& : K^*(E, (\xi,h)) \to K^{*-1}(\hat{E}, \hat{h})
\end{eqnarray*}
though $T,T_\xi$ viewed in this way are not natural transformations due to the choices involved.\\

In order for the T-duality transformations $T,T_\xi$ to be isomorphisms we need to choose the isomorphism $\eta : p^* (\alpha_h) \to \hat{p}^* (\alpha_{\hat{h}})$ to have a specific property \cite{brs}. For each point $b \in M$ consider the fibres $T_b = \pi^{-1}(b)$ and $\hat{T}_b = \hat{\pi}^{-1}(b)$ of $E$ and $\hat{E}$ over $b$. The fibre of the correspondence space $F$ over $b$ is then $T_b \times \hat{T}_b$. Since the fibres $T_b,\hat{T}_b$ are one dimensional we have that the restrictions of $h,\hat{h}$ to the fibres are trivial. Thus $\alpha_h$ and $\alpha_{\hat{h}}$ restricted to $T_b$ and $\hat{T}_b$ respectively admit trivializations $\tau : 0 \to \alpha_h|_{T_b}$, $\hat{\tau} : 0 \to \alpha_{\hat{h}}|_{\hat{T}_b}$. Now comparing the isomorphisms $\eta |_{T_b \times \hat{T}_b}$ and $\hat{p}^* \hat{\tau} \circ ( p^*{\tau} )^{-1}$ we have they differ by an automorphism of $p^*(\alpha_h)|_{T_b \times \hat{T}_b }$ and this projects to a class $\mu \in H^2( T_b \times \hat{T}_b , \mathbb{Z})$. If we change the trivializations $\tau,\hat{\tau}$ then the class $\mu$ changes by elements in the subspaces $H^2(T_b,\mathbb{Z}),H^2(\hat{T}_b,\mathbb{Z})$. In our case $T_b,\hat{T}_b$ are $1$-dimensional so in fact the class $\mu \in H^2(T_b \times \hat{T}_b , \mathbb{Z}) \simeq \mathbb{Z}$ is well defined. Note that since $E,\hat{E}$ have the same Stiefel-Whitney class the bundle $F \to M$ is oriented and a choice of orientation fixes the isomorphism $H^2(T_b \times \hat{T}_b , \mathbb{Z}) \simeq \mathbb{Z}$. 

The condition we require on $\eta$ is that for all $b \in M$ the class $\mu = \mu(\eta)$ under the above identification equals $1$. This condition ensures that the transformations $T,T_\xi$ restricted to individual fibres recovers the $K$-theory equivalent of the Fourier-Mukai transform. The class $1 \in \mathbb{Z} \simeq H^2(T_b \times \hat{T}_b , \mathbb{Z})$ is the Chern class of the Poincar\'e line bundle. An isomorphism $\eta$ with this property will be said to have the {\em Poincar\'e property}.\\

It is always possible to find an isomorphism $\eta$ with the Poincar\'e property. Our proof makes use of gerbes with connection and curving so we defer the proof until Section \ref{tcctd}, where it is proved in Proposition \ref{poinp}.

\begin{prop}\label{tki}
Let $M$ be a smooth manifold admitting a finite good cover. For $T$-dual pairs $(E,h)$, $(\hat{E},\hat{h})$ over $M$ we have isomorphisms
\begin{eqnarray*}
K^{*}(E,h) &\simeq & K^{*-1}(\hat{E},(\xi,\hat{h})) \\
K^{*}(E,(\xi,h)) &\simeq & K^{*-1}(\hat{E},\hat{h}).
\end{eqnarray*}
The T-duality transformations $T,T_\xi$ are isomorphisms provided we choose an isomorphism $\eta : p^* (\alpha_h) \to \hat{p}^* (\alpha_{\hat{h}})$ with the Poincar\'e property, as explained above.

\begin{proof}
Following \cite{bunksch} we will prove the result in the case that the base has a finite good cover using the Mayer-Vietoris sequence for twisted $K$-theory.\\

First note that for any open subset $U$ of $M$, we may define the restrictions $E_U = \pi^{-1}(U)$, $\hat{E}_U = \hat{\pi}^{-1}(U)$ of $E,\hat{E}$ and if $\alpha_h$, $\alpha_{\hat{h}}$ are twists representing $h,\hat{h}$ we can restrict $\alpha_h,\alpha_{\hat{h}}$ to $E_U,\hat{E}_U$. The morphism $u : p^*(\alpha_h) \to \hat{p}^*(\alpha_{\hat{h}})$ restricts to a morphism $u_U : p^*(\alpha_h|_U) \to \hat{p}^*(\alpha_{\hat{h}}|_U)$ over $p^{-1} \pi^{-1}(U) = \hat{p}^{-1}\hat{\pi}^{-1}(U)$. It follows that the T-duality maps $T,T_\xi$ can be restricted to open subsets and we get commutative squares
\begin{equation*}\xymatrix{
K^*(E,(\xi,h)) \ar[r]^-{T_\xi}  & K^{*-1}(\hat{E},\hat{h}) \\
K^*(E_U,(\xi|_U,h|_U)) \ar[r]^-{T_\xi} \ar[u]^{i_*} & K^{*-1}(\hat{E}_U,\hat{h}|_U) \ar[u]^{i_*}
}
\end{equation*}
and similarly for $T$.\\

Now suppose we have a decomposition $M = U \cup V$ of $M$ into two open subsets. Taking the open cover $E_U = \pi^{-1}(U)$, $E_V = \pi^{-1}(U)$ of $E$ and a similar open cover $\hat{E}_U,\hat{E}_V$ for $\hat{E}$ we get two associated Mayer-Vietoris sequences. Naturality of the Mayer-Vietoris sequence ensures that the T-duality map is a chain map between the two exact sequences. If the T-duality map restricted to $U$,$V$ and $U \cap V$ are isomorphisms then by the $5$-lemma the T-duality map itself is an isomorphism.\\

Now we prove that the T-duality maps are isomorphisms whenever the base has a finite good cover. Say a space has a good cover of order $n$ if it has a good cover with $n$ elements. We prove the result by induction on $n$.

When $n=1$ we have that $M$ is contractible and the result is trivial to verify. In fact $E,\hat{E}$ must be oriented circle bundles so the result already follows from \cite{bunksch}. Note it is here that we require the property $\mu(\eta) = 1$.\\

If the space $M$ has a good cover $U_1 , U_2 , \dots , U_n , U_{n+1}$ then let $U = U_1 \cup U_2 \cup \dots \cup U_n$, $V = U_{n+1}$. Clearly $U$ has a good cover of order $n$ and $V$ a good cover of order $1$, so by induction the $T$-duality map is an isomorphism on $U$ and $V$. By the Mayer-Vietoris sequence the $T$-duality map is an isomorphism on $M$ too.
\end{proof}
\end{prop}


\section{The twisted Chern character}

We have shown that one can define T-duality maps which are isomorphisms between twisted cohomology or twisted $K$-theory on T-dual circle bundles. Now we would like to refine this result to a statement that T-duality commutes with the Chern character in the sense that we have a commutative diagram of the form
\begin{equation*}\xymatrix{
K^*(E,h) \ar[r]^-{T} \ar[d]^{Ch} & K^{*-1}(\hat{E},(\xi,\hat{h})) \ar[d]^{Ch} \\
H^*(E,H) \ar[r]^-{T} & H^{*-1}(E,(\xi,\hat{H}))
}
\end{equation*}
and similarly with $(E,h),(\hat{E},\hat{h})$ interchanged. We will see that this is indeed the case although it requires some care. The main issue is that to define the twisted $K$-theory we need a graded gerbe, while to define the twisted cohomology we need a $3$-form representing the class of the gerbe.\\

To obtain a commutative diagram as above we need to unify these objects using gerbes with connection and curving. A related problem is that the T-duality maps (denoted by $T_\xi$ in the diagram) depend on certain choices. A morphism at the level of gerbes with connection and curving will simultaneously make choices for the T-duality transformations on $K$-theory and twisted cohomology which commute with the Chern character. 


\subsection{Gerbe connections}

Gerbe connections and curving for graded gerbes are identical to the ungraded case as can be found in \cite{mm}, \cite{stev}, \cite{bcmms}. In this section we review relevant details.\\

In order to define connections on gerbes we work in the category of smooth manifolds. Thus a pair $(Y,\pi)$ over a smooth manifold $X$ is a smooth manifold $Y$ and surjective submersion $\pi : Y \to X$. Let $Y^{[k]}$ denote the $k$-fold fibre product of $Y$ with itself. We get natural projections $\pi_i : Y^{[k+1]} \to Y^{[k]}$ for $i = 0, \dots , k$ which omit the $(i+1)$-th factor. Define $\delta : \Omega^*(Y^{[k]}) \to \Omega^*(Y^{[k+1]})$ to be
\begin{equation*}
\delta = \sum_{i=0}^k (-1)^i \pi_i^*.
\end{equation*}
For any $p$ the sequence
\begin{equation*}\xymatrix{
0 \ar[r] & \Omega^p(X) \ar[r]^{\pi^*} & \Omega^p(Y^{[1]}) \ar[r]^\delta & \Omega^p(Y^{[2]}) \ar[r]^\delta & \dots
}
\end{equation*}
is exact \cite{mm}.\\

 Let $(Y,\pi,L,p)$ be a gerbe on $X$, so in particular $p : L \to Y^{[2]}$ is a line bundle. Associated to $L$ is a principal ${\rm U}(1)$-bundle $Q \to Y^{[2]}$. A {\em gerbe connection} is a Hermitian connection $\theta \in \Omega^1(Q) \otimes \mathbb{C}$ for $Q \to Y^{[2]}$ which preserves the gerbe multiplication. Since $\theta$ repsects the multiplication one finds that the curvature of the connection $F \in \Omega^2(Y^{[2]}) \otimes \mathbb{C}$ satisfies $\delta (F) = 0$. Since $\theta$ is Hermitian we have that $F/2\pi i$ is real, so one can find a $2$-form $\omega \in \Omega^2(Y^{[1]}) \otimes \mathbb{C}$ such that $F = \delta(\omega)$ and such that $\omega/2 \pi i$ is real. Such a choice of $\omega$ is called a {\em curving} for the gerbe connection. We have that $\delta (d\omega) = d (\delta \omega) = d F = 0$ so that there exists a unique $3$-form $H \in \Omega^3(X)$ such that $\pi^*H = d \omega / (2 \pi i)$ called the {\em curvature} of the curving. Clearly $H$ is a representative in de Rham cohomology for the image of the Dixmier-Douady class of the gerbe in real cohomology. Note that every gerbe admits a connection and curving.\\

Fix a pair $(Y,\pi)$ over $X$. We will make a $2$-category out of gerbes associated to $(Y,\pi)$ with connection and curving. Let $(Y,\pi,L,p)$, $(Y,\pi,L',p')$ be gerbes associated to $(Y,\pi)$ we sometimes use just $L,L'$ to refer to the gerbes. Next suppose $\theta , \omega$ and $\theta',\omega'$ are connections and curvings for $G$ and $G'$ respectively. A $1$-morphism $(T,\tau,\nabla,B)$ consists of:
\begin{itemize}
\item{a Hermitian line bundle $T \to Y^{[1]}$}
\item{a connection $\nabla$ on $T$}
\item{a $2$-form $B \in \Omega^2(X)$}
\item{a strict isomorphism of gerbes $\tau : L \otimes \delta(T) \to L'$}
\end{itemize}
satisfying some conditions that we now explain. Note that the trivializable gerbe $\delta(T)$ inherits a natural connection and curving. The connection comes from $\nabla$ and the curving $F$ is the curvature of $T$. Then we require that $\tau$ identifies the product connection on $L \otimes \delta(T)$ with the connection on $L'$ and that the curvings $\omega,\omega'$ of $L,L'$ are related by
\begin{equation}\label{curvingshift}
\omega' = \omega + F + 2 \pi i B.
\end{equation}

A $2$-morphism $(\sigma,u)$ between $1$-morphisms $(T,\tau,\nabla,B),(T',\tau',\nabla',B')$ consists of an isomorphism $\sigma : T \to T'$ of line bundles inducing a strict isomorphism $\delta(\sigma) : \delta(T) \to \delta(T')$ and a $1$-form $u \in \Omega^1(X)$ satisfying some conditions which we explain. There must be a commutative diagram of strict isomorphisms:
\begin{equation*}\xymatrix{
L \otimes \delta(T) \ar[r]^\tau \ar[d]_{1 \otimes \delta(\sigma)} & L' \\
L \otimes \delta(T') \ar[ur]^{\tau'} & 
}
\end{equation*}
Moreover the connections $\nabla,\nabla'$ on $T,T'$ are related as follows:
\begin{equation*}
(\sigma^{-1})^* \nabla = \nabla' + 2 \pi i u
\end{equation*}
and the $2$-forms $B,B'$ are related by
\begin{equation*}
B' = B - du.
\end{equation*}

We can repeat all of the above definitions for graded gerbes. If two gerbes $L,L'$ are isomorphic and we give $L,L'$ any choice of connections $\theta,\theta'$ and curvings $\omega,\omega'$ there exists an isomorphism between $(L,\theta,\omega)$ and $(L',\theta',\omega')$. The group of of automorphisms of a gerbe with connection and curving modulo $2$-morphisms is the group of isomorphism classes of line bundles.\\

We call a $1$-morphism determined by a $2$-form $B \in \Omega^2(X)$ a {\em $B$-shift}. There is a more restricted notion of $1$-morphism in which we set the $2$-form $B = 0$. Obviously every morphism factors into a $B$-shift and a morphism with $B=0$.\\

So far we have made a $2$-category out of gerbes with connection and curving which are associated to a fixed pair $(Y,\pi)$. There is an evident notion of pull-back under diagrams of the form
\begin{equation*}\xymatrix{
Y' \ar[rr]^f \ar[dr]_{\pi'} & & Y \ar[dl]^{\pi} \\
& X &
}
\end{equation*}
or even a pull-back under base maps $f : X \to X'$. Using this we can speak of isomorphism between gerbes with connection and curving defined with respect to different pairs.


\subsection{Twisted Chern character}\label{tcc}

Given a graded gerbe $\mathcal{G}$ on $X$ with connection and curving one may define a twisted Chern character
\begin{equation*}
Ch_{\mathcal{G}} : K^{*}(X,\mathcal{G}) \to H^{*}(M,(\xi,H))
\end{equation*}
where $H$ is the curvature of $\mathcal{G}$ and $\xi$ is the $H^1(X,\mathbb{Z}_2)$ part of the Diximier-Douady class of $\mathcal{G}$. For simplicity we assume that $X$ is compact although it should be possible to define the Chern character more generally using compactly supported cohomology. We have used $Ch_{\mathcal{G}}$ to denote the twisted Chern character associated to the gerbe with connection and curving $\mathcal{G}$. This notation obscures the fact that $Ch_{\mathcal{G}}$ depends on the connection and curving however.\\

The twisted Chern character in the case of $K$-theory twisted by classes in $H^3(X,\mathbb{Z})$, that is by ungraded gerbes is defined in \cite{bcmms}, \cite{ms}. The extension to the more general class of twists is quite straightforward. Indeed, a class $\xi \in H^1(X,\mathbb{Z}_2)$ defines a double cover $m : X_2 \to X$ with an involution $\sigma : X_2 \to X_2$ which permutes the elements in each fibre. Associated to $\xi$ is a flat line bundle $L$ and associated to $L$ is a lifting gerbe $\gamma$ which has Diximier-Douady class $(\xi,0)$. Any graded gerbe $\mathcal{G}$ with Diximier-Douady class $(\xi,h)$ can be canonically factored into a product $\mathcal{G} = \gamma \otimes \mathcal{G}_0$ where $\mathcal{G}_0$ is an ungraded gerbe with Diximier-Douady class $h$. We then have an exact $6$-term sequence \cite{fht2}
\begin{equation*}\xymatrix{
K^0(X,\mathcal{G}) \ar[r]^-{m^*} & K^0(X_2, m^* \mathcal{G}_0 ) \ar[r] & K^0(X , \mathcal{G}_0) \ar[d] \\
K^1(X,\mathcal{G}_0) \ar[u] & K^1(X_2 , m^* \mathcal{G}_0 ) \ar[l] & K^1(X,\mathcal{G}) \ar[l]_-{m^*}
}
\end{equation*}
Now suppose we give $\mathcal{G}_0$ a connection and curving with curvature $H$. As an ungraded gerbe $\gamma$ is canoncially trivial so we can give it the trivial connection and curving. We then give $\mathcal{G}$ the product connection and curving which also has curvature $H$ and $m^* \mathcal{G}_0$ the pull-back connection and curving with curvature $m^* H$. 

There is a similar long exact sequence in twisted cohomology which is split:
\begin{equation*}\xymatrix{
H^0(X,(\xi,H)) \ar[r]^-{m^*} & H^0(X_2,m^*H) \ar[r] & H^0(X,H) \ar[d]^{0} \\
H^1(X,H) \ar[u]_{0} & H^1(X_2,m^*H) \ar[l] & H^1(X,(\xi,H)) \ar[l]_-{m^*}
}
\end{equation*}
Given this there is a unique way to define $Ch_{\mathcal{G}} : K^*(X,\mathcal{G}) \to H^*(X,(\xi,H))$ so that we have a commutative diagram
\begin{equation*}\xymatrix{
\dots \ar[r] & K^{*-1}(X,\mathcal{G}_0) \ar[r] \ar[d]^{Ch_{\mathcal{G}_0}} & K^*(X,\mathcal{G}) \ar[r]^{m^*} \ar[d]^{Ch_{\mathcal{G}}} & K^*(X_2, m^* \mathcal{G}_0 ) \ar[r] \ar[d]^{Ch_{m^*\mathcal{G}}} & \dots \\
\dots \ar[r] & H^{*-1}(X,H) \ar[r]^{0} & H^*(X,(\xi,H)) \ar[r]^{m^*} & H^*(X_2,m^*H) \ar[r] & \dots
}
\end{equation*}

We remark that for any compact smooth manifold $X$ the twisted Chern character $Ch_{\mathcal{G}} : K^*(X,\mathcal{G})\otimes \mathbb{R} \to H^*(X,(\xi,H))$ tensored by $\mathbb{R}$ is an isomorphism \cite{fht2}. The twisted Chern character respects pull-backs.\\

Let $\mathcal{G},\mathcal{G}'$ be graded gerbes with connection and curving associated to a pair $(Y,\pi)$ over $X$. Suppose there is an isomorphism $u :\mathcal{G} \to \mathcal{G}'$ which involves a $B$-shift by $B \in \Omega^2(X)$. Let $\xi$ be the $H^1(X,\mathbb{Z}_2)$-part of the Diximier-Douady class of $\mathcal{G}$ and $\mathcal{G'}$ and let $H,H'$ be the curvature $3$-forms. Clearly $H' = H + dB$. The twisted Chern characters of $\mathcal{G},\mathcal{G}'$ are related through the following commutative diagram:
\begin{equation*}\xymatrix{
K^*(X,\mathcal{G}) \ar[r]^u \ar[d]^{Ch_\mathcal{G}} & K^*(X,\mathcal{G'} \ar[d]^{Ch_\mathcal{G'}}) \\
H^*(X,(\xi,H)) \ar[r]^{e^{-B}} & H^*(X,(\xi,H'))
}
\end{equation*}

Let $f : X \to Y$ a smooth map, $\alpha$ an ungraded gerbe on $Y$ and $L(f)$ the lifting gerbe for $f$. Give $\alpha$ a connection and curving with curvature $H \in \Omega^3(Y)$. The ${\rm U}(1)$-bundle defining the lifting gerbe (with respect to some choice of metrics on $X,Y$) is associated to a $\mathbb{Z}_2$-bundle and so $L(f)$ can be equipped with a canonical flat gerbe connection and zero curving. Of fundamental importance is the following version of the Riemann-Roch formula for twisted $K$-theory \cite{cw}, \cite{cmw}:

\begin{prop}[Riemann-Roch formula for bundle gerbes]
Suppose $X,Y$ are oriented, $f : X \to Y$ a map and $\alpha$ is ungraded. If $a \in K^*(X,L(f)\otimes f^*\alpha)$ then
\begin{equation*}
Ch_{\alpha}( f_* a ) \hat{A}(Y) = f_* ( Ch_{L(f)\otimes f^*\alpha} (a) \hat{A}(X) ).
\end{equation*}
\end{prop}

Note that in order to make sense of this equation one must first define a push-forward operation for compactly supported cohomology twisted by $3$-forms and flat line bundles. We omit the details since the construction parallels the construction of the push-forward in twisted $K$-theory as in \cite{cw0}.\\

It takes only a little work to remove the hypotheses that $X,Y$ are oriented and $\alpha$ ungraded:
\begin{prop}[Riemann-Roch formula for graded bundle gerbes] For any $X,Y$, map $f: X \to Y$ and graded gerbe $\alpha$ we have for $a \in K^*(X,L(f)\otimes f^*\alpha)$:
\begin{equation*}
Ch_{\alpha}( f_* a ) \hat{A}(Y) = f_* ( Ch_{L(f)\otimes f^*\alpha} (a) \hat{A}(X) ).
\end{equation*}
\begin{proof}
We can prove this result in increasing stages of generality. For example assume $\alpha$ is ungraded, $Y$ is oriented but $X$ is not oriented. Let $m : X_2 \to X$ be the orientation double cover of $X$ and set $f_2 = f \circ m$. The Riemann-Roch formula for $f$ will follow from Riemann-Roch for $f_2$. Indeed take $a \in K^*(X,L(f)\otimes f^*\alpha) \otimes \mathbb{R}$. Then $a = m_*(a')$ where $a' = m^*(a)/2 \in K^*(X_2 , L(f_2) \otimes f_2^* \alpha) \otimes \mathbb{R}$. We first note that we have an equality
\begin{equation*}
m_* ( Ch_{L(f_2)\otimes f_2^* \alpha} (a') ) = Ch_{L(f)\otimes f^*\alpha} (m_* a').
\end{equation*}
Indeed the twisted Chern character commutes with pull-backs so
\begin{equation*}
m^* ( Ch_{L(f)\otimes f^*\alpha} (a) ) = Ch_{L(f_2) \otimes f_2^* \alpha } (m^* a) = 2 \, Ch_{L(f_2) \otimes f_2^* \alpha } (a')
\end{equation*}
then apply $m_*$ to both sides and use $m_* m^* = 2$.\\

Now we find
\begin{eqnarray*}
Ch_{\alpha}( f_* a ) \hat{A}(Y) &=& Ch_{\alpha}( (f_2)_* a' ) \hat{A}(Y) \\
&=& (f_2)_*( Ch_{L(f_2)\otimes f_2^*\alpha} (a') \hat{A}(X_2) ) \\
&=& f_* m_* ( Ch_{L(f_2)\otimes f_2^*\alpha} (a') \hat{A}(X_2) ) \\
&=& f_* ( Ch_{L(f)\otimes f^*\alpha} (a) \hat{A}(X) )
\end{eqnarray*}
where we have used Riemann-Roch for $f_2$.\\

A similar argument can be use to remove the other hypotheses and the result follows.
\end{proof}
\end{prop}


\subsection{The twisted Chern character in T-duality}\label{tcctd}

Now that we have reviewed the twisted Chern character we show how the T-duality transformations of twisted $K$-theory and twisted cohomology can be defined so as to commute with the twisted Chern characters.\\

Let $(E,h),(\hat{E},\hat{h})$ be $T$-duals over a compact base $M$. As in Section \ref{tdco} we may find twisted connections $A,\hat{A}$ with curvatures $F,\hat{F}$ and a $3$-form $H_3 \in \Omega^3(M)$ such that if we define $H,\hat{H}$ by
\begin{eqnarray*}
H &=& \pi^* (H_3 )+ A \wedge \pi^*(\hat{F}) \\
\hat{H} &=& \hat{\pi}^* (H_3 )+ \hat{A} \wedge \hat{\pi}^*(F)
\end{eqnarray*}
then $H,\hat{H}$ represent $h,\hat{h}$ in real cohomology. Next we choose (ungraded) bundle gerbes $\alpha_h,\alpha_{\hat{h}}$ representing the classes $h,\hat{h}$. We equip $\alpha_h , \alpha_{\hat{h}}$ with connections $\theta,\hat{\theta}$ and curvings $\omega,\hat{\omega}$ such that $H,\hat{H}$ are the respective curvatures. It is straightforward to see that given $h,\hat{h},H,\hat{H}$ we can find such triples $(\alpha_h,\theta,\omega),(\alpha_{\hat{h}},\hat{\theta},\hat{\omega})$. We write $\mathcal{G} = (\alpha_h,\theta,\omega)$, $\hat{\mathcal{G}} = (\alpha_{\hat{h}},\hat{\theta},\hat{\omega})$ for short.\\

On the correspondence space $F = E \times_M \hat{E}$ we have $p^*h = \hat{p}^* \hat{h}$. It follows that there is an isomorphism of gerbes with connection $U : p^* \mathcal{G} \to \hat{p}^* \hat{\mathcal{G}}$. Underlying $U$ is an isomorphism $\eta : \alpha_h \to \alpha_{\hat{h}}$ of gerbes without connections.

\begin{prop}\label{poinp} The isomorphism $U$ can be chosen so that $\eta$ satisfies the Poincar\'e property described in Section \ref{tdtk}.
\begin{proof}
First we make an observation about the automorphisms of a gerbe with connection and curving. Consider an automorphism determined by a line bundle with connection $(T,\nabla)$ and a shift in curving by a $2$-form $B$. We see that by (\ref{curvingshift}) we must have $B + 2 \pi i F = 0$, where $F$ is the curvature of $\nabla$. Thus if $c(T)$ is the Chern class of $T$ then we have that $B$ represents the image of $c(L)$ in real cohomology. Let us now apply this to the case at hand. Associated to the morphism $U$ is a $B$-shift by some $2$-form $C \in \Omega^2(F)$. We then have
\begin{equation*}
\hat{p}^* \hat{H} - p^*H = dC.
\end{equation*}

For any $b \in B$ let $T_b = \pi^{-1}(b)$, $\hat{T}_b = \hat{\pi}^{-1}(b)$ be the fibres of $E,\hat{E}$. The restrictions $h|_{T_b}$, $\hat{h}_{\hat{T}_b}$ vanish for dimensional reasons so it follows that there are trivializations $\tau : 0 \to \mathcal{G}|_{T_b}$, $\hat{\tau} : 0 \to \hat{\mathcal{G}}|_{\hat{T}_b}$ where $0$ denotes the trivial gerbe with trivial connection and curving. The isomorphisms $U|_{T_b \times \hat{T}_b}$ and $\hat{p}^* \hat{\tau} \circ p^* \tau $ map from $p^* (\mathcal{G}|_{T_b})$ to $\hat{p}^* (\hat{\mathcal{G}}|_{\hat{T}_b})$ so they differ by an automorphism. Up to $2$-morphism any automorphism is determined by a line bundle $T$ with connection $\nabla$. Suppose that associated to $\tau,\hat{\tau}$ are $B$-shifts by $2$-forms $\mu \in \Omega^2(T_b)$, $\hat{\mu} \in \Omega(\hat{T}_b)$. In fact since $T_b,\hat{T}_b$ are $1$-dimensional we have $\mu,\hat{\mu} = 0$. Let $F$ be the curvature of $\nabla$. In then follows that
\begin{equation*}
C|_{T_b \times \hat{T}_b} = F,
\end{equation*}
so $C$ restricted to the fibres is integral and represents the class $\mu \in H^2(T_b \times \hat{T}_b , \mathbb{Z})$ as described in Section \ref{tdtk}.\\

Now recall that the $2$-form $\mathcal{B} = -p^*A \wedge \hat{p}^* \hat{A}$ has the property $\hat{p}^* \hat{H} - p^*H = -d\mathcal{B}$. Therefore the $2$-form $\phi = C+\mathcal{B}$ is closed. Moreover restricted to the fibres $T_b \times \hat{T}_b$ we have that both $C$ and $\mathcal{B}$ represent closed integral $2$-forms. This means that even though the class $[C+\mathcal{B}] \in H^2(F,\mathbb{R})$ might not be integral its restriction to the fibres is. Making use of the filtration on cohomology associated to the Leray-Serre spectral sequence it follows that we may write $[C+\mathcal{B}] = x + p^*(y) + \hat{p}^*(z)$ where $x \in H^2(F,\mathbb{R})$ is integral, $y \in H^2(E,\mathbb{R})$, $z \in H^2(\hat{E},\mathbb{R})$. Now we can find a line bundle $T$ on $F$ with Chern class that equals $-x$ in real cohomology. Let us compose the morphism $U$ with the gerbe automorphism associated to this line bundle. Then redefine $U$ as this composition and redefine $C$ as the $2$-form shift associated to this newly defined $U$. It now follows that $C + \mathcal{B}$ restricted to the fibres is cohomologically trivial. Thus the class of $C$ restricted to the fibre $T_b \times \hat{T}_b$ agrees with the class of $-\mathcal{B}|_{T_b \times \hat{T}_b}$. But since $-\mathcal{B} = p^*A \wedge \hat{p}^* \hat{A}$ this class in $H^2(T_b \times \hat{T}_b , \mathbb{R})$ is exactly image of the class of the Poincar\'e line bundle. It then follows that $U$ has the Poincar\'e property.
\end{proof}
\end{prop}

If we now assume that $U$ satisfies the Poincar\'e property then as we have seen the $2$-form $d = C + \mathcal{B}$ has the property that restricted to the fibres it is cohomologically trivial. By considering the Leray-Serre spectral sequence for $q : F \to M$ we see that the class of $d$ can be written in the form $[d] = p^*(y) - \hat{p}^*(z)$ for some classes $y \in H^2(E,\mathbb{R})$, $z \in H^2(\hat{E},\mathbb{R})$. Let $\beta,\hat{\beta}$ be $2$-forms on $E,\hat{E}$ representing $y,z$. Now let us redefine the gerbes with connections and curvings $\mathcal{G},\hat{\mathcal{G}}$ by performing a shift in curving by $\beta,\hat{\beta}$ respectively. Then we get a morphism from $p^*\mathcal{G}$ to $\hat{p}^*\hat{\mathcal{G}}$ by composing $U$ with the apropriate $2$-form shifts. Let us redefine $U$ to be this composition and redefine $C$ to be the correspoding $2$-form shift. After doing this we find that $C$ and $-\mathcal{B}$ differ by an exact term. In summary given $h,\hat{h},H,\hat{H}$ as above we may choose representative gerbes with connections and curvings $\mathcal{G},\hat{\mathcal{G}}$ and an isomorphism $U : p^*\mathcal{G} \to \hat{p}^* \hat{\mathcal{G}}$ such that the $B$-shift associated to $U$ has the form $C = -\mathcal{B} + da$. From $U$ we get an induced isomorphism $u : K^*(F ,  p^* \mathcal{G}) \to K^*(F , \hat{p}^* \hat{\mathcal{G}} )$ of twisted $K$-theory that fits into a commutative diagram
\begin{equation*}\xymatrix{
K^*(F ,  p^* \mathcal{G}) \ar[r]^u \ar[d]^{Ch} & K^*(F , \hat{p}^* \hat{\mathcal{G}} ) \ar[d]^{Ch} \\
H^*(F,p^*H) \ar[r]^{e^{\mathcal{B}}} & H^*(F,\hat{p}^* \hat{H} ).
}
\end{equation*}

This is already close to proving the result that T-duality respects the twisted Chern character. We know that pull-back respects the twisted Chern character so we get a commutative diagram
\begin{equation*}\xymatrix{
K^*(E,\mathcal{G}) \ar[r]^-{p^*} \ar[d]^{Ch} & K^*(F ,  p^* \mathcal{G}) \ar[d]^{Ch} \\
H^*(E,H) \ar[r]^-{p^*} & H^*(F,p^*H)
}
\end{equation*}
The remaining ingredient is the behaviour of the twisted Chern character under push-forward. This relation is given by the Riemann-Roch formula for twisted $K$-theory. In the simple case at hand $q^*V = {\rm Ker}( \hat{p}_* : TF \to T\hat{E})$ is a real line bundle so $\hat{A}(q^*V) = 1$. It follows that Riemann-Roch in this case reduces to a commutative diagram
\begin{equation*}\xymatrix{
K^*(F , q^*\gamma \otimes \hat{p}^* \hat{\mathcal{G}} ) \ar[d]^{Ch} \ar[r]^-{\hat{p}_*} & K^*(\hat{E}, \hat{\mathcal{G}} ) \ar[d]^{Ch} \\
H^*(F,\hat{p}^* (\xi,\hat{H}) ) \ar[r]^-{\hat{p}_*} & H^*(\hat{E} , \hat{H})
}
\end{equation*}
where $\gamma$ is the lifting gerbe for $V$ equipped with flat connection and zero curving. Using the canonical trivialization $\gamma \otimes \gamma = 1$ we get a similar commutative diagram
\begin{equation*}\xymatrix{
K^*(F , \hat{p}^* \hat{\mathcal{G}} ) \ar[d]^{Ch} \ar[r]^-{\hat{p}_*} & K^*(\hat{E}, \hat{\pi}^*\gamma \otimes \hat{\mathcal{G}} ) \ar[d]^{Ch} \\
H^*(F,\hat{p}^* \hat{H} ) \ar[r]^-{\hat{p}_*} & H^*(\hat{E} , (\xi,\hat{H}))
}
\end{equation*}

Putting it all together obtain:
\begin{prop} We have a commutative diagram
\begin{equation*}\xymatrix{
K^*(E,\mathcal{G}) \ar[r]^-{T^K} \ar[d]^{Ch} & K^{*-1}(\hat{E}, \hat{\pi}^*\gamma \otimes \hat{\mathcal{G}}) \ar[d]^{Ch} \\
H^*(E,H) \ar[r]^-{T^H} & H^{*-1}(E,(\xi,\hat{H}))
}
\end{equation*}
where $T^K = \hat{p}_* \circ \, u \circ \, p^*$ and $T^H = \hat{p}_* \circ \, e^{\mathcal{B}} \circ \, p^*$. Similarly we have a commutative diagram
\begin{equation*}\xymatrix{
K^*(E,\pi^*\gamma \otimes \mathcal{G}) \ar[r]^-{T^K_\xi} \ar[d]^{Ch} & K^{*-1}(\hat{E}, \hat{\mathcal{G}}) \ar[d]^{Ch} \\
H^*(E,(\xi,H)) \ar[r]^-{T^H_\xi} & H^{*-1}(E,\hat{H})
}
\end{equation*}
with $T^K_\xi,T^H_\xi$ defined similarly.
\end{prop}


\section{Courant algebroids}\label{catd}

The structure of Courant algebroids is intimately linked to T-duality, see for example \cite[Chapter 8]{gual},\cite[Chapter 7]{cav} and \cite{hu}. We will demonstrate that such a relation extends to T-dual circle bundles by associating to a pair $(E,H)$ a Courant algebroid $\mathcal{E}(E,H)$ over $M$ and show that T-dual pairs have isomorphic Courant algebroids.\\

Courant algebroids were introduced in \cite{lwx} as a generalization of the bracket used by Courant \cite{cour} to define Dirac structures. We recall the definition of Courant algebroids from the point of view of the Dorfman bracket \cite{roy}. A {\em Courant algebroid} on a smooth manifold $X$ can be defined as a vector bundle $C \to X$ with non-degenerate bilinear form $( \, , \, )$ bilinear bracket $[ \, , \, ] : \Gamma(C) \otimes \Gamma(C) \to \Gamma(C)$ called the {\em Dorfman bracket} and bundle map $\rho : C \to TX$ called the {\em anchor} such that
\begin{itemize}
\item{$[a,[b,c]] = [[a,b],c] + [b,[a,c]]$}
\item{$\rho[a,b] = [\rho(a),\rho(b)]$}
\item{$[a,fb] = \rho(a)(f)b + f[a,b]$}
\item{$[a,b] + [b,a] = d (a,b)$}
\item{$\rho(a)(b,c) = ([a,b],c) + (b,[a,c])$}
\end{itemize}
where $a,b,c \in \Gamma(C)$, $f$ is a function on $X$ and $d$ is the operator $d : \mathcal{C}^\infty(X) \to \Gamma(C)$ defined by $(df , a ) = \rho(a)(f)$. Courant algebroids are more often defined in terms of a skew-symmetric bracket called a Courant bracket as in \cite{lwx}. The Courant bracket is the anti-symmetrization of the Dorfman bracket.\\

A Courant algebroid $C \to X$ is {\em exact} if the sequence
\begin{equation*}\xymatrix{
0 \ar[r] & T^*X \ar[r]^{\rho^*} & C \ar[r]^\rho & TX \ar[r] & 0
}
\end{equation*}
is exact. Here $\rho^*$ is the transpose $T^*X \to C^*$ of $\rho$ followed by the identification of $C$ and $C^*$ using the pairing. Exact Courant algebroids over a manifold $X$ are classified by third cohomology with real coefficients, $H^3(X,\mathbb{R})$. In fact given a closed $3$-form $H$ representing a class in $H^3(X,\mathbb{R})$ we make $C = TX \oplus T^*X$ into a Courant algebroid with $H$-twisted Dorfman bracket \cite{sw} given by
\begin{equation}\label{courb}
[(A , \alpha) , (B , \beta) ]_H = ([A,B] , \mathcal{L}_A \beta - i_B d \alpha + i_B i_A H),
\end{equation}
pairing $( \, , \, )$ given by the obvious paring of $TX$ with $T^*X$ and anchor the projection $\rho : C \to TX$. Then $(C,[ \, , \, ]_H , ( \, , \, ) , \rho)$ is an exact Courant algebroid and one can show that every exact Courant algebroid on $X$ is isomorphic to one of this form. Given two closed $3$-forms $H,H'$ the associated exact Courant algebroids are isomorphic if and only if $H$ and $H'$ represent the same class in $H^3(X,\mathbb{R})$.\\

Given a pair $(E,h)$ consisting of a circle bundle $\pi : E \to M$ over $M$ and a class $h \in H^3(E,\mathbb{Z})$ there is an exact Courant algebroid $C_h \to E$ over $E$ corresponding to the class $h_\mathbb{R} \in H^3(E,\mathbb{R})$. Given a representative $H \in \Omega^3(E)$ of the class $h$ we can realize $C_h$ as the bundle $TE \oplus T^*E$ equipped with the Dorfman bracket (\ref{courb}) twisted by $H$. If we choose an invariant representative for $H$ then the space of invariant sections of $TE \oplus T^*E$ is closed under the twisted Dorfman bracket and this defines a Courant algebroid $\mathcal{E}(E,H)$ over $M$. The anchor is the composition $C_h \to TE \to TM$.\\

Let $(E,h),(\hat{E},\hat{h})$ be T-dual circle bundles. As in Section \ref{tdco} we may find twisted connections $A,\hat{A}$ with curvatures $F,\hat{F}$ and a $3$-form $H_3 \in \Omega^3(M)$ such that if we define $H,\hat{H}$ by
\begin{eqnarray*}
H &=& \pi^* (H_3 )+ A \wedge \pi^*(\hat{F}) \\
\hat{H} &=& \hat{\pi}^* (H_3 )+ \hat{A} \wedge \hat{\pi}^*(F)
\end{eqnarray*}
then $H,\hat{H}$ represent $h,\hat{h}$ in real cohomology. The twisted connections $A$,$\hat{A}$ yield associated splittings
\begin{eqnarray*}
TE &=& TM \oplus \mathbb{R}_\xi \\
T^*E &=& T^*M \oplus \mathbb{R}_\xi \\
T\hat{E} &=& TM \oplus \mathbb{R}_\xi \\
T^*\hat{E} &=& T^*M \oplus \mathbb{R}_\xi
\end{eqnarray*}
where $\xi \in H^1(M,\mathbb{Z}_2)$ is the first Stiefel-Whitney class of $E$ and $\hat{E}$ and $\mathbb{R}_\xi$ the associated flat line bundle. Notice that the bundles $TE \oplus T^*E$ and $T\hat{E} \oplus T^* \hat{E}$ can both be identified with $TM \oplus \mathbb{R}_\xi \oplus \mathbb{R}_\xi \oplus T^*M$.\\

The idea now is that $T$-duality is an exchange of the two $\mathbb{R}_\xi$-factors. Define a map $\phi : TE \oplus T^*E \to T\hat{E} \oplus T^*\hat{E}$ to be the interchange
\begin{equation*}
\phi ( X,x,a,\lambda) = (X , a , x , \lambda)
\end{equation*}
where we have identified $TE \oplus T^*E$ and $T\hat{E} \oplus T^*\hat{E}$ with $TM \oplus \mathbb{R}_\xi \oplus \mathbb{R}_\xi \oplus T^*M$.

\begin{prop}\label{cai} The map $\phi : TE \oplus T^*E \to T\hat{E} \oplus T^*\hat{E}$ determines an isomorphism of Courant algebroids $\phi : \mathcal{E}(E,H) \to \mathcal{E}(\hat{E},\hat{H})$.
\begin{proof}
Our proof is a straightforward adaptation of \cite[Theorem 7.2]{cav}. Recall that for an invariant form $\omega \in \Omega^*_{\rm inv}(E,\mathbb{R})$ we have a decomposition
\begin{equation*}
\omega = \pi^*(\alpha) + A \wedge \pi^*(\beta)
\end{equation*}
and that the T-duality map $T = \hat{p}_* \circ e^{\mathcal{B}} \circ p^* $ sends $\omega$ to
\begin{equation*}
T \omega = \hat{\pi}^* (\beta) - \hat{A} \wedge \hat{\pi}^*(\alpha).
\end{equation*}

There is a natural action of sections of $TE \oplus T^*E$ on $\Omega^*(E)$. Let $a$ be a section of $TE \oplus T^*E$ and $\omega \in \Omega^*(E)$. Then
\begin{equation*}
a \omega = i_Y \omega + a \wedge \omega.
\end{equation*}

If $a$ and $\omega$ are invariant one sees easily that
\begin{equation}\label{clifft}
T( a \omega ) = -\phi(a) T \omega
\end{equation}
We also have that $T$ intertwines $d_H$ and $d_{\hat{H}}$ up to sign:
\begin{equation}\label{tint}
T (d_H \omega ) = - d_{\hat{H}} (T \omega).
\end{equation}

The Dorfman bracket is a derived bracket in the sense that
\begin{equation*}
[ a , b ]_H \omega = [ [d_H , a ] , b ] \omega
\end{equation*}
\cite{kos} where $a,b$ are sections of $TE \oplus T^*E$ and $\omega \in \Omega^*(E)$. Writing this out in full we have
\begin{equation*}
[a,b]_H \omega = d_H( ab \omega) - a d_H(b \omega) - b d_H (a \omega) + ba d_H \omega.
\end{equation*}
Applying $T$ and using properties (\ref{clifft}),(\ref{tint}) we find
\begin{eqnarray*}
\phi([a,b]_H) T \omega &=& d_{\hat{H}} ( \phi(a)\phi(b) T\omega) - \phi(a) d_{\hat{H}}(\phi(b) T\omega) \\
&& - \phi(b) d_{\hat{H}} (\phi(a) T\omega) + \phi(b)\phi(a) d_{\hat{H}} T\omega,
\end{eqnarray*}
which reduces to simply
\begin{equation*}
\phi( [a,b]_H ) = [ \phi(a) , \phi(b) ]_{\hat{H}}.
\end{equation*}
From this it follows easily that $\phi$ is an isomorphism of Courant algebroids.
\end{proof}
\end{prop}


\section{Examples}\label{examps}

We will consider some low dimensional examples of circle bundle $T$-duality. Since the case of oriented circle bundles has already been considered in the literature our examples will focus on non-oriented circle bundles.\\

In order to calculate twisted $K$-theory we make extensive use of the Atiyah-Hirzebruch spectral sequence. For a twisting class $\tau = (\xi , \eta) \in H^1(M,\mathbb{Z}_2) \times H^3(M,\mathbb{Z})$ The spectral sequence has $E_2^{p,q} = H^p(M , K^q({\rm pt}))$ and abuts to the twisted $K$-theory $K^*(M,\tau)$ \cite{atseg1}. Note that the coefficients in the $E_2$ term are twisted by $\xi$. 

For examples of dimension no greater than three the only possibly non-trivial differential is $d_3 : H^0(M , \mathbb{Z}_\xi) \to H^3(M , \mathbb{Z}_\xi)$. If the class $\xi$ is non-trivial then $H^0(M,\mathbb{Z}_\xi) = 0$, so $E_2 = E_\infty$. If $\xi = 0$ then it is known that $d_3(x) = Sq^3_{\mathbb{Z}}(x) - \eta  x$, where $Sq^3_\mathbb{Z}$ is an integral Steenrod operation \cite{atseg2}. Restricted to $d_3 : H^0(M,\mathbb{Z}) \to H^3(M,\mathbb{Z})$ however the differential reduces to just $d_3(x) = -\eta x$. In either case, we see that it is quite simple to apply the spectral sequence in low dimensions.

In some examples we will only be able to determine the twisted $K$-theory groups up to an extension problem. In such cases however we get an exact result on the T-dual, so T-duality allows us to solve the extension problem.


\subsection{Klein bottle}
Let $M = S^1$. There are two $S^1$-bundles over $M$, the trivial oriented bundle $T^2 \to S^1$ and the unique non-oriented bundle $K \to S^1$ which is the Klein bottle. Since these bundles are $2$-dimensional they can not have a non-trivial $H$-flux. We conclude that $T^2$ and $K$ are self $T$-dual.

Let $\xi$ denote the non-trivial class in $H^1(S^1,\mathbb{Z}_2) = \mathbb{Z}_2$. The cohomologies of $K$ with coefficients in $\mathbb{Z}$ or $\mathbb{Z}_\xi$ are easily computed to be
\begin{equation*}
\renewcommand{\arraystretch}{1.4}
\begin{tabular}{|l|l|l|}
\hline
$i$ & $H^i(K,\mathbb{Z})$ & $H^i(K,\mathbb{Z}_\xi)$ \\
\hline
$0$ & $\mathbb{Z}$ & $0$ \\
$1$ & $\mathbb{Z}$ & $\mathbb{Z} \oplus \mathbb{Z}_2$ \\
$2$ & $\mathbb{Z}_2$ & $\mathbb{Z}$ \\
\hline
\end{tabular}
\end{equation*}

From this table we can easily apply the Atiyah-Hirzebruch spectral sequence. There is no extension problem and we find
\begin{equation*}
\renewcommand{\arraystretch}{1.4}
\begin{tabular}{|l|l|l|}
\hline
$i$ & $K^i(K)$ & $K^{i}(K,\xi)$ \\
\hline
$0$ & $\mathbb{Z} \oplus \mathbb{Z}_2$ & $\mathbb{Z}$ \\
$1$ & $\mathbb{Z}$ & $\mathbb{Z} \oplus \mathbb{Z}_2$ \\
\hline
\end{tabular}
\end{equation*}
in accordance with T-duality.


\subsection{Non-oriented circle bundles over compact Riemann surfaces}

Let $\Sigma_g$ be a compact connected Riemann surface of genus $g$. We will consider circle bundles over $\Sigma_g$ with fibre orientation class $0 \neq \xi \in H^1(\Sigma_g , \mathbb{Z}_2)$. Note that $H^1(\Sigma_g , \mathbb{Z}) = \mathbb{Z}^{2g}_2$, so we assume $g > 0$.\\

The cohomologies of $\Sigma_g$ with coefficients $\mathbb{Z},\mathbb{Z}_\xi$ are
\begin{equation*}
\renewcommand{\arraystretch}{1.4}
\begin{tabular}{|l|l|l|}
\hline
$i$ & $H^i(\Sigma_g,\mathbb{Z})$ & $H^i(\Sigma_g,\mathbb{Z}_\xi)$ \\
\hline
$0$ & $\mathbb{Z}$ & $0$ \\
$1$ & $\mathbb{Z}^{2g}$ & $\mathbb{Z}^{2g-2} \oplus \mathbb{Z}_2$ \\
$2$ & $\mathbb{Z}$ & $\mathbb{Z}_2$ \\
\hline
\end{tabular}
\end{equation*}
in particular, $H^2(\Sigma_g , \mathbb{Z}_\xi) = \mathbb{Z}_2$, so there are exactly two circle bundles over $\Sigma_g$ with fibre orientation class equal to $\xi$. For $j \in \mathbb{Z}_2$, let $E_g^j$ denote the corresponding circle bundle.

We will determine the cohomology of $E_g^j$ with coefficients in $\mathbb{Z}$ and $\mathbb{Z}_\xi$. To do this we first calculate the cohomology with local coefficients $\mathbb{Z}_\xi$ by the Leray-Serre spectral sequence then we use Poincar\'{e} duality to determine the cohomology with $\mathbb{Z}$ coefficients. Doing this avoids an extension problem. We find
\begin{equation*}
\renewcommand{\arraystretch}{1.4}
\begin{tabular}{|l|l|l|l|l|}
\hline
$i$ & $H^i(E^0_g,\mathbb{Z})$ & $H^i(E^0_g,\mathbb{Z}_\xi)$ & $H^i(E^1_g,\mathbb{Z})$ & $H^i(E^1_g,\mathbb{Z}_\xi)$ \\
\hline
$0$ & $\mathbb{Z}$ & $0$ & $\mathbb{Z}$ & $0$\\
$1$ & $\mathbb{Z}^{2g}$ & $\mathbb{Z}^{2g-1} \oplus \mathbb{Z}_2$ & $\mathbb{Z}^{2g}$ & $\mathbb{Z}^{2g-1} \oplus \mathbb{Z}_2$ \\
$2$ & $\mathbb{Z}^{2g-1}\oplus \mathbb{Z}_2 $ & $\mathbb{Z}^{2g} \oplus \mathbb{Z}_2$ & $\mathbb{Z}^{2g-1}$ & $\mathbb{Z}^{2g}$ \\
$3$ & $\mathbb{Z}_2$ & $\mathbb{Z}$ & $\mathbb{Z}_2$ & $\mathbb{Z}$ \\
\hline
\end{tabular}
\end{equation*}
For either bundle we find $H^3(E_g^j , \mathbb{Z}) = \mathbb{Z}_2$. Let $k \in \mathbb{Z}_2$ and let $\eta_k \in H^3(E_g^j , \mathbb{Z})$ be the corresponding $H$-flux. T-duality of pairs $(E_g^j,\eta_k)$ in this case is easily seen to be the interchange $j \leftrightarrow k$.\\

We now turn to the calculation of the twisted $K$-theory. Using the Atiyah-Hirzebruch spectral sequence we obtain
\begin{equation*}
\renewcommand{\arraystretch}{1.4}
\begin{tabular}{|l|l|l|l|l|}
\hline
$i$ & $K^{i}(E^0_g,\eta_0)$ & $K^{i}(E^0_g,(\xi,\eta_0))$ & $K^{i}(E^1_g,\eta_0)$ & $K^{i}(E^1_g,(\xi,\eta_0))$ \\
\hline
$0$ & $\mathbb{Z}^{2g} \oplus \mathbb{Z}_2$ &  $\mathbb{Z}^{2g} \oplus \mathbb{Z}_2$ & $\mathbb{Z}^{2g}$ &  $\mathbb{Z}^{2g}$\\
$1$ & $\mathbb{Z}^{2g} \oplus \mathbb{Z}_2$ & $*$ & $\mathbb{Z}^{2g} \oplus \mathbb{Z}_2$ & $*$ \\
\hline
\hline
$i$ & $K^{i}(E^0_g,\eta_1)$ & $K^{i}(E^0_g,(\xi,\eta_1))$ & $K^{i}(E^1_g,\eta_1)$ & $K^{i}(E^1_g,(\xi,\eta_1))$ \\
\hline
$0$ & $\mathbb{Z}^{2g} \oplus \mathbb{Z}_2$ & $\mathbb{Z}^{2g} \oplus \mathbb{Z}_2$ & $\mathbb{Z}^{2g}$ &  $\mathbb{Z}^{2g}$ \\
$1$ & $\mathbb{Z}^{2g}$ & $*$ & $\mathbb{Z}^{2g}$ & $*$ \\
\hline
\end{tabular}
\end{equation*}
where the four $*$ terms indicates that the spectral sequence does not completely determine these groups. Fortunately we can use T-duality to fill in the remaining blanks, thus the complete list is
\begin{equation*}
\renewcommand{\arraystretch}{1.4}
\begin{tabular}{|l|l|l|l|l|}
\hline
$i$ & $K^{i}(E^0_g,\eta_0)$ & $K^{i}(E^0_g,(\xi,\eta_0))$ & $K^{i}(E^1_g,\eta_0)$ & $K^{i}(E^1_g,(\xi,\eta_0))$ \\
\hline
$0$ & $\mathbb{Z}^{2g} \oplus \mathbb{Z}_2$ &  $\mathbb{Z}^{2g} \oplus \mathbb{Z}_2$ & $\mathbb{Z}^{2g}$ &  $\mathbb{Z}^{2g}$\\
$1$ & $\mathbb{Z}^{2g} \oplus \mathbb{Z}_2$ & $\mathbb{Z}^{2g} \oplus \mathbb{Z}_2$ & $\mathbb{Z}^{2g} \oplus \mathbb{Z}_2$ & $\mathbb{Z}^{2g} \oplus \mathbb{Z}_2$ \\
\hline
\hline
$i$ & $K^{i}(E^0_g,\eta_1)$ & $K^{i}(E^0_g,(\xi,\eta_1))$ & $K^{i}(E^1_g,\eta_1)$ & $K^{i}(E^1_g,(\xi,\eta_1))$ \\
\hline
$0$ & $\mathbb{Z}^{2g} \oplus \mathbb{Z}_2$ & $\mathbb{Z}^{2g} \oplus \mathbb{Z}_2$ & $\mathbb{Z}^{2g}$ &  $\mathbb{Z}^{2g}$ \\
$1$ & $\mathbb{Z}^{2g}$ & $\mathbb{Z}^{2g}$ & $\mathbb{Z}^{2g}$ & $\mathbb{Z}^{2g}$ \\
\hline
\end{tabular}
\end{equation*}


\subsection{Non-oriented circle bundles over non-orientable compact surfaces}

Consider now connected, compact non-orientable surfaces. Every such surface can be expressed as an $n$-fold connected sum $M_n = \#^n \mathbb{RP}^n$. We have $H^1(M_n,\mathbb{Z}_2) = \mathbb{Z}^n_2$. We consider non-orientable circle bundles over $M_n$ with fibre orientation class $0 \neq \xi \in H^1(M_n,\mathbb{Z}_2)$. Let $w \in H^1(M_2,\mathbb{Z}_2)$ denote the first Stiefel-Whitney class of $M_n$. For simplicity we will only consider the case $\xi = w$, although the case $\xi \neq w$ is no harder to compute.\\

The cohomologies of $M_n$ with $\mathbb{Z}$ and $\mathbb{Z}_\xi$ coefficients are
\begin{equation*}
\renewcommand{\arraystretch}{1.4}
\begin{tabular}{|l|l|l|}
\hline
$i$ & $H^i(M_n,\mathbb{Z})$ & $H^i(M_n,\mathbb{Z}_\xi)$ \\
\hline
$0$ & $\mathbb{Z}$ & $0$ \\
$1$ & $\mathbb{Z}^{n-1}$ & $\mathbb{Z}^{n-1} \oplus \mathbb{Z}_2$ \\
$2$ & $\mathbb{Z}_2$ & $\mathbb{Z}$ \\
\hline
\end{tabular}
\end{equation*}
Let $j \in \mathbb{Z} = H^2(M_n,\mathbb{Z}_e)$ and $E_n^j \to M_n$ the corresponding circle bundle. Next we will compute the cohomology of $E_n^j$. For this the Leray-Serre spectral sequence by itself is inadequate and so we consider the fundamental group of $E_n^j$. One finds
\begin{equation*}
\pi_1(E_n^j) = \langle a_1 , a_2 , \dots , a_n , x \; | \; a_1^2 a_2^2 \dots a_n^2 = x^j, \, x^2 = 1 \rangle.
\end{equation*}
We conclude that $H_1(E_j,\mathbb{Z}) = \mathbb{Z}^{n-1} \oplus \mathbb{Z}_2 \oplus \mathbb{Z}_2$ if $j$ is even, $\mathbb{Z}^{n-1} \oplus \mathbb{Z}_4$ if $j$ is odd. From here the Leray-Serre spectral sequence suffices and we obtain
\begin{equation*}
\renewcommand{\arraystretch}{1.4}
\begin{tabular}{|l|l|l|l|l|}
\hline
$i$ & $H^i(E^j_n,\mathbb{Z})$, $j$ {\rm even} & $H^i(E^j_n,\mathbb{Z})$, $j$ {\rm odd} & $H^i(E^j_n,\mathbb{Z}_\xi)$, $j=0$ & $H^i(E^j_n,\mathbb{Z}_\xi)$, $j\neq 0$ \\
\hline
$0$ & $\mathbb{Z}$ & $\mathbb{Z}$ & $0$ & $0$ \\
$1$ & $\mathbb{Z}^{n-1}$ & $\mathbb{Z}^{n-1}$ & $\mathbb{Z}^{n} \oplus \mathbb{Z}_2$ & $\mathbb{Z}^{n-1} \oplus \mathbb{Z}_2$ \\
$2$ & $\mathbb{Z}^{n-1}\oplus \mathbb{Z}_2 \oplus \mathbb{Z}_2$ & $\mathbb{Z}^{n-1} \oplus \mathbb{Z}_4$ & $\mathbb{Z}^{n}$ & $\mathbb{Z}^{n-1} \oplus \mathbb{Z}_j $\\
$3$ & $\mathbb{Z}$ & $\mathbb{Z}$ & $\mathbb{Z}_2$ & $\mathbb{Z}_2$ \\
\hline
\end{tabular}
\end{equation*}
Let $k \in \mathbb{Z}$ and let $\eta_k \in H^3(E_n^j,\mathbb{Z})$ be the corresponding class. As in the last example $T$-duality amongst pairs $(E_n^j,\eta_k)$ is the interchange $j \leftrightarrow k$.

As usual we attempt to calculate the twisted $K$-theory with the Atiyah-Hirzebruch spectral sequence. We again find that some of the groups are determined up to an extension problem, but we can use T-duality to determine the answer. We find
\begin{equation*}
\renewcommand{\arraystretch}{1.4}
\begin{tabular}{|l|l|l|l|l|}
\hline
$i$ & $K^{i}(E^j_n,\eta_k)$ & $K^{i}(E^j_n,\eta_k)$ & $K^{i}(E^j_n,\eta_k)$ & $K^{i}(E^j_n,\eta_k)$ \\
    & $j$ {\rm even}, $k = 0$ & $j$ {\rm odd}, $k = 0$ & $j$ {\rm even}, $k \neq 0$ & $j$ {\rm odd}, $k \neq 0$ \\
\hline
$0$ & $\mathbb{Z}^{n} \oplus \mathbb{Z}_2 \oplus \mathbb{Z}_2$ &  $\mathbb{Z}^{n} \oplus \mathbb{Z}_4$ & $\mathbb{Z}^{n-1} \oplus \mathbb{Z}_2 \oplus \mathbb{Z}_2$ & $\mathbb{Z}^{n-1} \oplus \mathbb{Z}_4$\\
$1$ & $\mathbb{Z}^{n} $ & $\mathbb{Z}^{n} $ &  $\mathbb{Z}^{n-1} \oplus \mathbb{Z}_k$ & $\mathbb{Z}^{n-1} \oplus \mathbb{Z}_k$ \\
\hline
\hline
$i$ & $K^{i}(E^j_n,(\xi,\eta_k))$ & $K^{i}(E^j_n,(\xi,\eta_k))$ & $K^{i}(E^j_n,(\xi,\eta_k))$ & $K^{i}(E^j_n,(\xi,\eta_k))$ \\
    & $j = 0$, $k$ {\rm even} & $j \neq 0$, $k$ {\rm even} & $j=0$, $k$ {\rm odd} & $j \neq 0$, $k$ {\rm odd} \\
\hline
$0$ & $\mathbb{Z}^{n}$ & $\mathbb{Z}^{n-1} \oplus \mathbb{Z}_j$ & $\mathbb{Z}^{n}$ & $\mathbb{Z}^{n-1} \oplus \mathbb{Z}_j$ \\
$1$ & $\mathbb{Z}^{n} \oplus \mathbb{Z}_2 \oplus \mathbb{Z}_2$ & $\mathbb{Z}^{n-1} \oplus \mathbb{Z}_2 \oplus \mathbb{Z}_2$ & $\mathbb{Z}^n \oplus \mathbb{Z}_4$ & $\mathbb{Z}^{n-1} \oplus \mathbb{Z}_4$ \\
\hline
\end{tabular}
\end{equation*}


\end{document}